\documentclass[reqno,11pt]{amsart}
\usepackage[colorlinks=true, linkcolor=blue, citecolor=blue]{hyperref}

\usepackage{amssymb}
\usepackage{amsmath, graphicx, rotating}
\usepackage{color}     
\usepackage{soul}
\usepackage[dvipsnames]{xcolor}   
\usepackage{tcolorbox}

\usepackage{ifthen}
\usepackage{xkeyval}
\usepackage{todonotes}
\usepackage[T1]{fontenc}
\usepackage{lmodern}
\usepackage[english]{babel}

\usepackage{ upgreek }
\usepackage{stmaryrd}
\SetSymbolFont{stmry}{bold}{U}{stmry}{m}{n}
\usepackage{amsthm}
\usepackage{float}

\usepackage{ bbm }
\usepackage{ stmaryrd }
\usepackage{ mathrsfs }
\usepackage{ frcursive }
\usepackage{ comment }

\usepackage{pgf, tikz}
\usetikzlibrary{shapes}
\usepackage{varioref}
\usepackage{enumitem}
\usepackage{longtable}

\usepackage{mathtools}

\usepackage{dsfont}

\definecolor{rouge}{rgb}{0.7,0.00,0.00}
\definecolor{vert}{rgb}{0.00,0.5,0.00}
\definecolor{bleu}{rgb}{0.00,0.00,0.8}

\usepackage[margin=1.3in]{geometry}

\newtheorem{theorem}{Theorem}[section]
\newtheorem*{theorem*}{Theorem}
\newtheorem{lemma}[theorem]{Lemma}

\newtheorem{corollary}[theorem]{Corollary}
\newtheorem{proposition}[theorem]{Proposition}

\labelformat{hypothesis}{\textbf{M\kern-0.1mm#1}}

\newtheorem{condition}{Condition}

\labelformat{conditionA}{\textbf{A\kern-0.1mm#1}}

\renewcommand\dots{\hbox to 1em{.\hss.\hss.}}

\theoremstyle{definition}

\numberwithin{equation}{section}

\def\bb#1{\mathbb{#1}}

\def\bf#1{\mathbf{#1}}
\def\scr#1{\mathscr{#1}}

\def\geq{\geqslant}
\def\leq{\leqslant}

\def\ddd{\ldots}

\newcommand\ee{\varepsilon}

\DeclareMathOperator{\Leb}{Leb}

\def\geq{\geqslant}
\def\leq{\leqslant}
\def\Rd {\mathbb{R}^d}
\def\Rd*{(\mathbb{R}^d)^*}
\def\Pd{{\mathbb{P}}^{d-1}}
\def\Pd*{(\mathbb{P}^{d-1})^*}

\def\bb#1{\mathbb{#1}}

\begin{document}
\title[Conditioned random walks on linear groups]{Conditioned random walks on linear groups II:\\ local limit theorems} 

\author{Ion Grama}
\author{Jean-Fran\c cois Quint}
\author{Hui Xiao}

\curraddr[Grama, I.]{ Univ Bretagne Sud, CNRS UMR 6205, LMBA, Vannes, France.}
\email{ion.grama@univ-ubs.fr}

\curraddr[Quint, J.-F.]{IMAG, Univ Montpellier, CNRS, Montpellier, France.}
\email{Jean-Francois.Quint@umontpellier.fr}

\curraddr[Xiao, H.]{Academy of Mathematics and Systems Science, Chinese Academy of Sciences, Beijing 100190, China.}
\email{xiaohui@amss.ac.cn}

\date{\today }
\subjclass[2020]{Primary 60B15, 60B20, 60F05, 60J05. Secondary 22D40, 22E46}
\keywords{Random walks on groups, exit time, 
random walks conditioned to stay positive, local limit theorem, target harmonic measure}

\begin{abstract}
We investigate random walks on the general linear group constrained within a specific domain, with a focus on their asymptotic behavior. 
In a previous work \cite{GQX24a}, we constructed the associated harmonic measure, a key element in formulating the local limit theorem for conditioned random walks on groups. The primary aim of this paper is to prove this theorem. The main challenge arises from studying the conditioned reverse walk, whose increments, in the context of random walks on groups, depend on the entire future. 
To achieve our goal, we combine a Caravenna-type conditioned local limit theorem with the conditioned version of the central limit theorem for the reversed walk. The resulting local limit theorem is then applied to derive the local behavior of the exit time.
\end{abstract}

\maketitle

\tableofcontents


\section{Introduction and results}

\subsection{Notation and background} \label{sec-background and motiv}
Let $\bb V$ be a finite dimensional real vector space. Denote by $\bb G = \textrm{GL}(\bb V)$ the group of linear automorphisms of $\bb V$.
We equip $\bb V$ with a Euclidean norm $\| \cdot \|$. 
If $g$ is a linear endomorphism of $V$, we write $\|g\| = \sup_{v \in \bb V: v \neq 0} \frac{\| gv \|}{\| v \|}$ for the operator norm of $g$. 
The group $\bb G$ acts on the projective space $\bb P(\bb V)$ of $\bb V$ through the formula 
$ g ( \bb R v ) = \bb R (gv)$, where $g \in \bb G$ and $v \in \bb V \setminus \{0\}$.

Let $\mu$ be a Borel probability measure on the group $\bb G$. 
Denote by $\Gamma_{\mu}$ the closed subsemigroup of $\bb G$ spanned by the support of the measure $\mu$. 
In the sequel we will make use of the condition that $\Gamma_{\mu}$ is proximal, meaning that   
$\Gamma_{\mu}$ contains an element $g$ such that 
the characteristic polynomial of $g$ admits a unique root of maximal modulus and that this root is simple. 
We will also assume that $\Gamma_{\mu}$ is strongly irreducible, meaning that no finite union of proper non-zero subspaces 
of $\bb V$ is $\Gamma_{\mu}$ invariant. 
We will say that $\mu$ admits a finite exponential moment if there exists a constant $\alpha >0$ such that 
\begin{align}\label{Exponential-moment}
\int_{\bb G} \max \{ \|g\|,  \|g^{-1} \| \}^{\alpha} \mu(dg) < \infty. 
\end{align}
For $g \in \bb G$ and $x = \bb R v \in \bb P(\bb V)$, we define 
\begin{align*}
\sigma(g, x) = \log \frac{\|gv\|}{\|v\|}. 
\end{align*}
The function $\sigma: \bb G \times \bb P(\bb V) \to \bb R$ satisfies the cocycle identity: for any $g_1, g_2 \in \bb G$ and $x \in \bb P(\bb V)$, we have 
$\sigma(g_2 g_1, x) = \sigma(g_2, g_1 x) + \sigma(g_1, x).$ 

Throughout the paper, we fix the probability space $(\Omega, \scr A, \bb P)$,
where $\Omega = \bb G^{\bb N^*}$, $\bb N^*$ is the set  of positive integers, 
$\bb G$ is equipped with the Borel $\sigma$-algebra, $\mathscr A$ is the corresponding product $\sigma$-algebra, 
and $\bb P = \mu^{\otimes \bb N^*}$ is the product measure.
The sequence of coordinate functions $g_1, g_2, \ldots$ on $\Omega$ forms a sequence of independent identically distributed elements of $\bb G$ with law $\mu$.
For any starting point $x\in \bb P(\bb V)$, consider the random walk
\begin{align} \label{def of direct RW-001}
 \sigma(g_n \cdots g_1,x) = \sum_{k=1}^{n}  \sigma(g_k, g_{k-1} \cdots g_1 x) , 
 \quad n\geq 1,
\end{align}
where, by a convention applied throughout the paper,  the empty left product $g_m \cdots g_1$ for $m<1$ 
is identified with the identity matrix.  
The study of the random walk \eqref{def of direct RW-001} 
has attracted much attention, see, for example, the papers \cite{FK60, LeP82, GR85, GM89, BQ16, GL16, Ser19, AS21, CDJP23, AS24a, AS24b}, 
as well as the books \cite{Boug-Lacr85, BQ16b},
where various asymptotic properties, such as the law of large numbers, the central limit theorem, and large deviations, 
have been established.

In particular, from the results in \cite{Boug-Lacr85, BQ16b}, we know that if $\Gamma_{\mu}$ is proximal and strongly irreducible, and if the measure $\mu$ 
admits an exponential moment, 
there exists a real number $\lambda_{\mu}$, called the first Lyapunov exponent of $\mu$,
such that,  for any $x \in \bb P(\bb V)$, $\bb P$-almost surely, as $n \to \infty$, 
\begin{align}\label{def-Lyapunov-001}
\frac{1}{n}  \sigma(g_n \cdots g_1, x) \to \lambda_{\mu}. 
\end{align}
We assume that the random walk \eqref{def of direct RW-001} is centered, which means that $\lambda_{\mu} = 0$. 
Introduce the following stopping times: 
for any $x \in \bb P(\bb V)$ and $t\in \bb R$, 
\begin{align} \label{stopping time tau x t-001} 
\tau_{x, t}  & = \min \{ k \geq 1: t + \sigma(g_k \cdots g_1, x) < 0 \}  
\end{align}
and
\begin{align} 
\check \tau_{x, t}  & = \min \{ k \geq 1: t - \sigma(g_k \cdots g_1, x) < 0 \},   \label{stopping time tau x t-002}
\end{align}
where by convention $\min \emptyset =\infty$.

The study of these stopping times was initiated in \cite{GLP17}, 
where the asymptotic of the persistence probability $\bb P(\tau_{x,t} >n)$ 
and the central limit theorem for the random walk $t + \sigma(g_n \cdots g_1, x)$  
conditioned on the event $\{ \tau_{x,t} >n \}$ were determined. 
In this paper, we focus on the associated local limit theorem for the walks $(t \pm \sigma(g_n \cdots g_1, x))_{n\geq 0}$ 
conditioned to remain non-negative.  
While such result has been established before for some classical random walks,   
its proof for walks on groups has proven to be highly technical and challenging;
it requires the development of new techniques, primarily due to the difficulties related to  
the reversibility of the walk \eqref{def of direct RW-001}. 
We provide detailed explanations for this in Subsection \ref{sec: proof strategy-001}.    
 The main goal of this paper, along with the companion paper \cite{GQX24a}, is to introduce the appropriate techniques to
 address these difficulties and to prove such a conditioned local limit theorem.
 Potential applications across several related fields include, for instance, 
random walks on affine groups in the critical case \cite{BPP20, ABP24}, 
reflected random walks \cite{Lalley95, EsPeRa2013, EP15},
multitype branching processes in random environment \cite{LPPP18, LPPP21, PP23b},
branching random walks on the linear group \cite{BDGM14, Mentem16, GXM2024ExtrPos}. 
Additionally, the methods developed in this work may prove beneficial in studying random walks on hyperbolic spaces 
\cite{Gou09, Gou14, Gou15, BQ16aa, AS22, BMSS22}. 


Let us end this section by stating some notation to be used all over the peper. We write 
$\bb R_+=[0,\infty)$ for the  set of non-negative numbers and $\bb R_-=(-\infty,0]$ for the  set of non-positive numbers. 
The set of non-negative integers is denoted by $\bb N$. 
The letters $c, C$ will represent positive constants, whose values may differ with each occurrence. 
The expectation corresponding to the probability measure $\bb P$ is denoted by $\bb E$.
We denote by $\mathds 1_{B}$ the indicator function of the event $B\in \scr A$. 
 For brevity, given a random variable $X$  and an event $B\in \scr A$, we will write $\bb E (X; B)$ 
 for the expectation $\bb E (X \mathds 1_{B})$.

\subsection{Statement of main results} \label{subsec-statement results}
To describe the asymptotic behavior of the conditioned random walk 
$(t + \sigma(g_n \cdots g_1, x))_{n\geq 0}$, 
we need to introduce two key concepts:
the harmonic function $V$, which is associated with the random walk $(t + \sigma(g_n \cdots g_1, x))_{n\geq 0}$ 
killed upon exiting $\bb R_+$, and a dual object to $V$, 
which we refer to as the harmonic measure, denoted by $\rho$.

We begin by stating the following existence result \cite[Theorem 2.1]{GLP17}. 
Under the assumptions that $\Gamma_{\mu}$ is proximal and strongly irreducible, 
the measure $\mu$ admits an exponential moment and the Lyapunov exponent $\lambda_{\mu}$ is zero, 
we have that, for any $x \in \bb P(\bb V)$ and $t \in \bb R$, the following limits exist:
\begin{align} 
& \lim_{n \to \infty} \bb E \Big( t + \sigma(g_n \cdots g_1, x); \tau_{x, t} > n \Big) = V(x, t),  \label{harm function GLPP-001} \\
 & \lim_{n \to \infty} \bb E \Big( t - \sigma(g_n \cdots g_1, x); \check{\tau}_{x, t} > n \Big) = \check{V}(x, t).  \label{harm function GLPP-002}
\end{align}
Moreover, the functions $V$ and $\check V$ are non-negative, non-decreasing in $t$ and satisfy, uniformly in $x \in \bb P(\bb V)$, 
\begin{align*} 
\lim_{t \to \infty} \frac{V(x, t)}{t}  =\lim_{t \to \infty}\frac{\check{V}(x, t)}{t} = 1.
\end{align*} 
From \eqref{harm function GLPP-001} it follows that the function  $V$ is $Q$-harmonic, 
meaning that it satisfies $Q V=V$, where the operator 
$$Qh(x,t)= \bb E \Big( h(g_1 x, t + \sigma(g_1, x)); \tau_{x, t} > 1 \Big)$$
is defined for any 
 continuous compactly supported function $h$ on $\bb P(\bb V)\times \bb R$. 
 A similar property holds for the function $\check V$.

In \cite{GQX24a}, we proved the existence of dual objects to the functions $V$ and $\check V$, 
which are Radon measures on $\bb P(\bb V) \times \bb R$. 
More precisely, under the same assumptions as before, we showed that 
there exist Radon measures $\rho$ and $\check{\rho}$ 
on $\bb P(\bb V) \times \bb R$ such that,  
for any continuous compactly supported function $h$ on $\bb P(\bb V) \times \bb R$,  
the following limits exist uniformly in $x \in \bb P(\bb V)$
and are independent of $x$: 
\begin{align}
 \lim_{n \to \infty}  \int_{0}^{\infty} t  \bb E \Big( h (g_n \cdots g_1 x,  &\  t +  \sigma(g_n \cdots g_1, x)); \tau_{x, t} > n -1 \Big) dt  \notag \\
   & \qquad\qquad\quad = \int_{\bb P(\bb V) \times \bb R} h(x', t') \rho(dx', dt'),  \label{exist of measure rho-001} \\ 
   \lim_{n \to \infty}  \int_{0}^{\infty} t  \bb E \Big( h (g_n \cdots g_1 x,  &\  t - \sigma(g_n \cdots g_1, x) ); \check{\tau}_{x, t} > n -1 \Big) dt  \notag \\  
&\qquad\qquad\quad  = \int_{\bb P(\bb V) \times \bb R} h(x', t') \check{\rho}(dx', dt') \label{exist of measure rho-002}. 
\end{align}
Note that, by Corollary 1.6  of \cite{GQX24a}, the Radon measures $\rho$ and $\check \rho$ are non-zero. 
Furthermore, by Corollary 1.3  of \cite{GQX24a}, the marginals of the measures $\rho$ and $\check \rho$ on $\bb R$ are absolutely continuous with respect to the Lebesgue measure,
with non-decreasing densities denoted by $W$ and $\check{W}$, respectively. 
According to Corollary 1.4 of \cite{GQX24a},
the measure $\rho$ is $R$-harmonic, 
meanning  it satisfies $R^* \rho=\rho$, where the operator 
\begin{align} \label{operator R 001}
R h(x,t) = \mathds 1_{\{t\geq 0\}} \bb E \Big( h (g_1 x, t + \sigma(g_1, x) ) \Big)
\end{align}
is defined for any  continuous compactly supported function $h$ on $\bb P(\bb V)\times \bb R$. 
 A similar property holds for the Radon measure $\check \rho$. 

We will use the harmonic functions $V$ and $\check V$, 
along with the Radon measures $\rho$ and $\check\rho$, to establish the following conditioned local limit theorem.
To state this we also need 
the asymptotic variance   
\begin{align*}
\upsilon_{\mu}^2 = \lim_{n \to \infty} \frac{1}{n} \bb E \Big[ (\sigma(g_n \cdots g_1, x))^2  \Big],
\end{align*}
which exists for any $x\in \bb P(\bb V)$, is independent of $x$, and is strictly positive, 
as shown in \cite{BQ16b}.

\begin{theorem}\label{Thm-CLLT-cocycle}
Assume that $\Gamma_{\mu}$ is proximal and strongly irreducible, 
the measure $\mu$ admits an exponential moment and the Lyapunov exponent $\lambda_{\mu}$ is zero. 
Then, for any fixed $t\in \bb R$ and for any continuous compactly supported function $h$ on $\bb P(\bb V) \times \bb R$, 
 we have, uniformly in $x \in \bb P(\bb V)$,
\begin{align*}
  \lim_{n \to \infty}  n^{3/2}  \bb E \Big( h (g_n \cdots g_1 x, & \  t + \sigma(g_n \cdots g_1, x)); \tau_{x, t} > n -1 \Big)  \\
& \qquad = \frac{2 V(x, t)}{ \sqrt{2 \pi} \upsilon_{\mu}^3 } \int_{\bb P(\bb V) \times \bb R} h (x',t') \rho(dx',dt'), 
 \notag\\
 \lim_{n \to \infty}  n^{3/2}  \bb E \Big( h (g_n \cdots g_1 x, & \   t - \sigma(g_n \cdots g_1, x) );  \check{\tau}_{x, t} > n -1 \Big)  \\
& \qquad = \frac{2 \check{V}(x, t)}{ \sqrt{2 \pi} \upsilon_{\mu}^3 } \int_{\bb P(\bb V) \times \bb R} h (x',t') \check{\rho}(dx',dt'). 
\end{align*}
\end{theorem}

The target function $h$ in the statement of Theorem \ref{Thm-CLLT-cocycle} is assumed to be continuous. 
However, by standard approximation techniques, 
one can prove that Theorem \ref{Thm-CLLT-cocycle} also holds for compactly supported functions $h$ 
that are continuous on a set of full measure with respect to $\rho$. 
For instance, applying the previous theorem with $h(x,t) \mathds 1_{\{t\geq 0\}}$ and 
$h(x,t) \mathds 1_{\{t < 0\}}$ in place of $h(x,t)$, we obtain the following result. 

\begin{corollary}\label{coroll-CLLT-cocycle-tau>n}
Assume that $\Gamma_{\mu}$ is proximal and strongly irreducible, 
the measure $\mu$ admits an exponential moment and the Lyapunov exponent $\lambda_{\mu}$ is zero. 
Then, for any fixed $t\in \bb R$ and for any continuous compactly supported function $h$ on $\bb P(\bb V) \times \bb R$, 
 we have, uniformly in $x \in \bb P(\bb V)$,
\begin{align*}
  \lim_{n \to \infty}  n^{3/2}  \bb E \Big( h (g_n \cdots g_1 x, & \  t + \sigma(g_n \cdots g_1, x)); \tau_{x, t} > n \Big)  \\
&\qquad = \frac{2 V(x, t)}{ \sqrt{2 \pi} \upsilon_{\mu}^3 } \int_{\bb P(\bb V) \times \bb R_+} h (x',t') \rho(dx',dt'), 
 \notag\\
 \lim_{n \to \infty}  n^{3/2}  \bb E \Big( h (g_n \cdots g_1 x,  & \  t - \sigma(g_n \cdots g_1, x) );   \check{\tau}_{x, t} > n  \Big)  \\
&\qquad = \frac{2 \check{V}(x, t)}{ \sqrt{2 \pi} \upsilon_{\mu}^3 } \int_{\bb P(\bb V) \times \bb R_+} h (x',t') \check{\rho}(dx',dt')  
\end{align*}
and
\begin{align*}
  \lim_{n \to \infty}  n^{3/2}  \bb E \Big( h (g_n \cdots g_1 x,  & \  t + \sigma(g_n \cdots g_1, x)); \tau_{x, t} = n \Big)  \\
&\qquad = \frac{2 V(x, t)}{ \sqrt{2 \pi} \upsilon_{\mu}^3 } \int_{\bb P(\bb V) \times \bb R_-} h (x',t') \rho(dx',dt'), 
 \notag\\
 \lim_{n \to \infty}  n^{3/2}  \bb E \Big( h (g_n \cdots g_1 x, & \  t - \sigma(g_n \cdots g_1, x) );   \check{\tau}_{x, t} = n  \Big)  \\
&\qquad = \frac{2 \check{V}(x, t)}{ \sqrt{2 \pi} \upsilon_{\mu}^3 } \int_{\bb P(\bb V) \times \bb R_-} h (x',t') \check{\rho}(dx',dt'). 
\end{align*}
\end{corollary}

Similarly, when the target function is of the form $h(x,t) = \varphi(x)  \mathds 1_{[a,b]}(t)$, 
from Theorem \ref{Thm-CLLT-cocycle} one can deduce the following consequence. 

\begin{corollary}\label{Coro-CLLT-cocycle-001}
Assume that $\Gamma_{\mu}$ is proximal and strongly irreducible, 
the measure $\mu$ admits an exponential moment and the Lyapunov exponent $\lambda_{\mu}$ is zero. 
Then, for any fixed $t\in \bb R$, any continuous function $\varphi$ on $\bb P(\bb V)$ and any $- \infty < a < b < \infty$, 
 we have, uniformly in $x \in \bb P(\bb V)$,
\begin{align*}
  \lim_{n \to \infty}  n^{3/2}  \bb E \Big( \varphi (g_n \cdots g_1 x); & \   t + \sigma(g_n \cdots g_1, x)  \in [a, b], \tau_{x, t} > n -1 \Big)  \\
&\qquad\quad = \frac{2 V(x, t)}{ \sqrt{2 \pi} \upsilon_{\mu}^3 } \int_{\bb P(\bb V) \times [a, b]} \varphi (x') \rho(dx',dt'), 
 \notag\\
 \lim_{n \to \infty}  n^{3/2}  \bb E \Big( \varphi (g_n \cdots g_1 x);  & \  t - \sigma(g_n \cdots g_1, x)  \in [a, b],  \check{\tau}_{x, t} > n -1 \Big)  \\
&\qquad\quad = \frac{2 \check{V}(x, t)}{ \sqrt{2 \pi} \upsilon_{\mu}^3 } \int_{\bb P(\bb V) \times [a, b]}  \varphi (x') \check{\rho}(dx',dt'). 
\end{align*}
In particular, with $\varphi=1$, for any fixed $t\in \bb R$ and any $- \infty < a < b < \infty$, we have, uniformly in $x \in \bb P(\bb V)$,
\begin{align*}
&  \lim_{n \to \infty}  n^{3/2}  \bb P \Big(   t + \sigma(g_n \cdots g_1, x) \in [a, b],  \tau_{x, t} > n -1 \Big)   \notag\\
& \qquad\qquad\qquad\qquad = \frac{2 V(x, t)}{ \sqrt{2 \pi} \upsilon_{\mu}^3 }  \rho(\bb P(\bb V) \times [a, b]) 
= \frac{2 V(x, t)}{ \sqrt{2 \pi} \upsilon_{\mu}^3 } \int_a^b W(t') dt',  \notag\\
& \lim_{n \to \infty}  n^{3/2}  \bb P \Big(  t - \sigma(g_n \cdots g_1, x)  \in [a, b],   \check{\tau}_{x, t} > n -1 \Big)  \notag\\
& \qquad\qquad\qquad\qquad = \frac{2 \check{V}(x, t)}{ \sqrt{2 \pi} \upsilon_{\mu}^3 } \check{\rho}(\bb P(\bb V) \times [a, b]) 
 = \frac{2 \check{V}(x, t)}{ \sqrt{2 \pi} \upsilon_{\mu}^3 } \int_a^b \check{W}(t') dt'. 
\end{align*}
\end{corollary}

In the above statement, we have denoted by $W$ and $\check{W}$ the non-decreasing functions on $\bb R$, 
which are the densities of the marginals of the measures $\rho$ and $\check{\rho}$ on $\bb R$. 
Recall that from Corollary 1.6 of \cite{GQX24a}, we have
\begin{align*} 
& 0 < \rho \big( \bb P(\bb V) \times \bb R_- \big) = \int_{-\infty}^0 W(t') dt'  < \infty,  \notag\\
& 0 < \check \rho \big( \bb P(\bb V) \times \bb R_- \big) = \int_{-\infty}^0 \check{W}(t') dt'  < \infty. 
\end{align*}
Using our main results, we derive the following local limit asymptotics for the exit times 
$\tau_{x,t}$ and $\check\tau_{x,t}$. 
This result essentially follows from applying Theorem \ref{Thm-CLLT-cocycle} to the \emph{non-compactly supported} function
$h(x,t)=\varphi(x)\mathds 1_{\{t<0\}}$.

\begin{corollary}\label{Thm-CLLT-cocycle-002}
Assume that $\Gamma_{\mu}$ is proximal and strongly irreducible, 
the measure $\mu$ admits an exponential moment and the Lyapunov exponent $\lambda_{\mu}$ is zero. 
Then, for any fixed $t\in \bb R$ and any continuous function $\varphi$ on $\bb P(\bb V)$, 
 we have, uniformly in $x \in \bb P(\bb V)$,
\begin{align*}
&  \lim_{n \to \infty}  n^{3/2}  \bb E \Big( \varphi(g_n \cdots g_1 x);  \tau_{x, t} = n \Big)  
 = \frac{2 V(x, t)}{ \sqrt{2 \pi} \upsilon_{\mu}^3 }  \int_{\bb P(\bb V) \times \bb R_-} \varphi (x') \rho(dx',dt'), 
 \notag\\
& \lim_{n \to \infty}  n^{3/2}  \bb E \Big( \varphi(g_n \cdots g_1 x);  \check{\tau}_{x, t} = n \Big)   
  = \frac{2 \check{V}(x, t)}{ \sqrt{2 \pi} \upsilon_{\mu}^3 }  \int_{\bb P(\bb V) \times \bb R_-} \varphi (x') \check{\rho}(dx',dt'). 
\end{align*}
In particular, 
\begin{align*}
&  \lim_{n \to \infty}  n^{3/2}  \bb P (  \tau_{x, t} = n )  
 = \frac{2 V(x, t)}{ \sqrt{2 \pi} \upsilon_{\mu}^3 }  \rho \left( \bb P(\bb V) \times \bb R_- \right) 
 = \frac{2 V(x, t)}{ \sqrt{2 \pi} \upsilon_{\mu}^3 } \int_{-\infty}^0 W(t') dt', 
 \notag\\
& \lim_{n \to \infty}  n^{3/2}   \bb P (  \check{\tau}_{x, t} = n )   
  = \frac{2 \check{V}(x, t)}{ \sqrt{2 \pi} \upsilon_{\mu}^3 } \check{\rho} \left( \bb P(\bb V) \times \bb R_- \right)
  = \frac{2 \check{V}(x, t)}{ \sqrt{2 \pi} \upsilon_{\mu}^3 } \int_{-\infty}^0 \check{W}(t') dt'. 
\end{align*}
\end{corollary}

In the course of the proof of Theorem \ref{Thm-CLLT-cocycle}, we will establish the following estimate, 
which is also of independent interest.

\begin{theorem}\label{Cor-CLLT-cocycle-001}
Assume that $\Gamma_{\mu}$ is proximal and strongly irreducible, 
the measure $\mu$ admits an exponential moment and the Lyapunov exponent $\lambda_{\mu}$ is zero. 
Then, for any continuous compactly supported function $h$ on $\bb P(\bb V) \times \bb R$, 
 we have, uniformly in $x \in \bb P(\bb V)$,
\begin{align}
 \lim_{n \to \infty}  \sqrt{n} 
  \int_{\bb R}     \bb E \Big(  h (g_n \cdots g_1 x, & \  t + \sigma(g_n \cdots g_1, x) );   \tau_{x, t} > n - 1  \Big) dt  \notag\\
& = \frac{2}{ \sqrt{2 \pi} \upsilon_{\mu} }  \int_{\bb P(\bb V) \times \bb R} h (x',t') \rho(dx',dt'),  \label{Asym-rho-002}\\
  \lim_{n \to \infty}  \sqrt{n} 
  \int_{\bb R}    \bb E \Big(  h (g_n \cdots g_1 x, & \   t + \sigma(g_n \cdots g_1, x));  \tau_{x, t} > n  \Big) dt  \notag\\
&  =  \frac{2}{ \sqrt{2 \pi} \upsilon_{\mu} }  \int_{\bb P(\bb V) \times \bb R_+} h (x',t') \rho(dx',dt'). 
  \label{Asym-rho-001}
\end{align}
Moreover, similar assertions hold for $\check{\tau}_{x, t}$. 
\end{theorem}

We will also establish the following uniform bound, which is used to prove Theorem \ref{Thm-CLLT-cocycle}.

\begin{theorem}\label{Thm-CLLT-cocycle-bound-001}
Assume that $\Gamma_{\mu}$ is proximal and strongly irreducible, 
the measure $\mu$ admits an exponential moment and the Lyapunov exponent $\lambda_{\mu}$ is zero. 
Then, there exists a constant $c>0$ such that for any $t \in \bb R$, $a<b$, $x \in \bb P(\bb V)$ and $n \geq 1$, 
we have 
\begin{align*}
& \bb P \Big(  t + \sigma(g_n \cdots g_1, x) \in [a, b],   \tau_{x, t} > n -1 \Big) \notag\\
& \qquad\qquad \leq  c ( 1+ \max\{t, 0\} ) ( 1+ b-a ) (1+\max\{b, 0\})  n^{-3/2}, \notag\\
& \bb P \Big(  t - \sigma(g_n \cdots g_1, x) \in [a, b],   \check{\tau}_{x, t} > n -1 \Big) \notag\\
& \qquad\qquad \leq  c ( 1+ \max\{t, 0\} ) ( 1+ b-a ) (1+\max\{b, 0\})  n^{-3/2}. 
\end{align*}
\end{theorem}
We refer to \cite[Corollary 3.7]{LPPP21} and \cite[Theorem 1.2]{PP23} for  similar results 
concerning products of random non-negative matrices. 
Note that \cite[Theorem 1.2]{PP23} also provides lower bounds of the same order for these probabilities.

From Theorem \ref{Thm-CLLT-cocycle-bound-001} we can derive the following uniform upper bound 
for the local probability $\bb P(\tau_{x, t} = n)$ of the exit time.

\begin{corollary}\label{Corol-local prob for exit time}
Assume that $\Gamma_{\mu}$ is proximal and strongly irreducible, 
the measure $\mu$ admits an exponential moment and the Lyapunov exponent $\lambda_{\mu}$ is zero. 
Then, there exists a constant $c>0$ such that for any $t \in \bb R$, $x \in \bb P(\bb V)$ and $n \geq 1$, 
we have 
\begin{align*}
& \bb P \Big( \tau_{x, t} =n \Big)  \leq  c ( 1+ \max\{t, 0\} ) n^{-3/2}, \notag\\
& \bb P \Big( \check{\tau}_{x, t} =n \Big) \leq  c ( 1+ \max\{t, 0\} ) n^{-3/2}. 
\end{align*}
\end{corollary}

All the results stated in this section also apply to
 the stopping times defined using the large inequality $\leq$ instead of the strict inequality $<$. 
The corresponding results can be obtained using the same methods.


\subsection{A Caravenna-type result}\label{Sec-Caravenna-type}

In this section, we formulate a special type of local limit theorem for conditioned random walks on linear groups, 
which we refer to as the Caravenna-type conditioned local limit theorem. 
This theorem provides an asymptotic rate of order $n^{-1}$ 
and is effective when the support of the target function $h$ moves to infinity at a rate of $\sqrt{n}$.
It serves as an intermediate step between the conditioned central limit theorem (Theorem 2.3 of \cite{GLP17}) 
and the conditioned local limit theorem (Theorem \ref{Thm-CLLT-cocycle}). 

We denote by $\phi^+$ the standard Rayleigh density function:  
\begin{align} \label{Rayleigh law-001} 
\phi^+(t) = t e^{-t^2/2} \mathds 1_{\{ t \geq 0 \}}, 
\quad  t \in \bb R.  
\end{align}
Let $\nu$ denote the unique $\mu$-stationary Borel probability measure on $\bb P(\bb V)$, 
as discussed in \cite{Boug-Lacr85} and \cite{BQ16b}. 

\begin{theorem} \label{t-B 001}
Assumptions that $\Gamma_{\mu}$ is proximal and strongly irreducible, 
the measure $\mu$ admits an exponential moment and the Lyapunov exponent $\lambda_{\mu}$ is zero. 
Then, for any continuous compactly supported function $h$ on $\bb P(\bb V) \times \bb R$ and any $t\in \bb R$, 
uniformly for $x \in \bb P(\bb V)$ 
and $u \in \bb R$, 
\begin{align*}
\lim_{n \to \infty} \Bigg(  & \frac{ \upsilon_{\mu}^2 \sqrt{2 \pi} n }{ 2 } 
  \bb E \Big[  h \Big( g_n \cdots g_1 x,  t + \sigma(g_n \cdots g_1, x) - u  \Big);  \tau_{x, t} >n  \Big]    \nonumber\\
&\qquad  - V(x, t) \phi^+ \bigg( \frac{u}{ \upsilon_{\mu} \sqrt{n}} \bigg) 
  \int_{\bb P(\bb V)}  \int_{\bb R_+}  h \left(x', t' \right) dt' \nu(dx')  \Bigg) = 0. 
\end{align*}
\end{theorem}

Such type of results trace back to  
Caravenna \cite{Carav05} for random walks on the real line. 
We refer to  \cite{DW15}  for random walks in cones of $\bb R^d$ and to \cite{GLL20} for Markov chains with finite states. 
A similar local limit theorem has recently been obtained in \cite{PP23} for products of positive random matrices. 

The convolution technique used in our proof is inspired by \cite{GX2022IID}; 
however, the estimates of the remainder terms are significantly more laborious for random walks on linear groups. 
Unlike the previous works \cite{Carav05, DW15, GLL20, GQX23, GX2022IID, PP23},
the approach developed in this paper allows to establish this theorem without employing any reversal techniques.

As a corollary of the proof of Theorem \ref{t-B 001}, we derive the following uniform bound which is of independent interest. 
This bound is sharper than the one provided by Theorem \ref{Thm-CLLT-cocycle-bound-001} when $b \gg \sqrt{n}$. 

\begin{corollary}\label{Cor-bound-n-001}
Assume that $\Gamma_{\mu}$ is proximal and strongly irreducible, $\mu$ admits finite exponential moments
and that the Lyapunov exponent $\lambda_{\mu}$ is zero. 
Then, there exists a constant $c>0$ such that for any $t \in \bb R$, $x \in \bb P(\bb V)$, $a < b$ and $n \geq 1$, 
we have 
\begin{align*}
\bb P \Big(  t + \sigma(g_n \cdots g_1, x) \in [a, b],   \tau_{x, t} > n -1 \Big)
 \leq  \frac{c}{n} ( 1+ \max\{t, 0\} ) (b-a + 1). 
\end{align*}
\end{corollary}


\subsection{Extension to the case of flag manifolds}
As in \cite{GQX24a}, 
the methods developed in this paper can be extended to formulate conditioned local limit theorems for random walks on reductive groups.  
Let $\bf G$ be a real connected reductive group. Denote by $K$ a maximal compact subgroup of $G = \bf G(\bb R)$
and by $\bf A$ a maximal $\bb R$-split torus of $\bf G$, such that the Cartan involution of $\bf G$ associated with $K$ equals $-1$ on the Lie algebra $\mathfrak a$ of  $A = \bf A(\bb R)$. 
Let $\mathfrak a^+ \subset \mathfrak a$ be a Weyl chamber. 
Then we have the Cartan decomposition $G = K \exp (\mathfrak a^+) K$
and the associated Cartan projection $\kappa: G \to \mathfrak a^+$.

Let $\bf P$ be the unique minimal  $\bb R$-parabolic subgroup of $\bf G$ whose Lie algebra contains the root spaces associated with the elements of $\mathfrak a^+$. 
We have the Iwasawa decomposition $G = K P$, where $P = \bf P(\bb R)$. 
Let $\bf U$ be the unipotent radical of $\bf P$ so that $P = AU$, where $U = \bf U(\bb R)$,
and hence $G = KAU$. 
More precisely, for $g \in G$, the set $KgU$ contains a unique element of $\exp (\mathfrak a)$. 

We also denote by $\mathcal P = G/P$ the flag manifold of $G$, 
which is the set of minimal  $\bb R$-parabolic subgroups of $\bf G$, 
and by $\xi_0$ the unique fixed point of $P$ in $\mathcal P$. 
The set $\mathcal P$ is a compact homogeneous space of $G$.
For $g \in G$ and $\xi \in \mathcal P$, choose $k \in K$ such that $\xi = k \xi_0$,
and denote by $\sigma(g, \xi)$ the unique element of $\mathfrak a$ such that  
\begin{align*}
\exp(\sigma(g, \xi)) \in K g kU. 
\end{align*}
The map  $\sigma: G \times \mathcal P \to \mathfrak a$ 
is a smooth cocycle known as the Iwasawa cocycle. 

Let $\mu$ be a Borel probability measure on $G$. 
Assume that the first moment of $\mu$ is finite, meaning that $\int_{G} \| \kappa(g) \| \mu(dg) < \infty$
for some norm $\| \cdot \|$ on the vector space $\mathfrak a$. 
Then the limit $\lim_{n \to \infty} \frac{1}{n} \bb E \kappa(g_n \cdots g_1)$ exists and is called the Lyapunov vector $\lambda_{\mu} \in \mathfrak a^+$. 
Let $\phi$ be a linear functional on $\mathfrak a$ such that $\phi(\lambda_{\mu}) = 0$. 
For $t \in \bb R$ and $\xi \in \mathcal P$, we define the following stopping time 
\begin{align*}
\tau_{\xi, t}  = \min \{ k \geq 1: t + \phi( \sigma(g_k \cdots g_1, \xi) ) < 0 \}. 
\end{align*}
Denote by $\Gamma_{\mu}$ the subsemigroup of $G$ spanned by the support of $\mu$. 

In \cite[Theorems 1.8 and 1.9]{GQX24a},  we associated with the data determined by the cocycle $\phi \circ \sigma$ 
and the stopping times $\tau_{\xi, t}$ for $\xi \in \mathcal P$, 
a harmonic function $V$ on $\mathcal P \times \bb R$ 
as well as a harmonic Radon measure $\rho$ on the same space $\mathcal P \times \bb R$. 
These objects allow us to state the following analog of Theorem \ref{Thm-CLLT-cocycle}.

\begin{theorem}\label{Thm-CLLT-cocycle-flag}
Assume that $\Gamma_{\mu}$ is Zariski dense in $G$,  
the measure $\mu$ admits an exponential moment, meaning $\int_{G} e^{\alpha \|\kappa(g) \|} \mu(dg) < \infty$ for some $\alpha >0$, 
and $\phi(\lambda_{\mu}) = 0$. 
Then, for any fixed $t\in \bb R$ and for any continuous compactly supported function $h$ on $\mathcal P \times \bb R$, 
 we have, uniformly in $\xi \in \mathcal P$,
\begin{align*}
  \lim_{n \to \infty}  n^{3/2}  \bb E \Big[ h \Big( g_n \cdots g_1 \xi,  & \  t + \phi ( \sigma(g_n \cdots g_1, \xi)) \Big); \tau_{\xi, t} > n -1 \Big]  \\
& \qquad = \frac{2 V(\xi, t)}{ \sqrt{2 \pi} \upsilon_{\mu}^3 } \int_{\mathcal P \times \bb R} h (\xi',t') \rho(d\xi',dt'). 
\end{align*}
\end{theorem}

Analogues of Corollaries \ref{coroll-CLLT-cocycle-tau>n}, \ref{Coro-CLLT-cocycle-001} and \ref{Thm-CLLT-cocycle-002} 
can also be formulated in this setting. 
While we will not prove these results explicitly, they can be derived from Theorem \ref{Thm-CLLT-cocycle-flag} 
utilizing the framework presented in \cite{BQ16b}.


\subsection{Extension to the case of local fields}

Theorem \ref{Thm-CLLT-cocycle} and its corollaries are stated for random walks taking values in the linear group $\mathrm{GL}(d,\bb R)$ over the field of real numbers $\bb R$. 
However, by employing the framework of \cite{BQ16b}, they can be directly extended to random walks with values in the linear group $\mathrm{GL}(d,\bb K)$, 
where $\bb K$ is a local field, that is a locally compact topological field. 
The same extension applies to Theorem \ref{Thm-CLLT-cocycle-flag}.

\subsection{Proof strategy and organization of the paper} \label{sec: proof strategy-001}

The conditioned local limit theorem has garnered considerable attention in probability theory, with numerous applications in the study of random walks across various related fields. Following the pioneering work of Feller \cite{Feller} and Spitzer \cite{Spitzer}, this topic has been explored by several authors, including 
Borovkov \cite{Borovk62, Bor70},  
Iglehart \cite{Igle74}, 
Bolthausen \cite{Bolth}, 
Eppel \cite{Eppel-1979}, 
Bertoin and Doney \cite{BertDoney94},
Varopoulos \cite{Var1999, Var2000}, 
Caravenna \cite{Carav05}, 
Eichelsbacher and K\"onig \cite{EichKonig}, 
Vatutin and Wachtel \cite{VatWacht09}, 
Doney \cite{Don12},  
Denisov and Wachtel \cite{Den Wacht 2008, DW15},
Kersting and Vatutin \cite{KV17}, and Peign\' e and Pham \cite{PP23}. 
Various applications have been already mentioned in Subsection \ref{sec-background and motiv}. 


Let us briefly outline the proof strategy employed in this paper and stress the main difficulties. 
Our approach is based on the excursion decomposition,  
which is a common tool in the study of conditioned random walks and is often used alongside the classical Wiener-Hopf factorization technique (see \cite{Feller, Spitzer}). 
Our study diverges from this classical framework, instead utilizing excursions through a convolution approach. 

This method proves effective in handling dependent random variables and was applied in \cite{GLL20} to establish conditioned local limit theorems for finite-state Markov chains, inspired by \cite{DW15} for random walks in cones. 
Further advancements were achieved in \cite{GX2022IID} for random walks on $\bb R$, 
where the convolution approach was developed to refine previous asymptotics by specifying rates of convergence. 
Our current investigation into products of random matrices is also motivated by our earlier work on conditioned limit theorems for hyperbolic dynamical systems \cite{GQX23}, alongside \cite{GLL20, GX2022IID}.

In the excursion decomposition technique, the trajectory of a stochastic process is divided into three parts. The first part is analyzed using the conditioned central limit theorem, while the second part is examined via the ordinary local limit theorem. 
The combined analysis of these two parts, through a convolution of the resulting limits, 
leads to a Caravenna-type conditioned local limit theorem, as presented in Theorem \ref{t-B 001}.

The third part of the trajectory is analyzed using the reversal approach, 
which allows it to be treated as the trajectory of a dual process conditioned to remain positive. 
In certain cases, this reversed trajectory can be addressed using a conditioned central limit theorem, 
similar to the one applied in the first part of the excursion decomposition.

Indeed, in the classical case of random walks on $\bb R$, the reversed process is also a random walk on $\bb R$ 
(see \cite{VatWacht09, Don12}).
For finite-state Markov chains, the reversed process corresponds to the trajectory of a finite-state Markov chain, 
as demonstrated in \cite{GLL20}.
In the context of Birkhoff sums of an observable $f$ over a hyperbolic dynamical system $(\bb X,T)$, 
the reversed process is represented by the Birkhoff sums of the observable $f\circ T^{-1}$ over the dynamical system $(\bb X, T^{-1})$; 
see \cite{GQX23} for details.

However, in many interesting cases, including the random walks on linear groups considered in this paper, the reversed trajectory generally corresponds to a different dynamical object. The main challenge is then to develop an appropriate technique to effectively study the reversed process. In the context of random walks on linear groups, we introduced the language of random walks with future-dependent perturbations in our preliminary paper \cite{GQX24a}. This approach proves to be a suitable tool for addressing the issues related to the reversed trajectory.
In \cite{GQX24a}, 
this technique was used to construct the target harmonic measures $\rho$ and $\check\rho$, which appear as limits in \eqref{exist of measure rho-001} and \eqref{exist of measure rho-002} and are also featured in the main results of the present paper. Here, the analysis of the third part of the trajectory builds on the study of an abstract model with perturbations formulated in Appendix \ref{sec invar func}, leading to the conditioned central limit theorem stated in Theorem \ref{Thm-CCLT-limit}.

 Finally, we combine the Caravenna-type conditioned local limit theorem, derived from analyzing the first two parts of the excursion decomposition, with the conditioned central limit theorem for the reversed trajectory of the third part. This combination allows us to derive the statement of Theorem \ref{Thm-CLLT-cocycle}.

The organization of the paper is as follows:

Section \ref{sec EffectiveLLT}: We present an effective ordinary (non-conditioned) local limit theorem for products of random matrices, which will be used to address the middle part of the trajectory in the excursion decomposition.

Section \ref{sec: Effective CLLT}: We apply the local limit theorem discussed above in conjunction with the conditioned central limit theorem of 
\cite{GLP17} to prove the Caravenna-type local limit theorem (cf.\ Theorem \ref{t-B 001}). This section handles the behavior of the random walk over the union of the first and second parts of the trajectory in the excursion decomposition.

Section \ref{sec: reversed random walk}: We derive our reversal formula and use it, along with the results in Appendix \ref{sec invar func}, to prove Theorems \ref{Cor-CLLT-cocycle-001} and \ref{Thm-CLLT-cocycle-bound-001}. These results are instrumental in controlling remainder terms in the proof of our main result.

    Section \ref{Proof of main Theorem}: We establish our main result, Theorem \ref{Thm-CLLT-cocycle}, 
    as a convolution of the asymptotics obtained in Section \ref{sec: Effective CLLT} 
 and of the conditioned central limit theorem for the reversed random walk with perturbations from Theorem A.3 of the Appendix.
 As a consequence we deduce Corollary \ref{Thm-CLLT-cocycle-002}. 
    
    Appendix \ref{sec invar func}: We prove the conditioned central limit theorem (Theorem \ref{Thm-CCLT-limit}) in the abstract framework of random walks with future-dependent perturbations. This result is essential for analyzing the third part of the trajectory in the excursion decomposition.


\section{Effective local limit theorems}\label{sec EffectiveLLT}

In this section, we shall establish an effective version of the ordinary local limit theorem 
for the random walk $\sigma(g_n \cdots g_1, x)$. 
The arguments are adapted from the analogous statements in \cite[Section 5]{GQX23}.


\subsection{Spectral gap theory} \label{Spectral gap theory}

Let $\bb V$ be a finite dimensional real vector space which we equip with a Euclidean norm $\| \cdot \|$.
We equip the exterior square $\wedge^2 \bb V$ with the associated Euclidean norm so that if $v, w \in \bb V$, 
we have $\| v \wedge w \|^2 + \langle v, w \rangle^2 = \|v\|^2 \|w\|^2$, where $\langle \cdot, \cdot \rangle$ denotes the scalar product. 
We define a distance on $\bb P(\bb V)$ by setting $d(x, y) = \frac{\| v \wedge w \|}{\|v\| \|w\|}$
for $x = \bb R v$ and $y = \bb R w$. 

For $\gamma \in (0,1)$ and a real-valued function $\varphi$ on $\bb P(\bb V)$, we define the $\gamma$-H\"older constant $\omega_{\gamma}(\varphi)$ as
\begin{align*}
\omega_{\gamma}(\varphi) = \sup_{x, y \in \bb P(\bb V): x \neq y} \frac{| \varphi(x) - \varphi(y) |}{d(x, y)^{\gamma}}. 
\end{align*}
The function $\varphi$ is said to be $\gamma$-H\"older continuous if $\omega_{\gamma}(\varphi)$ is finite.
The space of all $\gamma$-H\"older continuous functions on $\bb P(\bb V)$ is denoted by $\mathscr B_{\gamma}$. 
It is a Banach space when equipped with the norm $\| \cdot \|_{\scr B_{\gamma}}$ defined by 
\begin{align*}
\| \varphi \|_{\scr B_{\gamma}} = \sup_{x \in \bb P(\bb V)} |\varphi(x)| + \omega_{\gamma}(\varphi). 
\end{align*}
Note that the norm $\| \cdot \|_{\scr B_{\gamma}}$ depends on the original choice of a Euclidean structure on $\bb V$,
but the space $\scr B_{\gamma}$ does not. 
Denote by $\mathscr L(\scr B_{\gamma}, \scr B_{\gamma})$ 
the set of all bounded linear operators from $\scr B_{\gamma}$ to $\scr B_{\gamma}$
equipped with the standard operator norm $\left\| \cdot \right\|_{\scr B_{\gamma} \to \scr B_{\gamma}}$. 

We fix a Borel probability measure $\mu$ on $\bb G = {\rm GL}(\bb V)$, which admits a finite exponential moment, meaning that \eqref{Exponential-moment} holds for some constant $\alpha>0$. 
We denote by $P$ the Markov operator on $\bb P(\bb V)$ associated with $\mu$, i.e., 
for any continuous function $\varphi$ on $\bb P(\bb V)$ and any $x \in \bb P(\bb V)$,
we set 
\begin{align*}
P \varphi(x) = \int_{\bb G} \varphi(gx) \mu(dg). 
\end{align*}
More generally, for $z \in \bb C$ with $|\Re(z)| < \alpha$, we denote by $P_z$ the perturbation of the operator $P$
defined by, for any continuous function $\varphi$ on $\bb P(\bb V)$ and any $x \in \bb P(\bb V)$, 
\begin{align*}
P_z \varphi(x) = \int_{\bb G} e^{z \sigma(g, x)} \varphi(gx) \mu(dg). 
\end{align*}
Note that, for $n \geq 0$, we have
\begin{align*} 
P_z^n \varphi(x) = \int_{\bb G^n} e^{z \sigma(g_n \cdots g_1, x)} \varphi(g_n \cdots g_1 x) \mu(dg_1) \ldots \mu(dg_n). 
\end{align*}
By \cite{Boug-Lacr85}, there exist $\gamma \in (0,1)$ and $\alpha >0$ such that for any $z \in \bb C$ with $|\Re(z)| < \alpha$,
the operator $P_z$ is bounded on $\scr B_{\gamma}$ and the operator valued map $z \mapsto P_z$ defines a holomorphic map
from the strip $\{z \in \bb C: |\Re(z)| < \alpha \}$ towards the space $\mathscr L(\scr B_{\gamma}, \scr B_{\gamma})$. 
Once for all, we fix such constant $\gamma \in (0,1)$. 

Suppose that the subsemigroup $\Gamma_{\mu}$ spanned by the support of $\mu$ is proximal and strongly irreducible,
and denote by $\nu$ the unique $\mu$-stationary Borel probability measure on $\bb P(\bb V)$ (see \cite{Boug-Lacr85}). 
The following result (see \cite{Boug-Lacr85}) provides the spectral gap properties for the perturbed transfer operator $P_z$. 

\begin{lemma} \label{Lem_Perturbation}
Assume that $\Gamma_{\mu}$ is proximal and strongly irreducible, and $\mu$ admits finite exponential moments. 
Then, there exists a constant $\delta > 0$ such that for any $z \in \bb D_{\delta}: = \{z \in \bb C:  |z| < \delta \}$,
\begin{align}
P_z^n  = \lambda^{n}_{z} \Pi_{z} + N^{n}_{z},  
\quad  n \geq 1, \label{perturb001}
\end{align}
where the mappings $z \mapsto \Pi_{z}: \bb D_{\delta} \to \mathscr L(\scr B_{\gamma}, \scr B_{\gamma})$
and $z \mapsto N_{z}: \bb D_{\delta} \to \mathscr L(\scr B_{\gamma}, \scr B_{\gamma})$ are analytic
in the operator norm topology, 
$\Pi_{z}$ is a rank-one projection with 
$\Pi_{0}(\varphi)(z) = \nu (\varphi)$ for any $\varphi \in \scr B_{\gamma}$ and $x \in \bb P(\bb V)$,
$\Pi_{z} N_{z} = N_{z} \Pi_{z} = 0$. 
Moreover, there exist $n_0 \geq 1$ and $q \in (0,1)$ such that for any $z \in \bb D_{\delta}$,
we have $\|N_{z}^{n_0}\|_{\scr B_{\gamma} \to \scr B_{\gamma}} \leq q$. 
\end{lemma}  

Recall that we denote by $\lambda_{\mu}$ the first Lyapunov exponent of $\mu$, see \eqref{def-Lyapunov-001}. 
We now assume that $\lambda_{\mu} = 0$. 
In this case, the eigenvalue $\lambda_{z}$ has the asymptotic expansion: as $z \to 0$,   
\begin{align} \label{decomp-lambda001}
\lambda_{z} = 1 + \frac{\upsilon_{\mu}^2}{2} z^2 + O(|z|^3). 
\end{align}
Recall that the asymptotic variance $\upsilon_{\mu}^2$ appearing in \eqref{decomp-lambda001} is strictly positive, see \cite{Boug-Lacr85, BQ16b}.

\begin{lemma}\label{Lem_StrongNonLattice}
Assume that $\Gamma_{\mu}$ is proximal and strongly irreducible, and $\mu$ admits finite exponential moments. 
Then,
for any $t \neq 0$, the operator $P_{{\bf i} t}$ has spectral radius strictly less than $1$
in $\scr B_{\gamma}$. 
More precisely,  for any compact set $K \subset \bb R \setminus \{0\}$, 
there exist constants $c_K, c_K' >0$ such that for any $\varphi\in \scr B_{\gamma}$ and $n \geq 1$, 
\begin{align}\label{Spectral_Radius_Uniform}
\sup_{t \in K}   \| P_{ {\bf i}t }^n \varphi \|_{\scr B_{\gamma}}
\leq  c_K' e^{- c_K n}  \|\varphi\|_{\scr B_{\gamma}}.
\end{align}
\end{lemma}

\begin{proof}
The proof of the first assertion can be found in \cite{Boug-Lacr85, BQ16b}. 
By adapting the arguments of Lemma 5.2 in \cite{GQX23}, we obtain the inequality \eqref{Spectral_Radius_Uniform}. 
\end{proof}


\subsection{Local limit theorems}

In the following we state local limit theorems for products of random matrices with a precise estimation of remainder terms.
The main difficulty is to give explicit dependence of the remainder terms on target functions.  


The next lemma is an adaptation of  Lemma 5.3 in \cite{GQX23} to the case of functions defined on $\bb P(\bb V) \times \bb R$. 
It describes the class of target functions that we will use all over the paper. 

\begin{lemma}\label{Lem_Measurability}
Let $F$ be a real-valued function on $\bb P(\bb V) \times \bb R$ such that
\begin{enumerate}[label=\arabic*., leftmargin=*]
\item For any $t \in \bb R$, the function $x \mapsto F(x, t)$ is $\gamma$-H\"older continuous on $\bb P(\bb V)$.  
\item For any $x \in \bb P(\bb V)$, the function $t \mapsto F(x, t)$ is measurable on $\bb R$.
 \end{enumerate}
Then, the function $(x, t) \mapsto F(x,t)$ is measurable on $\bb P(\bb V) \times \bb R$ and 
the function $t \mapsto \| F (\cdot, t) \|_{\scr B_{\gamma}}$ is measurable on $\bb R$. 
Moreover, if the integral $\int_{\bb R} \| F (\cdot, t) \|_{\scr B_{\gamma}} dt$ is finite, 
we define the partial Fourier transform $\widehat F$ of $F$ by setting for any $x \in \bb P(\bb V)$ and $u \in \bb R$,
$$\widehat F(x, u) = \int_{\bb R} e^{itu}F (x, t) dt.$$ 
This is a continuous function on $\bb P(\bb V) \times \bb R$. 
In addition, for every $u \in \bb R$, 
the function $x \mapsto \widehat F(x, u)$ is $\gamma$-H\"older continuous
and $\| \widehat F (\cdot, u) \|_{\scr B_{\gamma}} \leq \int_{\bb R} \| F (\cdot, t) \|_{\scr B_{\gamma}} dt$. 
\end{lemma}

The proof of Lemma \ref{Lem_Measurability} is similar to that of Lemma 5.3 in \cite{GQX23},
and is therefore omitted.  

We denote by $\scr H_{\gamma}$ the set of real-valued functions on $\bb P(\bb V) \times \bb R$ 
such that conditions (1) and (2) of Lemma \ref{Lem_Measurability} hold and 
the integral $\int_{\bb R} \| F (\cdot, t) \|_{\scr B_{\gamma}} dt$ is finite. 
For any compact set $K\subset \bb R $, 
denote by $\mathscr H_{\gamma, K}$ the set of functions $F \in \scr H_{\gamma}$
such that the Fourier transform $\widehat F(x, \cdot)$ has a support contained in $K$  for any $x \in \bb P(\bb V)$. 
Let $\phi$ be the standard normal density: 
$$
\phi(t)= \frac{1}{\sqrt{2\pi}}e^{-t^2/2}, \quad t \in \bb R.
$$

Now we state an ordinary local limit theorem for products of random matrices with smooth target functions,
along with explicit estimates for the remainder terms. 

\begin{theorem} \label{Theor-LLT-bound001}
Assume that $\Gamma_{\mu}$ is proximal and strongly irreducible, $\mu$ admits finite exponential moments
and that the Lyapunov exponent $\lambda_{\mu}$ is zero. 
Let $K \subset \bb R$ be a compact set. 
Then there exists a constant $c_K >0$ 
such that for any $F\in \mathscr H_{\gamma, K}$, 
$n\geq 1$ and $x \in \bb P(\bb V)$,
\begin{align*}
& \bigg| \sqrt{n}  \int_{\bb G^n} F \Big( g_n \cdots g_1 x,  \sigma (g_n \cdots g_1, x) \Big) \mu(dg_1) \ldots \mu(dg_n)   \notag\\
&  \quad  -  \int_{\bb P(\bb V) \times \bb R}  \frac{1}{\upsilon_{\mu}}  \phi \bigg( \frac{u}{ \upsilon_{\mu} \sqrt{n}} \bigg)
    F \left(x', u\right)  du  \, \nu(dx') \bigg|  
 \leq  \frac{c_K}{\sqrt{n} }   \int_{\bb R} \| F (\cdot, t) \|_{\scr B_{\gamma}} dt.  
\end{align*} 
\end{theorem}

\begin{proof}
This is obtained as Theorem 5.4 in \cite{GQX23} 
by using the spectral gap properties of the perturbed transfer operator
$P_z$ (see Lemmas \ref{Lem_Perturbation} and \ref{Lem_StrongNonLattice}). 
\end{proof}



Next we extend Theorem \ref{Theor-LLT-bound001} to a larger class of target functions. 
To this aim, we need to introduce additional notation. 
Let $\ee >0$. For functions $f$ and $g$ on $\bb R$, 
we say that the function $g$ $\ee$-dominates the function $f$
(or $f$ $\ee$-minorates $g$) if for any $t \in \bb R$, it holds that
\begin{align*}
f(t) \leq g(t +v),  \quad  \forall  \  |v| \leq \ee.
\end{align*}
In this case we write $f \leq_{\ee} g$ or $g \geq_{\ee} f$.
For any functions $F$ and $G$  on $\bb P(\bb V) \times \bb R$, we say that $F \leq_{\ee} G$ if $F(x, \cdot) \leq_{\ee} G(x, \cdot)$ 
for any $x \in \bb P(\bb V)$.

%
%
%
%

Below, for any function $F \in \scr H_{\gamma}$,  we use the notation
\begin{align}\label{def-norm-H-gamma}
\| F \|_{\scr H_{\gamma}} =  \int_{\bb R} \| F(\cdot, t) \|_{\scr B_{\gamma}}   dt,
\quad
\| F \|_{\nu \otimes \Leb} = \int_{\bb R} \int_{\bb P(\bb V)}  | F(x, t) |    \nu(dx) dt,
\end{align}
where we recall that $\nu$ is the unique $\mu$-stationary Borel probability measure on $\bb P(\bb V)$.


\begin{theorem}\label{Lem_LLT_Nonasm}
Assume that $\Gamma_{\mu}$ is proximal and strongly irreducible, $\mu$ admits finite exponential moments
and that the Lyapunov exponent $\lambda_{\mu}$ is zero. 
There exists a constant $c>0$ with the following property:
for any $\ee \in (0, \frac{1}{8})$, 
there exists a constant $c_{\ee} >0$ such that for any non-negative function $F$ and any function $G \in \scr H_{\gamma}$
satisfying $F \leq_{\ee} G$, $n\geq 1$ and $x \in \bb P(\bb V)$, 
\begin{align}
&   \int_{\bb G^n}  
  F \Big( g_n \cdots g_1 x,  \sigma (g_n \cdots g_1, x) \Big) \mu(dg_1) \ldots \mu(dg_n)  \notag\\
&   \leq   \frac{1}{\sqrt{n}} \int_{\bb P(\bb V)}  \int_{\bb R}  
    \frac{1}{\upsilon_{\mu}} \phi \bigg(\frac{u}{\upsilon_{\mu} \sqrt{n}} \bigg)  G(x', u)  du  \nu(dx')  
        +  \frac{c\ee}{\sqrt{n}} \| G \|_{\nu \otimes \Leb} 
  +  \frac{c_{\ee}}{ n }  \| G \|_{\scr H_{\gamma}},    \label{LLT_Upper_aa}
\end{align}
and for any non-negative  function $F$ and non-negative  functions $G, H \in \scr H_{\gamma}$
satisfying $H \leq_{\ee} F \leq_{\ee} G$, $n\geq 1$ and $x \in \bb P(\bb V)$, 
\begin{align}
&  \int_{\bb G^n}  
  F \Big( g_n \cdots g_1 x,  \sigma (g_n \cdots g_1, x) \Big) \mu(dg_1) \ldots \mu(dg_n)   \notag\\
& \geq  \frac{1}{\sqrt{n}} \int_{\bb P(\bb V)} \int_{\mathbb{R}}  
  \frac{1}{\upsilon_{\mu}} \phi \bigg(\frac{u}{\upsilon_{\mu} \sqrt{n}} \bigg)   H(x', u)  du  \nu(dx')  
    -  \frac{ c \ee}{\sqrt{n}} \| G \|_{\nu \otimes \Leb}  
   -  \frac{c_{\ee}}{ n } 
      \left(  \| G \|_{\scr H_{\gamma}} 
       +  \| H \|_{\scr H_{\gamma}} \right).   
       \label{LLT_Lower_aa}
\end{align}
\end{theorem}

\begin{proof}
This is obtained as Theorem 5.8 in \cite{GQX23} 
by using Theorem \ref{Theor-LLT-bound001} and the elementary properties of the relation $\leq_{\ee}$. 
\end{proof}

One of the main difficulties of the paper will be to estimate the norms $\| G \|_{\nu \otimes \Leb}$ and $\| G \|_{\scr H_{\gamma}}$
 on the right-hand sides of \eqref{LLT_Upper_aa} and \eqref{LLT_Lower_aa}.
 We will address this issue in Sections \ref{sec: Effective CLLT} and \ref{Proof of main Theorem}, particularly in Propositions \ref{Lem_Inequality_Aoverline}
 and \ref{Lem_HolderNormPsi}.


\section{Effective conditioned local limit theorems} \label{sec: Effective CLLT}

\subsection{Formulation of the result}
One of the key contributions of this paper is an effective version of the conditioned local limit theorem, 
referred to as the Caravenna-type local limit theorem, as formulated in Subsection \ref{Sec-Caravenna-type}. 
This result serves as a crucial step toward establishing Theorem \ref{Thm-CLLT-cocycle}. 
As previously mentioned, our approach does not rely on the reversal of the random walk, distinguishing it from prior works
 \cite{Carav05, DW15, GLL20, GQX23, GX2022IID, PP23}, 
 where the reversal technique is central. More specifically, our method closely resembles that of 
 \cite{GQX23} which was developed for hyperbolic dynamical systems; however, we encounter additional challenges in our context. First, we cannot expect strong contractivity of the Markov chain, which prevents us from simply substituting one starting point for another
 in a straightforward manner. Second, we must address the unboundedness of the cocycle $\sigma$, 
 which complicates the analysis and necessitates the use of truncation techniques.



\begin{theorem} \label{t-A 001}
Assume that $\Gamma_{\mu}$ is proximal and strongly irreducible, $\mu$ admits finite exponential moments
and that the Lyapunov exponent $\lambda_{\mu}$ is zero. 
Then, 
there exists a constant $\eta >0$ 
with the following properties. 
\begin{enumerate}[label=\arabic*., leftmargin=*]
\item  
For any $\ee \in (0,\frac{1}{8})$ and $\beta >0$, 
there exist constants $c_{\beta}, c_{\ee} > 0$ such that for any $n \geq 1$, $x \in \bb P(\bb V)$, $t \leq n^{1/2 - \beta}$,  
any measurable functions $F, G: \bb P(\bb V) \times \bb R \to \bb R_+$ satisfying $F \leq_{\ee} G$ and $G \in \scr H_{\gamma}$, 
\begin{align}\label{eqt-A 001}
&  n  \bb E \Big[  F \Big( g_n \cdots g_1 x,  t + \sigma (g_n \cdots g_1, x)  \Big);  \tau_{x, t} >n  \Big]   \notag\\
&  \leq  \frac{2 V(x, t)}{  \sqrt{2\pi}  \upsilon_{\mu}^2 } 
  \int_{\bb P(\bb V)}  \int_{\mathbb R_+}  G \left(x', t' \right)
\phi^+ \bigg( \frac{t'}{\upsilon_{\mu} \sqrt{n}} \bigg) dt' \nu(dx')   \notag\\
& \quad + c_{\beta} (1 + \max\{t, 0\}) \left( \ee^{1/4}  +  \ee^{-1/4} n^{-\eta} \right) 
  \| G \|_{\nu \otimes \Leb }     +   \frac{c_{\ee} (1 + \max\{t, 0\})}{ \sqrt{n}} \|  G \|_{\scr H_{\gamma}}.  
\end{align}
\item 
For any $\ee \in (0,\frac{1}{8})$ and $\beta >0$, 
there exist constants $c_{\beta}, c_{\ee} > 0$ such that for any $n \geq 1$, $x \in \bb P(\bb V)$, $t \leq n^{1/2 - \beta}$,  
any measurable functions $F, G, H: \bb P(\bb V) \times \bb R \to \bb R_+$ satisfying $H \leq_{\ee} F \leq_{\ee} G$
and $G, H \in \scr H_{\gamma}$, 
\begin{align} \label{eqt-A 002}
&  n  \bb E \Big[  F \Big( g_n \cdots g_1 x,  t + \sigma (g_n \cdots g_1, x)  \Big);  \tau_{x, t} >n  \Big]   \nonumber\\
&  \geq  \frac{2 V(x, t)}{  \sqrt{2\pi} \upsilon_{\mu}^2} 
\int_{\bb P(\bb V)}  \int_{\bb R_+}  H(x', t') \phi^+ \bigg(\frac{t'}{ \upsilon_{\mu} \sqrt{n} } \bigg) dt' \nu(d x')  \notag \\
& \quad  - c_{\beta} (1 + \max\{t, 0\}) \left( \ee^{1/12}  +   \ee^{-1/4}n^{-\eta}  \right)    \| G \|_{\nu \otimes \Leb }  \notag\\
& \quad  -  \frac{c_{\ee} (1 + \max\{t, 0\}) }{\sqrt{n}} 
  \left( \left\Vert  G  \right\Vert _{\scr H_{\gamma}} +  \left\Vert  H  \right\Vert _{\scr H_{\gamma}}  \right). 
\end{align}
\end{enumerate}
\end{theorem}

As a corollary of the proof of Theorem \ref{t-A 001}, 
we also obtain the following bound which holds uniformly in the starting point $t\in \bb R$. 

\begin{corollary} \label{Cor-Cara-Bound}
Assume that $\Gamma_{\mu}$ is proximal and strongly irreducible, $\mu$ admits finite exponential moments
and that the Lyapunov exponent $\lambda_{\mu}$ is zero. 
Then, for any $\ee \in (0,\frac{1}{8})$,  there exists a constant $c>0$ such that for any $n \geq 1$, $x \in \bb P(\bb V)$, $t \in \bb R$, 
any measurable functions $F, G: \bb P(\bb V) \times \bb R \to \bb R_+$ satisfying $F \leq_{\ee} G$ and $G \in \scr H_{\gamma}$, 
\begin{align*}
&   \bb E \Big[  F \Big( g_n \cdots g_1 x,  t + \sigma (g_n \cdots g_1, x)  \Big);  \tau_{x, t} >n  \Big]    \notag\\
 & \qquad \leq   \frac{c}{n} \left( 1 + \max\{t, 0\} \right) 
\left( 
  \| G \|_{\nu \otimes \Leb }  
    +   \frac{1}{ \sqrt{n}} \|  G \|_{\scr H_{\gamma}} \right).  
\end{align*}
\end{corollary}

Corollary \ref{Cor-Cara-Bound} will be applied to establish Theorem \ref{Thm-CLLT-cocycle-bound-001} as well as Corollary \ref{Cor-bound-n-001}. 





\subsection{Preparatory statements}
The normal density with variance $v > 0$ is given by 
\begin{align}\label{def-normal-density-v}
\phi_{v} (t) = \frac{1}{\sqrt{2 \pi v} } e^{- \frac{t^2}{2 v}},  \quad t \in \bb R,  
\end{align}
while the Rayleigh density with scale parameter $\sqrt{v}$ is denoted by 
\begin{align*}
\phi^+_{v}(t)=\frac{t}{v} e^{-\frac{t^2}{2 v}} \mathds 1_{\mathbb R_+} (t), \quad  t \in \mathbb R. 
\end{align*}
Recall that, when $v=1$, we have $\phi(t) = \phi_1(t)$ and $\phi^+(t) = \phi^+_{1}(t)$, $t \in \mathbb R$. 
The following lemma from \cite{GX2022IID}  
shows that when $v$ is small, the convolution $\phi_{v} * \phi^+_{1-v}$ closely approximates the Rayleigh density. 
\begin{lemma} \label{t-Aux lemma}
For any $v \in (0,1/2]$ and $t\in \bb R$, it holds that 
\begin{align*} 
 - |t| e^{- \frac{t^2}{2}} \mathds 1_{\{t < 0\}}
\leq  \phi_{v} * \phi^+_{1- v}(t) -  \sqrt{1-v} \phi^+(t)
\leq   \sqrt{v}  e^{ -\frac{t^2}{2v} } +   |t| e^{- \frac{t^2}{2}} \mathds 1_{\{t < 0\}}.   
\end{align*}
\end{lemma}

We shall need a generalization of Fuk's inequality for martingales due to Haeusler \cite[Lemma 1]{Hae84}, 
 which will play an important role in our analysis. 

\begin{lemma}\label{Lem_Fuk}
Let $\xi_1, \ldots, \xi_n$ be a martingale difference sequence with respect to the non-decreasing 
$\sigma$-fields $\mathscr F_0, \mathscr F_1, \ldots, \mathscr F_n$.
Then, for all $u, v, w > 0$,
\begin{align*} 
\bb P \left(  \max_{1 \leq k \leq n} \left| \sum_{i=1}^k \xi_i \right| \geq u \right)
& \leq  2 \exp \left\{ \frac{u}{v} \left( 1 - \log \frac{uv}{w} \right) \right\}
 \notag \\
& \quad  +   \sum_{i=1}^n \bb P \left( |\xi_i|  > v \right)  + 2 \bb P \left(  \sum_{i=1}^n \bb E \left( \xi_i^2 |  \mathscr F_{i-1} \right) > w \right). 
\end{align*}
\end{lemma}

Using this lemma together with the spectral gap properties stated in Section \ref{Spectral gap theory}, 
we derive the following Fuk-type inequality that involves  
a target function on the Markov chain $(g_n \cdots g_1 x)_{n\geq 0}$.
This result will be used to control the remainder term in the proof of the lower bound \eqref{eqt-A 002} 
in the Caravenna-type conditioned local limit theorem. 

 \begin{lemma} \label{FK-joint-inequality}
Assume that $\Gamma_{\mu}$ is proximal and strongly irreducible, $\mu$ admits finite exponential moments
and that the Lyapunov exponent $\lambda_{\mu}$ is zero. 
Then, there exist constants $c, c', c_0>0$ such that, for any  $M > c_0$, $n\geq 1$, $x \in \bb P(\bb V)$ and any nonnegative function $\varphi \in \scr B_{\gamma}$, 
\begin{align*} 
 \bb E   \Big( \varphi \left( g_n \cdots g_1 x  \right);  \max_{1 \leq j \leq n } | \sigma(g_j \cdots g_1, x) |  \geq  M \sqrt{n}   \Big)  
  \leq    2 \nu(\varphi)  e^{ - c \sqrt{M} }  
   + c' e^{-cn^{1/6}} \| \varphi \|_{\scr B_{\gamma}}. 
\end{align*}
\end{lemma}

\begin{proof}
It follows from Lemma 10.18 in \cite{BQ16b} that the cocycle $\sigma$ can be centered. Specifically, 
there exists a continuous function 
$\psi_0$ on $\bb P(\bb V)$ such that, for any $x \in \bb P(\bb V)$, 
\begin{align} \label{lyapunov exp is 0-002}
\int_{\bb G } (\sigma(g, x) + \psi_0(g x) - \psi_0(x))  \mu(dg) = 0.
\end{align}
For $g \in \bb G$ and $x \in \bb P(\bb V)$, we set 
\begin{align*}
\sigma_0(g, x) = \sigma(g, x) + \psi_0(g x) - \psi_0(x). 
\end{align*}
Then, for any $x \in \bb P(\bb V)$, the sequence $(\sigma_0(g_n \cdots g_1, x))_{n \geq 1}$ forms a martingale with respect to the natural filtration,
and there exists a constant $c_0 >0$ such that, uniformly on $\Omega$, 
\begin{align*}
\sup_{n \geq 1} \sup_{x \in \bb P(\bb V)} |\sigma(g_n \cdots g_1, x) - \sigma_0(g_n \cdots g_1, x)| \leq c_0. 
\end{align*} 
Set $p=n-[n^{1/3}]$ and define 
\begin{align*}
B_{n,x} = \left\{ \max_{k \in [p+1, n]} \sigma_0(g_k, g_{k-1} \cdots g_1 x) \leq \frac{1}{3} M n^{1/6} \right\}. 
\end{align*}
On the set $B_{n,x}$, we have
\begin{align*}
\max_{1 \leq j \leq n } | \sigma(g_j \cdots g_1, x) |
& \leq \max_{1 \leq j \leq n } | \sigma_0(g_j \cdots g_1, x) | + c_0 \notag\\
&  \leq \max_{1 \leq j \leq p } | \sigma_0(g_j \cdots g_1, x) | + \frac{1}{3} M \sqrt{n} + c_0.  
\end{align*}
Therefore, for sufficiently large $M$ (so that $M > 6 c_0$), we get that, for any $x \in \bb P(\bb V)$, 
\begin{align}\label{Lower_F_ee_kkk-002}
 \bb E  \Big(   \varphi \left( g_n \cdots g_1 x  \right);  \max_{1 \leq j \leq n } | \sigma(g_j \cdots g_1, x) |  \geq  M \sqrt{n}   \Big)  
  \leq  J_n (x) + \|\varphi\|_{\infty} \bb P (B_{n,x}^c), 
\end{align}
where 
\begin{align*}
J_n (x) := 
\bb E   
 \left(   \varphi \left( g_n \cdots g_1 x  \right);  \max_{1 \leq j \leq p } | \sigma_0(g_j \cdots g_1, x) |  \geq \frac{1}{2} M \sqrt{n}   \right). 
\end{align*}
By the definition of $B_{n,x}$ and applying Chebyshev's inequality, we obtain 
\begin{align}\label{Fuk-martingale-000}
\bb P (B_{n,x}^c)
& \leq  \sum_{k = p+1}^n  \bb P \bigg( \sup_{x' \in \bb P(\bb V)} \sigma_0(g_k, x') > \frac{1}{3} M n^{1/6} \bigg) \notag\\
& \leq  e^{ - \frac{\alpha}{3} M n^{1/6}} \sum_{k = p+1}^n  \bb E \left( \|g_k\|^{\alpha} \right)
\leq  c n e^{ - 2 \alpha c_0 n^{1/6}}
\leq  c  e^{ -  \alpha c_0   n^{1/6}}, 
\end{align}
where $\alpha >0$ is a constant from \eqref{Exponential-moment}. 
Note that $P^k \varphi  (x) = \bb E \varphi ( g_k \cdots g_1 x)$ for $k\geq 1$.
By Lemma \ref{Lem_Perturbation} with $z=0$, there exist constants $c, c' >0$ such that for any $k\geq 1$,
$$
\sup_{x \in \bb P(\bb V) } \left| P^k \varphi  (x) -   \nu(\varphi) \right|  
\leq  c'  e^{-c k} \| \varphi \|_{\scr B_{\gamma}}. 
$$ 
Using this and the Markov property, we get that for any $n \geq 1$ and $x \in \bb P(\bb V)$, 
\begin{align}\label{KKK-markov property-aa}
J_n (x) 
& = \bb E \left(  P^{[n^{1/3}]} \varphi  \left( g_p \cdots g_1 x \right); \max_{1 \leq j \leq p } | \sigma_0(g_j \cdots g_1, x) |  \geq \frac{1}{2} M \sqrt{n}  \right) \notag\\
& \leq  
\nu(\varphi) 
 \bb P \left(  \max_{1 \leq j \leq p } | \sigma_0(g_j \cdots g_1, x) |  \geq \frac{1}{2} M \sqrt{n}   \right)   
+ c' e^{-c n^{1/3}} \| \varphi \|_{\scr B_{\gamma}}.
\end{align}
Applying Fuk's inequality for martingales (Lemma \ref{Lem_Fuk}) 
with $u = \frac{1}{2}M\sqrt{ n} $, $v =  \frac{1}{2} \sqrt{n}$ and $w=  \frac{1}{4 e^2} M n$ 
(so that $\frac{u}{v} = M$ and $\frac{uv}{w} = e^2$), 
we obtain 
\begin{align}\label{Fuk-martingale-001}
 \bb P \left(  \max_{1 \leq j \leq p } | \sigma_0(g_j \cdots g_1, x) |  \geq \frac{1}{2} M \sqrt{n}   \right) 
& \leq  \bb P \left(  \max_{1 \leq j \leq n } | \sigma_0(g_j \cdots g_1, x) |  \geq \frac{1}{2} M \sqrt{n}   \right)  \notag\\
& \leq  2 e^{- M} + n \sup_{x' \in \bb P(\bb V)} \bb P \Big( |\sigma(g_1, x')| > \frac{1}{2} \sqrt{n} \Big)  \notag\\
& \quad +  \bb P \bigg( \frac{1}{n} \sum_{i=1}^n \sup_{x' \in \bb P(\bb V)} |\sigma(g_i, x')|^2 >  \frac{1}{4 e^2} M   \bigg).
 \end{align} 
 For the second term on the right-hand side of \eqref{Fuk-martingale-001}, 
 by Chebyshev's inequality and the fact that $\mu$ has finite exponential moments (cf.\ \eqref{Exponential-moment}), we have 
 \begin{align}\label{Fuk-martingale-002}
n \sup_{x' \in \bb P(\bb V)} \bb P \Big( |\sigma(g_1, x')| > \frac{1}{2} \sqrt{n} \Big) 
\leq  n e^{ - \frac{\alpha}{2} \sqrt{n} } 
\leq  c e^{ - \frac{\alpha}{3} \sqrt{n} }. 
\end{align}
 For the last term on the right-hand side of \eqref{Fuk-martingale-001}, we use Chebyshev's inequality to get
 \begin{align*}
 \bb P \left( \frac{1}{n} \sum_{i=1}^n \sup_{x' \in \bb P(\bb V)} |\sigma(g_i, x')|^2 >  \frac{1}{4 e^2} M   \right) 
 \leq  e^{- \frac{\alpha}{2 e} \sqrt{M} }  \bb E  \exp \left( \alpha \sqrt{\frac{1}{n} \sum_{i=1}^n \sup_{x' \in \bb P(\bb V)} |\sigma(g_i, x')|^2} \right). 
\end{align*}
Note that the function $u \mapsto e^{\alpha \sqrt{u}}$ is convex on the interval $[\alpha^{-2}, \infty)$
and strictly increasing on the interval $[0, \infty)$.
In particular, on $[0, \infty)$, this function is dominated  by the convex function $u \mapsto \max\{e, e^{\alpha \sqrt{u}} \}$. 
Consequently, the exponential term within the latter expectation can be controlled as follows: 
\begin{align*}
 \exp \left( \alpha \sqrt{\frac{1}{n} \sum_{i=1}^n \sup_{x' \in \bb P(\bb V)} |\sigma(g_i, x')|^2} \right) 
& \leq  \frac{1}{n} \sum_{i=1}^n  \max\left\{ e,  \exp \left( \alpha \sup_{x' \in \bb P(\bb V)} |\sigma(g_i, x')| \right) \right\} \notag\\
& \leq \frac{1}{n} \sum_{i=1}^n  \max\left\{ e,  \|g_i\|^{\alpha} \right\}. 
\end{align*}
By the moment assumption \eqref{Exponential-moment}, the expectation of the above expression is bounded by a constant. 
Combining this with the inequalities \eqref{Lower_F_ee_kkk-002}, \eqref{Fuk-martingale-000},
 \eqref{KKK-markov property-aa}, \eqref{Fuk-martingale-001} and \eqref{Fuk-martingale-002} completes the proof of the lemma.
\end{proof}

We need an effective bound that will enable a precise comparison between $\sigma (g_n \cdots g_1, x)$ with $\log \|g_n \cdots g_1\|$.
This result will be frequently utilized to control the remainder terms when establishing the lower bound  
\eqref{eqt-A 002} of the Caravenna-type conditioned local limit theorem. 

\begin{lemma}\label{Lem-cocycle-norm-001}
Assume that $\Gamma_{\mu}$ is proximal and strongly irreducible, and $\mu$ admits finite exponential moments. 
Then, there exists a constant $\eta_0 > 0$ with the following property. 
For any $\eta \in (0, \eta_0)$, there exist constants $\beta, c >0$ such that, for any $1 \leq j \leq n$ and $x \in \bb P(\bb V)$, 
\begin{align}\label{inequa-sigma-xx-001}
\bb P \Big( \sigma (g_n \cdots g_1, x) - \log \|g_n \cdots g_1\| \leq - \eta j  \Big) 
\leq c e^{-\beta j}, 
\end{align}
and for any $1 \leq j \leq n$ and $x, x'  \in \bb P(\bb V)$, 
\begin{align}\label{inequa-sigma-xx-002}
\bb P \Big( | \sigma (g_n \cdots g_1, x) -  \sigma (g_n \cdots g_1, x')| \geq  \eta j  \Big)
\leq c e^{-\beta j}. 
\end{align}
\end{lemma}

\begin{proof}
We retain the notation from Subsection 14.1 in \cite{BQ16b}. 
We claim that there exist constants $\beta, \gamma, C >0$ such that for any $1 \leq j \leq n$, 
\begin{align}\label{inequa-density-point-001}
\bb P \left( d \left( y_{g_n \cdots g_1}^m, y_{g_j \cdots g_1}^m \right) \leq  e^{-\beta j} \right) > 1 -  C e^{-\gamma j}. 
\end{align}
Indeed, fix an arbitrary $x \in \bb P(\bb V)$. By (14.6) of Proposition 14.3 in \cite{BQ16b}, for any $1 \leq j \leq n$, 
we have simultaneously for some constants $\beta, \gamma, C >0$, 
\begin{align*}
& \bb P \left( d \left( g_n \cdots g_{1} x, x_{g_n \cdots g_{1}}^M \right) \leq  \frac{1}{2} e^{-\beta n} \right) > 1 - \frac{C}{2} e^{-\gamma n}, \notag\\
& \bb P \left( d \left( g_n \cdots g_{1} x, x_{g_n \cdots g_{n-j+1}}^M \right) \leq  \frac{1}{2} e^{-\beta j} \right) > 1 - \frac{C}{2}  e^{-\gamma j}. 
\end{align*}
By the triangle inequality, we get 
\begin{align*}
\bb P \left( d \left( x_{g_n \cdots g_{1}}^M, x_{g_n \cdots g_{n-j+1}}^M \right) \leq  e^{-\beta j} \right) > 1 -  C e^{-\gamma j}. 
\end{align*}
Applying this relation in the dual representation of $\Gamma_{\mu}$ in the dual space $\bb V^*$ of $\bb V$ yields
\begin{align*}
\bb P \left( d \left( y_{g_1 \cdots g_n}^m, y_{g_{n-j+1} \cdots g_n}^m \right) \leq  e^{-\beta j} \right) > 1 -  C e^{-\gamma j}. 
\end{align*}
Reversing the order of the variables, we derive \eqref{inequa-density-point-001}. 

To conclude the proof of the lemma, 
we apply inequality (14.5) in  \cite[Proposition 14.3]{BQ16b}, which gives, for any $x \in \bb P(\bb V)$
and $1 \leq j \leq n$, 
\begin{align*}
\bb P \left( \delta \left( x, y_{g_j \cdots g_1}^m \right) \geq   e^{-\eta j} \right) > 1 -  C e^{-c j}. 
\end{align*}
By \eqref{inequa-density-point-001}, we obtain 
\begin{align*}
\bb P \left( \delta \left( x, y_{g_n \cdots g_1}^m \right) \geq   e^{-\eta j} - e^{-\beta j} \right) > 1 -  C e^{-c j} -  C e^{-\gamma j}. 
\end{align*} 
Thus, inequality \eqref{inequa-sigma-xx-001} follows by the first inequality in \cite[Lemma 14.2 (i)]{BQ16b}. 
Finally, applying \eqref{inequa-sigma-xx-001} with $x = x'$ and using the triangle inequality, 
we get \eqref{inequa-sigma-xx-002}. 
\end{proof}

Below we will need the following classical interpolation inequality. 

\begin{lemma}\label{Lem-interpolation-inequ}
Let $X$ be a metric space, and let $f$ be a Lipschitz 
continuous function with Lipschitz constant $B>0$, such that $\sup |f|=A$. 
Then, for any $0< \gamma \leq 1$, $f$ is $\gamma$-H\"older continuous with 
H\"older constant at most $(2A)^{1 - \gamma} B^{\gamma}$. 
\end{lemma}

\begin{proof}
Fix $u >0$, whose value is to be determined later. 
For $x,x'$ in $X$, 
if $d(x,x')\leq u$, we have
$$|f(x)-f(x')|\leq B d(x,x')\leq B u \left( u^{-1} d(x,x') \right)
\leq  B u^{1-\gamma} d(x,x')^\gamma,$$
with the latter inequality following from the fact that $u^{-1} d(x,x')\leq 1$. 
On the other hand, if $d(x,x')\geq u$, we have $u^{\gamma} \leq d(x,x')^\gamma$, which gives
$$|f(x)-f(x')|\leq 2A \leq 2A u^{-\gamma} d(x,x')^\gamma.$$
The conclusion follows by taking $u = \frac{2A}{B}$.
\end{proof}

The following important lemma shows the contraction property of the random walk $(\sigma(g_n\cdots g_1, x))_{n \geq 0}$, 
which will be frequently used in subsequent analysis. 

\begin{lemma}\label{Lem-Holder-cocycle}
Assume that $\mu$ has finite exponential moments, and that $\Gamma_{\mu}$ is proximal and strongly irreducible. 
Then, for any $p \geq 1$, there exists a constant $\gamma_0 > 0$ with the following property.
For any $\gamma \in (0, \gamma_0)$,
there exists a constant $C>0$ such that, for any $x, x' \in \bb P(\bb V)$ and $n\geq 1$, we have 
\begin{align*}
\bb E^{1/p} \Big( |\sigma(g_n\cdots g_1,x)-\sigma(g_n\cdots g_1,x')|^p \Big)  \leq C  d(x,x')^{\gamma}
\end{align*}
and 
\begin{align*}
\bb E^{1/p} \Big( \max_{1 \leq j \leq n} |\sigma(g_j \cdots g_1,x)-\sigma(g_j\cdots g_1,x')|^p \Big) \leq C d(x,x')^{\gamma}. 
\end{align*}
\end{lemma}

\begin{proof}
We only prove the second inequality since it implies the first one. 
We shall apply Lemma \ref{Lem-interpolation-inequ} to the norm cocycle. For $g \in \bb G$, we set $N(g)=\max\{\|g\|,\|g^{-1}\| \}$. By the definition of the norm cocycle, we have that for any $g \in \bb G$ and $x \in \bb P(\bb V)$, 
\begin{align*}
|\sigma(g,x)|\leq\log N(g)
\end{align*}
and, by Lemma 13.2 in Subsection 13.1 of \cite{BQ16b}, 
there exists a constant $C>0$ such that for any $g \in \bb G$ and $x,x' \in \bb P(\bb V)$, 
\begin{align*}
|\sigma(g,x)-\sigma(g,x')|\leq CN(g)^2 d(x,x'). 
\end{align*}
Applying Lemma \ref{Lem-interpolation-inequ}, we obtain that, for any $g \in \bb G$, 
and $x, x' \in \bb P(\bb V)$ and $0< \gamma \leq 1$,
\begin{align*}
|\sigma(g,x)-\sigma(g,x')|\leq  C^{\gamma} (2\log N(g))^{1 - \gamma} N(g)^{2 \gamma} d(x,x')^{\gamma}. 
\end{align*}
In particular, since $\mu$ has finite exponential moments, for any $p \geq 1$ and small 
enough $\gamma > 0$, there exists a constant $C_{p, \gamma} >0$ such that, for any $x,x'$ in $\bb P(\bb V)$,
\begin{align*}
\mathbb E ( |\sigma(g_1,x)-\sigma(g_1,x')|^p)\leq  C_{p, \gamma} d(x,x')^{p \gamma}. 
\end{align*}
By the cocycle property, we get that, for any $1 \leq j \leq n$ and $x,x'$ in $\bb P(\bb V)$, 
\begin{align*}
 |\sigma(g_j \cdots g_1,x)-\sigma(g_j \cdots g_1,x')|
& \leq  \sum_{k = 1}^j \left| \sigma(g_k,g_{k-1}\cdots g_1x)-\sigma(g_k,g_{k-1}\cdots g_1x') \right| \notag\\
& \leq  \sum_{k = 1}^n \left| \sigma(g_k,g_{k-1}\cdots g_1x)-\sigma(g_k,g_{k-1}\cdots g_1x') \right|. 
\end{align*}
Now, for any $p \geq 1$, $n\geq 1$ and $x,x'$ in $\bb P(\bb V)$, 
we write, by Minkowski's inequality, 
\begin{align*}
& \bb E^{1/p} \left( \max_{1 \leq j \leq n} |\sigma(g_j \cdots g_1,x)-\sigma(g_j\cdots g_1,x')|^p \right) \notag\\
& \leq \sum_{k=1}^{n} \bb E^{1/p} \Big( |\sigma(g_k,g_{k-1}\cdots  g_1x) - \sigma(g_k,g_{k-1} \cdots g_1x')|^p \Big).
\end{align*}
Using the Markov property and the above bound, we get that, for any $1\leq k\leq n$, 
\begin{align*}
& \bb E \Big( |\sigma(g_k,g_{k-1}\cdots g_1x)-\sigma(g_k,g_{k-1}\cdots g_1x')|^p \Big) \notag\\
& \leq C_{p, \gamma} \mathbb E \Big( d(g_{k-1}\cdots g_1x,g_{k-1} \cdots  g_1x')^{p \gamma} \Big).
\end{align*}
Now, applying the contraction property of the random walk (see Bougerol and Lacroix \cite[Chapter V, Proposition 2.3]{Boug-Lacr85}), 
we know that, for any $p \geq 1$ and small enough $\gamma > 0$, there exist constants 
$\beta, C>0$ such that, for any $k\geq 1$ and $x,x'$ in $\bb P(\bb V)$,
\begin{align*}
\mathbb E \Big( d(g_{k-1}\cdots g_1x,g_{k-1}\cdots g_1x')^{p \gamma} \Big)
\leq  Ce^{-\beta k} d(x,x')^{p \gamma}.
\end{align*}
Gathering all the inequalities, we conclude that
\begin{align*}
 \bb E^{1/p} \left( \max_{1 \leq j \leq n} |\sigma(g_j \cdots g_1,x) - \sigma(g_j\cdots g_1,x')|^p \right)  
& \leq  \left( C C_{p, \gamma} \right)^{1/p} \sum_{k=1}^{n} e^{- \beta p^{-1} k}  d(x,x')^{\gamma}  \notag\\
& \leq  \frac{ \left( C C_{p, \gamma} \right)^{1/p} }{1 - e^{-\beta p^{-1}}} d(x,x')^{\gamma},
\end{align*}
as should be proved.
\end{proof}

The following result implies that the integral of the function $t \mapsto \bb P (t + \sigma (g_m \cdots g_1, x) \in [a, b])$
is bounded by a constant, for any $m \geq 1$ and $-\infty < a < b < \infty$. 
The proof  is based on Lemma \ref{Lem-cocycle-norm-001}, 
the main difficulty lying in proving the uniformity in the starting point $x \in \bb P(\bb V)$. 

\begin{lemma}\label{Lem-inte-H-sup}
Assume that $\Gamma_{\mu}$ is proximal and strongly irreducible, $\mu$ admits finite exponential moments
and that the Lyapunov exponent $\lambda_{\mu}$ is zero.  
Let $\kappa$ be a non-negative smooth compactly supported function on $\bb R$.
Then, there exists a constant $c>0$ such that for any function $H \in L^1(\bb R)$ and any $m \geq 1$,  
we have 
\begin{align*}
\int_{\bb R}   \sup_{x \in \bb P(\bb V)}  \bb E  \left[ H *  \kappa \Big( t + \sigma (g_m \cdots g_1, x) \Big) \right]  dt
\leq  c \int_{\bb R} |H(t)|dt.  
\end{align*}
\end{lemma}

\begin{proof}
Without loss of generality, we assume that the function $H$ is non-negative on $\bb R$. 
Let $\eta >0$ be a fixed constant such that Lemma \ref{Lem-cocycle-norm-001} holds. 
For $x \in \bb P(\bb V)$ and $1 \leq j \leq m$, we set 
\begin{align*}
B_{x, m, j} = \Big\{ \sigma (g_m \cdots g_1, x) - \log \| g_m \cdots g_1\| \in (- j \eta, - (j-1) \eta] \Big\}.
\end{align*}
By Lemma \ref{Lem-cocycle-norm-001}, there exist constants $\beta, c >0$ such that, for any $1 \leq j \leq m$ and $x \in \bb P(\bb V)$, 
\begin{align*}
\bb P (B_{x, m, j})  \leq c e^{-\beta j}. 
\end{align*}
By the change of variable $t = t' -  \log \| g_m \cdots g_1\|$, we get 
\begin{align}\label{decom-Integral-sup-001}
&  \int_{\bb R}   \sup_{x \in \bb P(\bb V)}  \bb E  \left[ H *  \kappa \Big( t + \sigma (g_m \cdots g_1, x) \Big) \right]  dt \notag\\
& = \int_{\bb R}   \sup_{x \in \bb P(\bb V)}  \bb E  \left[ H *  \kappa \Big( t + \sigma (g_m \cdots g_1, x) -  \log \| g_m \cdots g_1\|  \Big) \right]  dt \notag\\
& =  \sum_{j =1}^{m} \int_{\bb R}   \sup_{x \in \bb P(\bb V)} 
\bb E  \left[ H *  \kappa \Big( t + \sigma (g_m \cdots g_1, x) - \log \| g_m \cdots g_1\| \Big); B_{x, m, j}  \right] dt \notag\\
& \quad +  \int_{\bb R}   \sup_{x \in \bb P(\bb V)} 
\bb E  \Big[ H *  \kappa \Big( t + \sigma (g_m \cdots g_1, x)  \Big);  
 \sigma (g_m \cdots g_1, x) - \log \| g_m \cdots g_1\| \leq - m \eta \Big] dt \notag\\
& =: J_1 + J_2. 
\end{align}
%
%
Let $a > 0$ be such that $\kappa$ has support in $[-a, a]$. Thus, for any $u \in \bb R$,
we have $\kappa(u) \leq \|\kappa\|_{\infty} \mathds 1_{\{|u| \leq a\}}$. 
Consequently, for the first term $J_1$,  we derive that for any $1 \leq j \leq m$, 
\begin{align}\label{decom-Integral-sup-002}
 J_1  
 & \leq  \|\kappa\|_{\infty}  \sum_{j =1}^{m} \sup_{x \in \bb P(\bb V)}  \bb P \left( B_{x, m, j} \right) \int_{\bb R}   H * \mathds 1_{[-a - \eta, a + \eta] } (t) dt   
 \notag\\
&  \leq c (a + \eta) \|\kappa\|_{\infty} \sum_{j =1}^{m} e^{-\beta j}  \int_{\bb R}   H(t)dt 
 \leq c \int_{\bb R}   H(t)dt. 
\end{align}
For the second term $J_2$, we use Fubini's theorem and the Cauchy-Schwarz inequality to get that, 
for any $x \in \bb P(\bb V)$, $t \in \bb R$ and $m \geq 1$, 
\begin{align*}
& \bb E  \left[ H *  \kappa \Big( t + \sigma (g_m \cdots g_1, x)  \Big); \sigma (g_m \cdots g_1, x) - \log \| g_m \cdots g_1\| \leq - m \eta \right]  \notag\\
& =  \int_{\bb R} H(s) \bb E \Big[ \kappa  \Big( t+ \sigma (g_m \cdots g_1, x) - s  \Big);  
 \sigma (g_m \cdots g_1, x) - \log \| g_m \cdots g_1\| \leq - m \eta \Big] ds  \notag\\ 
& \leq  \bb P^{1/2} \Big( \sigma (g_m \cdots g_1, x) - \log \| g_m \cdots g_1\| \leq - m \eta \Big)  \notag\\
& \qquad  \times 
  \int_{\bb R} H(s)   \bb E^{1/2} \left[ \kappa^2  \Big( t+ \sigma (g_m \cdots g_1, x) - s  \Big) \right]   ds. 
\end{align*}
By Lemma \ref{Lem-cocycle-norm-001} and the fact that $\kappa$ is supported in $[-a, a]$, this leads to 
\begin{align*}
& \bb E  \left[ H *  \kappa \Big( t + \sigma (g_m \cdots g_1, x)  \Big); \sigma (g_m \cdots g_1, x) - \log \| g_m \cdots g_1\| \leq - m \eta \right]  \notag\\
& \leq c e^{-\frac{\beta}{2} m}  \|\kappa\|_{\infty}  \int_{\bb R} H(s)   \bb P^{1/2} \left(  |t+ \sigma (g_m \cdots g_1, x) - s| \leq a  \right) ds \notag\\
& = c e^{-\frac{\beta}{2} m}  \|\kappa\|_{\infty}   \int_{\bb R} H(s) 
  \bb P^{1/2} \Big(  s-t \in [ \sigma (g_m \cdots g_1, x) -a, \sigma (g_m \cdots g_1, x)+a]  \Big) ds
\notag\\
& \leq c e^{-\frac{\beta}{2} m}  \|\kappa\|_{\infty}  
 \int_{\bb R} H(s)  
 \bb P^{1/2} \left(  s-t \in \left[ -\log \|(g_m \cdots g_1)^{-1}\| -a, \log \|g_m \cdots g_1\| +a \right]  \right) ds. 
\end{align*}
Integrating over $t \in \bb R$, we get 
\begin{align*}
& \int_{\bb R} 
\int_{\bb R} H(s)   \bb P^{1/2} \left(  s-t \in \left[ -\log \|(g_m \cdots g_1)^{-1}\| -a, \log \|g_m \cdots g_1\| +a \right]  \right) ds dt  \notag\\
& = \int_{\bb R} H(s) ds  \int_{\bb R} 
   \bb P^{1/2} \left(  u \in \left[ -\log \|(g_m \cdots g_1)^{-1}\| -a, \log \|g_m \cdots g_1\| +a \right]  \right) du \notag\\
& \leq  \int_{\bb R} H(s) ds  
  \int_{\bb R}   
\bb P^{1/2} \left( |u| \leq \max \left\{ \log \|g_m \cdots g_1\|, \log \|(g_m \cdots g_1)^{-1}\|  \right\} + a \right) du  \notag\\
& \leq  2 \int_{\bb R} H(s) ds  
 \int_{0}^{\infty} 
 \bb P^{1/2} \left( u \leq \max \left\{ \log \|g_m \cdots g_1\|, \log \|(g_m \cdots g_1)^{-1}\|  \right\} + a \right) du.  
\end{align*}
We fix $\alpha > 0$, whose value is to be chosen later. 
By the exponential Chebyshev inequality, for $u \geq 1$, we have 
\begin{align}\label{Chebyshev-log-Ng}
& \bb P \left( u \leq \max \left\{ \log \|g_m \cdots g_1\|, \log \|(g_m \cdots g_1)^{-1}\|  \right\} + a \right) \notag\\
& \leq  e^{- \alpha (u-a)}  \bb E  \left(  \max  \left\{  \|g_m \cdots g_1\|, \|(g_m \cdots g_1)^{-1}\|  \right\}^{\alpha}  \right) \notag\\
& \leq c e^{- \alpha (u-a)}  \left(  \bb E  \left(  \max \{  \|g_1\|, \|g_1^{-1}\|  \}^{\alpha}  \right) \right)^m. 
\end{align}
Therefore, by taking $\alpha>0$ to be small enough, we obtain 
\begin{align}\label{decom-Integral-sup-003} 
J_2 & \leq c e^{-\frac{\beta}{2} m}  \|\kappa\|_{\infty}  
\int_{\bb R} H(s) ds  \int_{0}^{\infty}  e^{- \frac{\alpha}{2} (u-a)} du
 \left(  \bb E  \left(  \max \{  \|g_1\|, \|g_1^{-1}\|  \}^{\alpha}  \right) \right)^{m/2} \notag\\
 & \leq c' e^{-\frac{\beta}{4} m} \int_{\bb R} H(s) ds. 
\end{align}
Substituting \eqref{decom-Integral-sup-002} and \eqref{decom-Integral-sup-003} into \eqref{decom-Integral-sup-001} gives the desired result. 
\end{proof}

To manage certain natural quantities that arise in the proof of Theorem \ref{t-A 001},
we first introduce the following definitions. 
For $\ee  >0$, define  
\begin{align}\label{Def_chiee}
\chi_{\ee} (t) = 0  \  \mbox{for} \ t \leq -\ee,  
\   \chi_{\ee} (t) = \frac{t + \ee}{\ee}   \  \mbox{for} \  t  \in (-\ee,0),
\   \chi_{\ee} (t) = 1  \  \mbox{for} \  t \geq 0.  
\end{align}
Denote $\overline\chi_{\ee}(t) = 1 - \chi_{\ee}(t)$ for $t \in \bb R$, and note that
\begin{align} \label{bounds-reversedindicators-001} 
\chi_{\ee}  \left( t-\ee \right) \leq  \mathds 1_{(0,\infty)} \left( t \right) \leq \chi_{\ee}  \left( t\right), 
\quad
\overline\chi_{\ee}  \left( t \right) \leq \mathds 1_{(-\infty,0]} \left( t \right) \leq \overline\chi_{\ee}  \left( t-\ee \right).
\end{align}
The subsequent proposition, which provides an important perspective on the behavior of
a family of linear operators 
$(A_m)_{m \geq 1}$ acting on the space $\scr H_{\gamma}$,
plays a crucial role in the proof of Theorem \ref{t-A 001}. 
The proof of this proposition, while technically intricate, draws primarily on the key insights provided by Lemmas
\ref{Lem-cocycle-norm-001}, \ref{Lem-Holder-cocycle} and \ref{Lem-inte-H-sup}, 
highlighting the difficulties in computing the associated $\gamma$-H\"older norms. 
In this spirit, a similar study will be undertaken in Section \ref{Proof of main Theorem}. 


\begin{proposition}\label{Lem_Inequality_Aoverline}
Assume that $\Gamma_{\mu}$ is proximal and strongly irreducible, $\mu$ admits finite exponential moments
and that the Lyapunov exponent $\lambda_{\mu}$ is zero.  
Let $\kappa$ be a non-negative smooth compactly supported function on $\bb R$. 
Then, for any small enough $\gamma >0$, 
 there exists a constant $c>0$ such that for any $\ee \in (0, 1)$, $G \in \scr H_{\gamma}$ and any $m \geq 1$,  
the function $A_m G$ defined on $\bb P(\bb V) \times \bb R$ by 
\begin{align*}
 A_{m} G (x, t):  =   \bb E    \bigg[ G *  \kappa \Big( g_m \cdots g_1 x, t + \sigma (g_m \cdots g_1, x) \Big)   
  \overline\chi_{\ee}  \left( t- 2 \ee + \min_{1 \leq  j \leq m}  \sigma (g_j \cdots g_1, x) \right) \bigg], 
\end{align*}
belongs to $\scr H_{\gamma}$ and satisfies 
\begin{align}\label{bounds-AmG-001}
\|  A_{m} G \|_{\nu \otimes \Leb }  \leq   \int_{\bb R} \kappa(t) dt  \| G \|_{\nu \otimes \Leb },  
\qquad 
\| A_{m} G \|_{ \scr H_{\gamma} }  \leq \frac{c}{\ee}  \| G\|_{ \scr H_{\gamma} }.  
\end{align}
\end{proposition}

\begin{proof}
As in the proof of Lemma \ref{Lem-inte-H-sup}, 
we let $a > 0$ be such that $\kappa$ has support in $[-a, a]$. Thus, for any $u \in \bb R$,
we have $\kappa(u) \leq \|\kappa\|_{\infty} \mathds 1_{\{|u| \leq a\}}$. 

To show the first inequality in \eqref{bounds-AmG-001}, since $|\overline\chi_{\ee}| \leq 1$, we have 
\begin{align*}
|A_{m} G (x, t) | \leq \bb E    \left[ |G *  \kappa| \Big( g_m \cdots g_1 x, t + \sigma (g_m \cdots g_1, x) \Big)  \right]. 
\end{align*}
As the Lebesgue measure is invariant under translation and the measure $\nu$ is $\mu$-stationary, 
by \eqref{def-norm-H-gamma}, we get that for any $m \geq 1$, 
\begin{align*}
\|  A_{m} G \|_{\nu \otimes \Leb }
& \leq  \bb E   \int_{\bb P(\bb V) \times \bb R}     \left[ |G *  \kappa| \Big( g_m \cdots g_1 x, t + \sigma (g_m \cdots g_1, x) \Big)  \right]  \nu(dx) dt  \notag\\
& =  \bb E   \int_{\bb P(\bb V) \times \bb R}     \left[ |G *  \kappa| \Big( g_m \cdots g_1 x, t'  \Big)  \right]  \nu(dx) dt'  \notag\\
& =   \int_{\bb P(\bb V) \times \bb R}    | G *  \kappa |  \left(  x, t \right)  \nu(dx) dt   \notag\\
& \leq  \int_{\bb R} \kappa(t) dt   \| G \|_{\nu \otimes \Leb }.  
\end{align*}
This finishes the proof of the first inequality in \eqref{bounds-AmG-001}.  

To prove the second inequality in \eqref{bounds-AmG-001}, 
we first recall that, by \eqref{def-norm-H-gamma}, 
\begin{align*} 
\|  A_{m} G \|_{\scr H_{\gamma}}
 =  \int_{\bb R}   \sup_{x \in \bb P(\bb V)}  | A_{m} G \left(x, t \right)|  dt 
  +  \int_{\bb R} \sup_{x, x' \in \bb P(\bb V): x \neq x'}  
    \frac{| A_{m} G \left(x, t \right)  - A_{m} G \left(x', t \right)  |}{d(x, x')^{\gamma}} dt. 
\end{align*}
For $t \in \bb R$, we set $H(t) = \sup_{x \in \bb P(\bb V)} |G(x, t)|$. 
By Lemma \ref{Lem-inte-H-sup}, there exists a constant $c>0$ such that for any $m \geq 1$, 
\begin{align}\label{Bound-Am-Norm-001}
 \int_{\bb R}   \sup_{x \in \bb P(\bb V)}  | A_{m} G \left(x, t \right)|  dt 
& \leq  \int_{\bb R}   \sup_{x \in \bb P(\bb V)}  \bb E  \left[ H *  \kappa \Big( t + \sigma (g_m \cdots g_1, x) \Big) \right]  dt  \notag\\
&  \leq  c \int_{\bb R} H(t) dt 
\leq c \| G\|_{ \scr H_{\gamma} }. 
\end{align}

Now we dominate the second term in the norm  $\|  A_{m}G \|_{\scr H_{\gamma}}$. 
We will split this term into three parts,  denoted by $I_1(x,x',t)$, $I_2(x,x',t)$ and $I_3(x,x',t)$,
which will be precisely defined below. 
 First, we define, for $x, x' \in \bb P(\bb V)$ and $t \in \bb R$, 
\begin{align}
& I_1(x,x',t)  : = \bigg| A_{m} G (x,t)   -  \bb E    \bigg[ G *  \kappa \Big( g_m \cdots g_1 x, t + \sigma (g_m \cdots g_1, x') \Big)    \notag\\ 
& \qquad\qquad\qquad\qquad\qquad\qquad
  \times \overline\chi_{\ee}  \left( t - 2 \ee + \min_{1 \leq  j \leq m}  \sigma (g_j \cdots g_1, x) \right) \bigg]  \bigg|.  \label{def-I1-x1-x2-t} 
\end{align}
Using again the fact that $|\overline\chi_{\ee}| \leq 1$, we get
\begin{align}\label{inequa-I1-Delta1}
 I_1(x,x',t) \leq \bb E   \Delta_1(x, x', m, t), 
\end{align}
where 
\begin{align*}
\Delta_1(x, x', m, t) &= \bigg| G *  \kappa \Big( g_m \cdots g_1 x, t + \sigma (g_m \cdots g_1, x) \Big)  \notag \\
   & \qquad -   G *  \kappa \Big( g_m \cdots g_1 x, t + \sigma (g_m \cdots g_1, x') \Big)  \bigg|. 
\end{align*}
We fix a constant $\eta >0$ such that Lemma \ref{Lem-cocycle-norm-001} holds. 
As in the proof of Lemma \ref{Lem-inte-H-sup}, for $x \in \bb P(\bb V)$ and $1 \leq j \leq m$, set 
\begin{align}\label{def-B-xmj}
B_{x, m, j} = \Big\{ \sigma (g_m \cdots g_1, x) - \log \| g_m \cdots g_1\| \in (- j \eta, - (j-1) \eta] \Big\}.
\end{align}
By Lemma \ref{Lem-cocycle-norm-001}, there exist constants $\beta, c >0$ such that for any $1 \leq j \leq m$ and $x \in \bb P(\bb V)$, 
\begin{align}\label{proba-B-xmj}
\bb P (B_{x, m, j})  \leq c e^{-\beta j}. 
\end{align} 
By \eqref{inequa-I1-Delta1} and \eqref{def-B-xmj}, it holds that 
\begin{align}\label{decomposi-I1-I123}
  I_1(x,x',t)  
& \leq  \sum_{1 \leq j, j' \leq m} 
\bb E  \Big( \Delta_1(x, x', m, t); B_{x, m, j} \cap B_{x', m, j'}  \Big)  \notag\\
& \quad +  
\bb E  \Big( \Delta_1(x, x', m, t); \sigma (g_m \cdots g_1, x) - \log \| g_m \cdots g_1\| \leq - m \eta \Big) \notag\\
& \quad +  
\bb E  \Big( \Delta_1(x, x', m, t); \sigma (g_m \cdots g_1, x') - \log \| g_m \cdots g_1\| \leq - m \eta \Big) \notag\\
& =: I_{11}(x, x', t) + I_{12}(x, x', t) + I_{13}(x, x', t). 
\end{align}
Now we deal with the first term $I_{11}(x, x', t)$. 
Recall that the function $\kappa$ has support in $[-a, a]$.  
For $t, t' \in \bb R$, $b \geq 0$ and $x \in \bb P(\bb V)$, if $|t - t'| \leq b$, then we derive that
\begin{align}\label{Lipschitz-G-xt}
| G * \kappa(x,t) - G * \kappa(x,t') | \leq  \|\kappa'\|_{\infty} |t-t'|  H * \mathds 1_{[-a - b, a + b]} (t), 
\end{align}
where $\kappa'$ is the first derivative of $\kappa$ and, 
as above, $H(t) = \sup_{x \in \bb P(\bb V)} |G(x, t)|$ for $t \in \bb R$.  
Therefore, by \eqref{Lipschitz-G-xt}, 
\begin{align*}
& \int_{\bb R} \sup_{x, x' \in \bb P(\bb V): x \neq x'} \frac{I_{11}(x, x', t)}{ d(x, x')^{\gamma} } dt \notag\\
& \leq  \|\kappa'\|_{\infty}   \sum_{1 \leq j, j' \leq m}   \int_{\bb R} \sup_{x, x' \in \bb P(\bb V): x \neq x'}  d(x, x')^{-\gamma}  
  \bb E  \bigg[ |\sigma (g_m \cdots g_1, x) - \sigma (g_m \cdots g_1, x')|  \notag\\
& \qquad  H * \mathds 1_{[-a - (|j' - j| +1) \eta,  a + (|j' - j| +1) \eta]} (t + \sigma (g_m \cdots g_1, x));  B_{x, m, j} \cap B_{x', m, j'}  \bigg] dt \\
&=:   \|\kappa'\|_{\infty}   \sum_{1 \leq j, j' \leq m} J_m(j,j').
\end{align*}
By the change of variable $t= t' - \log \|g_m \cdots g_1\|$, we have
\begin{align*}
& J_m(j,j') 
 =  \int_{\bb R} \sup_{x, x' \in \bb P(\bb V): x \neq x'}  d(x, x')^{-\gamma}  
 \bb E  \bigg[ |\sigma (g_m \cdots g_1, x) - \sigma (g_m \cdots g_1, x')|  \notag\\
&  
 \times  H * \mathds 1_{[-a - (|j' - j| +1) \eta, a + (|j' - j| +1) \eta]} \Big( t  + \sigma (g_m \cdots g_1, x) -  \log \|g_m \cdots g_1\| \Big);  \notag \\
& \qquad \qquad\qquad\qquad\qquad B_{x, m, j} \cap B_{x', m, j'}  \bigg] dt. 
\end{align*}
By the definition of $B_{x, m, j}$ and $B_{x', m, j'}$ (cf.\  \eqref{def-B-xmj}), it follows that
\begin{align*}
& J_m(j,j') 
 \leq   \int_{\bb R} \sup_{x, x' \in \bb P(\bb V): x \neq x'}  d(x, x')^{-\gamma}  
 \bb E  \bigg[ |\sigma (g_m \cdots g_1, x) - \sigma (g_m \cdots g_1, x')|  \notag\\
& \qquad \times  H * \mathds 1_{[-a - (|j' - j| +2) \eta, a + (|j' - j| +2) \eta]} (t - j \eta);  B_{x, m, j} \cap B_{x', m, j'}  \bigg] dt  \notag\\
& =      \sup_{x, x' \in \bb P(\bb V): x \neq x'}  
 \frac{\bb E  \Big[ |\sigma (g_m \cdots g_1, x) - \sigma (g_m \cdots g_1, x')|;  B_{x, m, j} \cap B_{x', m, j'}  \Big] }{ d(x, x')^{\gamma}   } \notag\\
& \qquad\qquad \times    \int_{\bb R} H * \mathds 1_{[-a - (|j' - j| +2) \eta, a + (|j' - j| +2) \eta]} (t) dt  \notag\\
& =      \sup_{x, x' \in \bb P(\bb V): x \neq x'}    
 \frac{\bb E  \Big[ |\sigma (g_m \cdots g_1, x) - \sigma (g_m \cdots g_1, x')|;  B_{x, m, j} \cap B_{x', m, j'}  \Big]}{ d(x, x')^{\gamma} }  \notag\\
& \qquad\qquad \times    \int_{\bb R} H(t) dt  \,   \Big( 2 a + 2 (|j' - j| +2) \eta \Big). 
\end{align*}
Applying H\"older's inequality, Lemma \ref{Lem-Holder-cocycle} and \eqref{proba-B-xmj}, we get that there exist 
constants $\beta, C>0$ such that, for any $x, x' \in \bb P(\bb V)$ and $1 \leq j, j' \leq m$, 
\begin{align*}
& \bb E  \Big( |\sigma (g_m \cdots g_1, x) - \sigma (g_m \cdots g_1, x')|;  B_{x, m, j} \cap B_{x', m, j'}  \Big)  \notag\\
& \leq \bb E^{1/3}  \Big( |\sigma (g_m \cdots g_1, x) - \sigma (g_m \cdots g_1, x')|^3  \Big) \bb P^{1/3} \left( B_{x, m, j} \right) \bb P^{1/3} \left( B_{x', m, j'} \right)
\notag\\
& \leq C d(x,x')^{\gamma} e^{- \frac{\beta}{3} (j + j')}. 
\end{align*}
Consequently, using the fact that $\|\kappa'\|_{\infty} \leq c$ and $\int_{\bb R} H(t) dt \leq \| G\|_{ \scr H_{\gamma} }$, we obtain 
\begin{align}\label{Bound-Am-Norm-002}
  \int_{\bb R} \sup_{x, x' \in \bb P(\bb V): x \neq x'} \frac{I_{11}(x, x', t)}{ d(x, x')^{\gamma} } dt 
& \leq  C  \int_{\bb R} H(t) dt   \sum_{1 \leq j, j' \leq m}    e^{- \frac{\beta}{3} (j + j')}   \Big( 2 a + 2 (|j' - j| +2) \eta \Big) \notag\\
& \leq C'  \| G\|_{ \scr H_{\gamma} }. 
\end{align}

Now we deal with the second term $I_{12}(x, x', t)$ defined by \eqref{decomposi-I1-I123}. 
For $g \in \bb G$, we set $N(g)=\max\{\|g\|,\|g^{-1}\| \}$. 
Using the fact that $|\sigma (g_m \cdots g_1, x)| \leq \log N(g_m \cdots g_1)$
and applying \eqref{Lipschitz-G-xt} with $b = 2 \log N(g_m \cdots g_1)$, 
we obtain that, for any $m \geq 1$, $x, x' \in \bb P(\bb V)$ and $t \in \bb R$, 
\begin{align*}
&\Delta_1(x, x', m, t) =  \Big| G * \kappa \Big( g_m \cdots g_1 x, t + \sigma (g_m \cdots g_1, x) \Big)   \notag\\
& \qquad\qquad\qquad\qquad\qquad - G * \kappa \Big( g_m \cdots g_1 x, t + \sigma (g_m \cdots g_1, x') \Big)  \Big|  \notag\\
& \leq  \|\kappa'\|_{\infty}  \Big| \sigma (g_m \cdots g_1, x) - \sigma (g_m \cdots g_1, x') \Big| \notag\\ 
&\qquad \times  H * \mathds 1_{[-a - 2 \log N(g_m \cdots g_1), a + 2 \log N(g_m \cdots g_1)]} \Big( t + \sigma (g_m \cdots g_1, x) \Big) \notag\\
& \leq  \|\kappa'\|_{\infty} \Big| \sigma (g_m \cdots g_1, x) - \sigma (g_m \cdots g_1, x') \Big|   \notag\\
& \qquad \times H * \mathds 1_{[-a - 3 \log N(g_m \cdots g_1), a + 3 \log N(g_m \cdots g_1)]} (t). 
\end{align*}
Substituting this bound into $I_{12}(x, x', t)$ (cf.\ \eqref{decomposi-I1-I123}) leads to 
\begin{align*}
& \int_{\bb R} \sup_{x, x' \in \bb P(\bb V): x \neq x'} \frac{I_{12}(x, x', t)}{ d(x, x')^{\gamma} } dt \notag\\
& \leq  \|\kappa'\|_{\infty}     \int_{\bb R} \sup_{x, x' \in \bb P(\bb V): x \neq x'}  d(x, x')^{-\gamma}  
 \bb E  \bigg[ |\sigma (g_m \cdots g_1, x) - \sigma (g_m \cdots g_1, x')|  \notag\\
& \qquad\qquad \times  H * \mathds 1_{[-a - 3 \log N(g_m \cdots g_1), a + 3 \log N(g_m \cdots g_1)]} (t);  \notag\\
& \qquad\qquad\qquad\qquad   \sigma (g_m \cdots g_1, x) - \log \| g_m \cdots g_1\| \leq - m \eta   \bigg] dt \notag\\
& = \|\kappa'\|_{\infty}     \int_{\bb R} \sup_{x, x' \in \bb P(\bb V): x \neq x'}  d(x, x')^{-\gamma}   
 \int_{\bb R} H(s) 
 \bb E  \bigg[ |\sigma (g_m \cdots g_1, x) - \sigma (g_m \cdots g_1, x')|  \notag\\
& \qquad\qquad \times  \mathds 1_{[-a - 3 \log N(g_m \cdots g_1), a + 3 \log N(g_m \cdots g_1)]} (t-s);  \notag\\
& \qquad\qquad\qquad\qquad\qquad  \sigma (g_m \cdots g_1, x) - \log \| g_m \cdots g_1\| \leq - m \eta   \bigg] ds dt, 
\end{align*}
where in the last equality we use the definition of $H * \mathds 1_{[a', b']}$. 
By H\"older's inequality and Lemmas \ref{Lem-Holder-cocycle} and \ref{Lem-cocycle-norm-001}, 
there exist constants $\beta, c', c >0$ such that for any $m \geq 1$, 
\begin{align}\label{Bound-Am-Norm-003}
& \int_{\bb R} \sup_{x, x' \in \bb P(\bb V): x \neq x'} \frac{I_{12}(x, x', t)}{ d(x, x')^{\gamma} } dt \notag\\ 
& \leq  \|\kappa'\|_{\infty}     \int_{\bb R} \sup_{x, x' \in \bb P(\bb V): x \neq x'}  d(x, x')^{-\gamma}  \notag\\ 
& \quad \times \int_{\bb R} H(s) 
 \bb E^{1/3}  \left( \left| \sigma (g_m \cdots g_1, x) - \sigma (g_m \cdots g_1, x') \right|^3 \right)  \notag\\
& \quad \times \bb P^{1/3} \Big(  -a - 3 \log N(g_m \cdots g_1) \leq t-s \leq  a + 3 \log N(g_m \cdots g_1)  \Big) \notag\\
& \quad \times  \bb P^{1/3}  \Big( \sigma (g_m \cdots g_1, x) - \log \| g_m \cdots g_1\| \leq - m \eta   \Big) ds dt  \notag\\
& \leq c'  e^{- \frac{\beta}{3} m} 
  \int_{\bb R}  \int_{\bb R} H(s)  
   \bb P^{1/3} \Big(  -a - 3 \log N(g_m \cdots g_1) \leq t-s \leq  a + 3 \log N(g_m \cdots g_1)  \Big) 
  ds dt \notag\\
 & =  c'  e^{- \frac{\beta}{3} m}   \int_{\bb R} H(s) ds   
  \int_{\bb R} \bb P^{1/3} \Big( |u|  \leq  a + 3 \log N(g_m \cdots g_1)  \Big) du \notag\\
 & \leq  c e^{- \frac{\beta}{6} m}   \int_{\bb R} H(s) ds, 
\end{align}
where in the last inequality we have adapted the proof of \eqref{Chebyshev-log-Ng} and \eqref{decom-Integral-sup-003}. 

For the third term $I_{13}(x, x', t)$ (cf.\  \eqref{decomposi-I1-I123}),
proceeding in the same way as in the proof of the second term $I_{12}(x', x, t)$, 
one has 
\begin{align}\label{Bound-Am-Norm-003-sym}
\int_{\bb R} \sup_{x, x' \in \bb P(\bb V): x \neq x'} \frac{I_{13}(x, x', t)}{ d(x, x')^{\gamma} } dt
\leq c e^{- \frac{\beta}{6} m}   \int_{\bb R} H(s) ds. 
\end{align}
Combining \eqref{decomposi-I1-I123}, \eqref{Bound-Am-Norm-002},
\eqref{Bound-Am-Norm-003} and \eqref{Bound-Am-Norm-003-sym}, 
we get 
\begin{align}\label{bound-I1-x1-x2-first}
\int_{\bb R} \sup_{x, x' \in \bb P(\bb V): x \neq x'} \frac{I_{1}(x, x', t)}{ d(x, x')^{\gamma} } dt
\leq  c \| G\|_{ \scr H_{\gamma} }. 
\end{align}


Next, we handle the second term in the decomposition of the norm  $\|  A_{m}G \|_{\scr H_{\gamma}}$,
 denoted by $I_2(x,x',t)$ and defined as follows:  
for any $x,x' \in \bb P(\bb V)$ and $t \in \bb R$, 
\begin{align*}
 I_2(x,x',t) & : = \bigg|   \bb E  \bigg[  \overline\chi_{\ee}  \left( t- 2\ee + \min_{1 \leq  j \leq m}  \sigma (g_j \cdots g_1, x) \right)  \notag\\
 & \qquad \times  \bigg(  G *  \kappa \left( g_m \cdots g_1 x, t + \sigma (g_m \cdots g_1, x') \right)    \notag\\
 &  \qquad\qquad  -  G *  \kappa \left( g_m \cdots g_1 x', t + \sigma (g_m \cdots g_1, x') \right)  \bigg)  \bigg]  \bigg|. 
\end{align*}
As the function $G$ is in the space $\scr H_{\gamma}$, the function 
\begin{align*}
L(t) = \sup_{x, x' \in \bb P(\bb V): x \neq x'}   \frac{|G(x,t) - G(x',t)|}{ d(x, x')^{\gamma} }, \quad t \in \bb R 
\end{align*}
is integrable on $\bb R$.
Therefore, 
using the fact that $|\overline\chi_{\ee}| \leq 1$, 
we get that, for any $x,x' \in \bb P(\bb V)$ and $t \in \bb R$, 
\begin{align}\label{Bound_I2zzt}
 I_2(x,x',t) 
  \leq  d(x, x')^{\gamma}  \bb E   \Big( L * \kappa \left(  t + \sigma (g_m \cdots g_1, x')  \right) \Big).
\end{align}
By Lemma \ref{Lem-inte-H-sup}, there exists a constant $c>0$ such that for any $m \geq 1$,  
\begin{align}\label{Bound-Am-Norm-004}
 \int_{\bb R} \sup_{x, x' \in \bb P(\bb V): x \neq x'} \frac{I_{2}(x, x', t)}{ d(x, x')^{\gamma} } dt 
& \leq  \int_{\bb R} \sup_{x' \in \bb P(\bb V)}  \bb E \Big(  L * \kappa \left(  t + \sigma (g_m \cdots g_1, x')  \right)  \Big)  dt \notag\\ 
& \leq  c   \int_{\bb R} L(t) dt \leq c \| G\|_{ \scr H_{\gamma} }. 
\end{align}

Finally, we manage the third term in the decomposition of the norm  $\|  A_{m}G \|_{\scr H_{\gamma}}$,
which is denoted by $I_3(x,x',t)$ and defined by, for $x,x' \in \bb P(\bb V)$ and $t \in \bb R$,   
\begin{align*}
& I_3(x,x',t)  : = \bigg|   \bb E    
     G *  \kappa \Big( g_m \cdots g_1 x', t + \sigma (g_m \cdots g_1, x') \Big)   \notag\\
 &  \times
 \bigg[ \overline\chi_{\ee}  \Big( t- 2 \ee + \min_{1 \leq  k \leq m}   \sigma (g_k \cdots g_1, x) \Big) 
  -  \overline\chi_{\ee}  \Big( t-\ee + \min_{1 \leq  k \leq m}   \sigma (g_k \cdots g_1, x')  \Big)  \bigg]   \bigg|. 
\end{align*}
By \eqref{Def_chiee} and recalling that $\overline\chi_{\ee}(t) = 1 - \chi_{\ee}(t)$ for $t \in \bb R$, 
we have that $\overline\chi_{\ee}$ is $1/\ee$-Lipschitz continuous on $\bb R$. 
By reasoning in the same way as in \eqref{decom-Integral-sup-001}, we get
\begin{align}\label{Bound_I3zzt}
& I_3(x,x',t)  
 \leq \frac{c}{\ee}  \bb E  \Big[  H *  \kappa \Big(  t + \sigma (g_m \cdots g_1, x') \Big)  
 \max_{1 \leq k \leq m} \Big| \sigma (g_k \cdots g_1, x) - \sigma (g_k \cdots g_1, x') \Big|  \Big] \notag\\
 & = \frac{c}{\ee}  \sum_{j = 1}^m \bb E  \Big[  H *  \kappa \Big(  t + \sigma (g_m \cdots g_1, x') \Big)  
 \max_{1 \leq k \leq m} \Big| \sigma (g_k \cdots g_1, x) - \sigma (g_k \cdots g_1, x') \Big|;  B_{x', m, j} \Big] \notag\\
 & \quad + \frac{c}{\ee}   \bb E  \bigg[  H *  \kappa \Big(  t + \sigma (g_m \cdots g_1, x') \Big) 
 \max_{1 \leq k \leq m} \Big| \sigma (g_k \cdots g_1, x) - \sigma (g_k \cdots g_1, x') \Big|;  \notag\\
 & \qquad\qquad \sigma (g_m \cdots g_1, x') - \log\|g_m \cdots g_1\| \leq - m \eta \bigg]  \notag\\
&  =: I_{31}(x,x',t) + I_{32}(x,x',t), 
\end{align}
where $B_{x', m, j}$ is defined by \eqref{def-B-xmj}. 
For $I_{31}(x,x',t)$, using the facts that $\kappa$ has support in $[-a, a]$ and 
that $|\sigma (g_m \cdots g_1, x') - \log \| g_m \cdots g_1\| + j \eta | \leq \eta$ on the set $B_{x', m, j}$, 
we get 
\begin{align*}
 \Big| H *  \kappa \Big(  t + \sigma (g_m \cdots g_1, x') \Big)  \Big|  
 \leq  \|\kappa\|_{\infty} 
H *  \mathds 1_{[-a -\eta, a + \eta]} \Big(  t + \log\|g_m \cdots g_1\| - j \eta \Big), 
\end{align*}
so that 
\begin{align*}
& \int_{\bb R} \sup_{x, x' \in \bb P(\bb V): x \neq x'} \frac{I_{31}(x, x', t)}{ d(x, x')^{\gamma} } dt  \notag\\
& \leq \frac{c}{\ee} \|\kappa\|_{\infty} \sum_{j = 1}^m   \int_{\bb R} \sup_{x, x' \in \bb P(\bb V): x \neq x'} d(x, x')^{-\gamma}  
 \bb E  \bigg[  H *  \mathds 1_{[-a -\eta, a + \eta]} \Big(  t + \log\|g_m \cdots g_1\| - j \eta \Big)  \notag\\
& \qquad\qquad \times  \max_{1 \leq k \leq m}  \Big| \sigma (g_k \cdots g_1, x) - \sigma (g_k \cdots g_1, x') \Big|; B_{x', m, j} \bigg] dt \notag\\
& = \frac{c}{\ee} \|\kappa\|_{\infty} \sum_{j = 1}^m   \int_{\bb R} \sup_{x, x' \in \bb P(\bb V): x \neq x'} d(x, x')^{-\gamma}
 \bb E  \bigg[  H *  \mathds 1_{[-a -\eta, a + \eta]} \left(  t \right)  \notag\\
& \qquad \times \max_{1 \leq k \leq m} \Big| \sigma (g_k \cdots g_1, x) - \sigma (g_k \cdots g_1, x') \Big|; B_{x', m, j} \bigg] dt \notag\\
& = \frac{c}{\ee} 2 (a + \eta) \|\kappa\|_{\infty} \int_{\bb R} H(t) dt  \sum_{j = 1}^m \sup_{x, x' \in \bb P(\bb V): x \neq x'} d(x, x')^{-\gamma} \notag\\
& \qquad \times \bb E  \bigg(   \max_{1 \leq k \leq m} \Big| \sigma (g_k \cdots g_1, x) - \sigma (g_k \cdots g_1, x') \Big|; B_{x', m, j} \bigg), 
\end{align*}
where the first equality holds by using the change of variable $t' = t + \log\|g_m \cdots g_1\| - j \eta$. 
By the Cauchy-Schwarz inequality, Lemma \ref{Lem-Holder-cocycle} and \eqref{proba-B-xmj},  we conclude that
\begin{align}\label{Bound-Am-Norm-005}
\int_{\bb R} \sup_{x, x' \in \bb P(\bb V): x \neq x'} \frac{I_{31}(x, x', t)}{ d(x, x')^{\gamma} } dt
\leq  \frac{c}{\ee} \int_{\bb R} H(t) dt \leq  \frac{c}{\ee} \| G\|_{ \scr H_{\gamma} }. 
\end{align}
For $I_{32}(x,x',t)$ defined by \eqref{Bound_I3zzt}, note that, as the function $\kappa$ has support in $[-a, a]$
and $|\sigma (g_m \cdots g_1, x)| \leq \log N(g_m \cdots g_1)$, we have
\begin{align*}
H *  \kappa \Big(  t + \sigma (g_m \cdots g_1, x') \Big) 
\leq  H *  \mathds 1_{[-a - \log N(g_m \cdots g_1), a + \log N(g_m \cdots g_1)]}  (t). 
\end{align*}
Therefore, by H\"older's inequality, 
there exist constants $\beta, c, c' > 0$ such that for any $x,x' \in \bb P(\bb V)$ and $t \in \bb R$, 
\begin{align*}
& I_{32}(x,x',t)   \leq \frac{c}{\ee}  \|\kappa\|_{\infty}  
\bb E  \bigg[  H *  \mathds 1_{[-a - \log N(g_m \cdots g_1), a + \log N(g_m \cdots g_1)]}  (t) 
  \notag\\
 & \qquad\qquad\qquad\qquad\qquad \times \max_{1 \leq k \leq m} \Big| \sigma (g_k \cdots g_1, x) - \sigma (g_k \cdots g_1, x') \Big|;  \notag\\
 & \qquad\qquad\qquad\qquad\qquad\qquad  \sigma (g_m \cdots g_1, x') - \log\|g_m \cdots g_1\| \leq - m \eta \bigg]  \notag\\
  & = \frac{c}{\ee}  \|\kappa\|_{\infty}  \int_{\bb R}  H(s) \bb E  \bigg[  
\mathds 1_{[-a - \log N(g_m \cdots g_1), a + \log N(g_m \cdots g_1)]} \left(  t - s \right) 
  \notag\\
 & \qquad\qquad\qquad\qquad\qquad \times \max_{1 \leq k \leq m} \Big| \sigma (g_k \cdots g_1, x) - \sigma (g_k \cdots g_1, x') \Big|;   \notag\\
 & \qquad\qquad\qquad\qquad\qquad\qquad \sigma (g_m \cdots g_1, x') - \log\|g_m \cdots g_1\| \leq - m \eta \bigg] ds \notag\\
 & \leq \frac{c}{\ee}  \|\kappa\|_{\infty}  \int_{\bb R}  H(s) \bb P^{1/3}  \Big( |t - s| \leq  a + \log N(g_m \cdots g_1) \Big) ds 
  \notag\\
 &\quad \times \bb E^{1/3} \left( \max_{1 \leq k \leq m} \Big| \sigma (g_k \cdots g_1, x) - \sigma (g_k \cdots g_1, x') \Big|^3 \right) \\
&\quad \times \bb P^{1/3}  \Big( \sigma (g_m \cdots g_1, x') - \log\|g_m \cdots g_1\| \leq - m \eta \Big)  \notag\\
 & \leq  \frac{c'}{\ee} e^{-\frac{\beta}{3} m} d(x, x')^{\gamma} \int_{\bb R}  H(s) \bb P^{1/3}  \Big( |t - s| \leq  a + \log N(g_m \cdots g_1) \Big) ds,
\end{align*}
where in the last inequality we have used Lemmas \ref{Lem-Holder-cocycle} and \ref{Lem-cocycle-norm-001}. 
Consequently, 
\begin{align}\label{Bound-Am-Norm-006}
 \int_{\bb R} \sup_{x, x' \in \bb P(\bb V): x \neq x'} \frac{I_{32}(x, x', t)}{ d(x, x')^{\gamma} } dt 
& \leq  \frac{c}{\ee} e^{-\frac{\beta}{3} m} \int_{\bb R} \bb P^{1/3}  \Big( |u| \leq  a + \log N(g_m \cdots g_1) \Big) du \int_{\bb R} H(t) dt  \notag\\
& \leq  \frac{c}{\ee} \| G\|_{ \scr H_{\gamma} }, 
\end{align}
where the last inequality is obtained by using the same method as in the proof of
 \eqref{Chebyshev-log-Ng} and \eqref{decom-Integral-sup-003}.

Putting together \eqref{Bound-Am-Norm-001}, 
\eqref{bound-I1-x1-x2-first},  \eqref{Bound-Am-Norm-004}, \eqref{Bound-Am-Norm-005} 
 and \eqref{Bound-Am-Norm-006} yields the second inequality in \eqref{bounds-AmG-001}.  
 This completes of the proof of Proposition \ref{Lem_Inequality_Aoverline}. 
\end{proof}


\subsection{Conditioned central limit estimates}
In order to establish Theorem \ref{t-A 001}, 
we will make use of the following central limit type estimates for conditioned random walks.
We denote by $\Phi^+$ the standard Rayleigh distribution function:  
\begin{align*} 
 \Phi^+(t) = \left(1- e^{-t^2/2}\right) \mathds 1_{\{ t \geq 0 \}},  
\quad  t \in \bb R.  
\end{align*}
It is clear that $\Phi^+(t) = \int_{-\infty}^t \phi^+(u) du$ for $t \in \bb R$,
 where $\phi^+$ is defined by \eqref{Rayleigh law-001}. 

\begin{theorem}\label{Thm-CCLT-001}
Assume that $\Gamma_{\mu}$ is proximal and strongly irreducible, $\mu$ admits finite exponential moments
and that the Lyapunov exponent $\lambda_{\mu}$ is zero. 
Then, there exists a constant $\eta >0$ with the following properties. For any $\beta>0$, 
there exists a constant $c_{\beta} >0$ such that for any $n \geq 1$, $x \in \bb P(\bb V)$, $s \geq 0$ and $t \leq n^{1/2 - \beta}$, 
\begin{align}\label{CCLT-effective-version-001}
 \left| \bb P \left( \frac{t + \sigma (g_n \cdots g_1, x)}{\upsilon_{\mu} \sqrt{n}}  \leq  s,  \tau_{x, t} > n \right)
 -  \frac{2 V(x, t)}{ \sqrt{2 \pi n} \upsilon_{\mu} }  \Phi^+ (s)  \right| 
 \leq c_{\beta}  \frac{(1 + \max\{t, 0\})}{n^{1/2 + \eta} }. 
\end{align}
Moreover, there exists a constant $c >0$ such that for any $n \geq 1$, $x \in \bb P(\bb V)$ and $t \in \bb R$, 
\begin{align}\label{Upper-bound-tau-xt}
\bb P \left(  \tau_{x, t} > n \right) \leq c \frac{1 + \max\{t, 0\}}{\sqrt{n}}. 
\end{align}
\end{theorem}

The asymptotic \eqref{CCLT-effective-version-001} can be obtained by Theorems 2.4 and 2.5 of \cite{GLL20}, 
which can be applied to our situation thanks to \cite{GLL24}. 
The bound \eqref{Upper-bound-tau-xt} follows from Theorem 2.3 of \cite{GLP17}. 
Note that these result are only stated for $t \geq 0$, but the case of $t<0$ can be dealt with in the same way.
In \cite{GLP17} the result is uniform over $s \in [0, s_0]$, for some constant $s_0>0$, but using the approach from
\cite{GX2022IID}, we can show that it is uniform in $s\in \bb R_+$. 


Actually, from Theorem \ref{Thm-CCLT-001} 
we will deduce a more general version
with target functions, which will be used in the sequel. 


\begin{corollary}\label{Cor-CCLT-target-function}
Assume that $\Gamma_{\mu}$ is proximal and strongly irreducible, $\mu$ admits finite exponential moments
and that the Lyapunov exponent $\lambda_{\mu}$ is zero. 
Then, there exists a constant $\eta >0$ with the following property. 
For any $\beta>0$, there exists constants $c_{\beta}>0$ 
such that, for any 
$n \geq 1$, $x \in \bb P(\bb V)$ and $t \leq n^{1/2 - \beta}$, 
and any smooth function $\varphi$ with integrable derivative on $\bb R$ vanishing at $\pm \infty$, 
\begin{align*}
&  \left| \bb E \left( \varphi \left( \frac{t + \sigma (g_n \cdots g_1, x)}{\upsilon_{\mu} \sqrt{n}} \right);   \tau_{x, t} > n \right)
 -  \frac{2 V(x, t)}{ \sqrt{2 \pi n} \upsilon_{\mu} }  \int_{\bb R_+}  \varphi(s) \phi^+ (s) ds  \right|   \notag\\
& \qquad\qquad\qquad\qquad\qquad\qquad\qquad\qquad
  \leq  \frac{c_{\beta}  (1 + \max\{t, 0\}) }{n^{1/2+\eta}} \int_{\bb R_+}  |\varphi'(s)| ds. 
\end{align*}
\end{corollary}

\begin{proof}
By integration by parts and the fact that $\varphi(s) \to 0$ as $s \to \infty$, we derive that for any 
$n \geq 1$, $x \in \bb P(\bb V)$ and $t \in \bb R$, 
\begin{align*}
& \bb E \left( \varphi \left( \frac{t + \sigma (g_n \cdots g_1, x)}{\upsilon_{\mu} \sqrt{n}} \right);   \tau_{x, t} > n \right)
 - \frac{2 V(x, t)}{ \sqrt{2 \pi n} \upsilon_{\mu} }  \int_{\bb R_+}  \varphi(s) \phi^+ (s) ds  \notag\\
& = - \int_{0}^{\infty}  \left[ \bb P \left( \frac{t + \sigma (g_n \cdots g_1, x)}{\upsilon_{\mu} \sqrt{n}}  \leq  s,  \tau_{x, t} > n \right) 
- \frac{2 V(x, t)}{ \sqrt{2 \pi n} \upsilon_{\mu} }  \Phi^+ (s)  \right] \varphi'(s) ds. 
\end{align*}
By Theorem \ref{Thm-CCLT-001}, 
for any $\beta>0$, 
there exists a constant $c_{\beta} >0$ such that for any $n \geq 1$, $x \in \bb P(\bb V)$ and $t \leq n^{1/2 - \beta}$, 
\begin{align*}
\Bigg| \int_{0}^{\infty} &  \Bigg[ \bb P \left( \frac{t + \sigma (g_n \cdots g_1, x)}{\upsilon_{\mu} \sqrt{n}}  \leq  s,   \tau_{x, t} > n \right) 
- \frac{2 V(x, t)}{ \sqrt{2 \pi n} \upsilon_{\mu} }  \Phi^+ (s)  \Bigg] \varphi'(s) ds \Bigg| \\
& \qquad\qquad\qquad\qquad\qquad\qquad\qquad \leq 
\frac{c_{\beta} (1 + \max\{t, 0\})}{n^{1/2+\eta}}  \int_{0}^{\infty}  |\varphi'(s)| ds. 
\end{align*}
This concludes the proof of the corollary. 
\end{proof}


\subsection{Proof of Theorem \ref{t-A 001}: upper bound} 
To goal of this subsetion is to establish the inequality \eqref{eqt-A 001} in Theorem \ref{t-A 001}. 
It is enough to prove \eqref{eqt-A 001} only for sufficiently large $n > n_0 := n_0(\ee)$, 
where $n_0(\ee)$ depends on $\ee$; 
otherwise, the bound becomes trivial.

Let $\ee\in (0,\frac{1}{8})$. With $\delta = \sqrt{\ee}$, set
$m=\left[ \delta n \right]$ and $k = n-m$, where $\left[ \delta n \right]$ denotes the integrer part of $\delta n$. 
Note that $\frac{1}{2}\delta \leq \frac{m}{k} \leq \frac{\delta}{1-\delta}$ for $n \geq \frac{2}{\sqrt{\ee}}$. 
Denote, for $n \geq 1$, $x \in \bb P(\bb V)$ and $t \in \bb R$,  
\begin{align}\label{def-Psi-x-t-001}
\Psi_n (x,t) = \bb E \Big[  F \Big( g_n \cdots g_1 x, t + \sigma (g_n \cdots g_1, x) \Big);  \tau_{x, t} > n   \Big], 
\end{align}
where $F$ is a non-negative measurable function on $\bb P(\bb V) \times \bb R$. 
Applying the Markov property gives that for any $x \in \bb P(\bb V)$ and $t \in \bb R$, 
\begin{align}\label{JJJ-markov property}
\Psi_n (x,t) = \bb E \Big[ \Psi_m \Big( g_k \cdots g_1 x, t + \sigma (g_k \cdots g_1, x) \Big);  \tau_{x, t} > k  \Big]. 
\end{align}
By bounding the indicator function $\mathds 1_{ \{\tau_{x, t} > n\} }$ 
by $\mathds 1_{ \left\{ t + \sigma (g_m \cdots g_1, x) \geq 0 \right\}} $ 
in the definition of $\Psi_m$, we get
\begin{align}\label{JJJJJ-1111-000}
\Psi_m (x,t) 
  & \leq  \bb E \Big[ F \Big( g_m \cdots g_1 x, t + \sigma (g_m \cdots g_1, x) \Big); t + \sigma (g_m \cdots g_1, x) \geq 0  \Big]   \notag\\
 & =: J_{m}(x,t). 
\end{align}
Let $G_{\ee} (x, t) = G (x, t) \chi_{\ee} (t - \ee)$ for $x \in \bb P(\bb V)$ and $t \in \bb R$, 
where $G$ is a non-negative measurable function on $\bb P(\bb V) \times \bb R$,
$\ee \in (0, \frac{1}{8})$ and $\chi_\ee$ is defined in \eqref{Def_chiee}. 
By \eqref{def-norm-H-gamma}, it is straightforward to verify that $\| G_\ee \|_{\nu \otimes \Leb} \leq \| G \|_{\nu \otimes \Leb}$ 
and  $ \| G_\ee \|_{\scr H_{\gamma}} \leq \| G \|_{\scr H_{\gamma}}$. 
Given the assumption $F \leq_{\ee} G$, it follows that $F \mathds 1_{[0, \infty)} \leq_{\ee} G_{\ee}$. 
Consequently, applying the effecitive version of the local limit theorem (cf.\  \eqref{LLT_Upper_aa} of Theorem \ref{Lem_LLT_Nonasm}), 
we obtain that there exist constants $c, c_{\ee} >0$ such that, 
for any $m \geq 1$, $x \in \bb P(\bb V)$ and $t \in \bb R$, 
\begin{align} \label{JJJJJ-1111-001}
 J_{m}(x, t) 
& \leq H_m(t) +  \frac{ c\ee }{\sqrt{m}} \| G_\ee \|_{\nu \otimes \Leb}   
+ \frac{c_{\ee}}{m}  \| G_\ee  \|_{\scr H_{\gamma}}  \notag\\
& \leq H_m(t) +  \frac{ c\ee }{\sqrt{m}} \| G\|_{\nu \otimes \Leb}   
+ \frac{c_{\ee}}{m}  \| G \|_{\scr H_{\gamma}}, 
\end{align}
where for brevity we set, for $m \geq 1$ and $t \in \bb R$, 
\begin{align}\label{JJJ005}
H_m(t) = \int_{\bb P(\bb V)}  \int_{\bb R}    
  G_\ee (x', u)  \frac{1}{\upsilon_{\mu} \sqrt{m}} \phi \bigg( \frac{u - t}{\upsilon_{\mu} \sqrt{m}} \bigg)  du  \nu(dx'). 
\end{align}
From \eqref{JJJ-markov property}, \eqref{JJJJJ-1111-000} and \eqref{JJJJJ-1111-001}, it follows that 
\begin{align*}
& \Psi_m \Big( g_k \cdots g_1 x, t + \sigma (g_k \cdots g_1, x) \Big)
 \leq  J_m \Big( g_k \cdots g_1 x, t + \sigma (g_k \cdots g_1, x) \Big)  \notag\\
& \leq  H_m \Big( t + \sigma (g_k \cdots g_1, x) \Big) +  \frac{ c\ee }{\sqrt{m}} \| G\|_{\nu \otimes \Leb}   
+ \frac{c_{\ee}}{m}  \| G \|_{\scr H_{\gamma}}. 
\end{align*}
Substituting this into \eqref{JJJ-markov property} and using the bound \eqref{Upper-bound-tau-xt}, 
we derive that, there exist constants $c, c_{\ee} >0$ such that 
for any $n \geq n_0$, $x \in \bb P(\bb V)$ and $t \in \bb R$, 
\begin{align} \label{JJJ004}
& \Psi_n (x,t)
  \leq  \bb E \Big[   H_m \Big( t + \sigma (g_k \cdots g_1, x)  \Big);  \tau_{x, t} > k  \Big]  \notag \\  
& \qquad\quad +   \frac{c\ee (1 + \max\{t, 0\})}{\sqrt{km}}  \| G \|_{\nu \otimes \Leb}
     +  \frac{c_{\ee} (1 + \max\{t, 0\}) }{m \sqrt{k} }  \| G \|_{\scr H_{\gamma}}. 
\end{align}
Now we apply the conditioned central limit theorem with target functions (Corollary \ref{Cor-CCLT-target-function})
to handle the first term on the right-hand side of \eqref{JJJ004}. 
Denote $L_m(s) = H_m(\upsilon_{\mu} \sqrt{k} s)$ for $s \in \mathbb R$. 
By \eqref{JJJ005} and the change of variable $u = \upsilon_{\mu} \sqrt{k} u'$, we have 
\begin{align} \label{JJJ006}
L_m(s) 
& = \int_{\bb P(\bb V)}  \int_{\bb R}    
  G_\ee (x', u)  \frac{1}{\upsilon_{\mu} \sqrt{m}} 
   \phi \bigg( \frac{ \upsilon_{\mu} \sqrt{k} s - u}{\upsilon_{\mu} \sqrt{m}} \bigg)  du  \nu(dx')  \notag\\
& =  \int_{\bb P(\bb V)}  \int_{\bb R}  G_\ee \left(x', \upsilon_{\mu} \sqrt{k} u' \right)
 \frac{1}{ \sqrt{m/k}} \phi \bigg( \frac{s - u'}{ \sqrt{m/k}} \bigg)  du'  \nu(dx').
\end{align}
Since the function $s \mapsto L_m(s)$ is differentiable on $\bb R$ and vanishes as $s \to \pm \infty$, 
we are able to apply Corollary \ref{Cor-CCLT-target-function} to provide an upper bound
for the first term on the right-hand side of \eqref{JJJ004}. 
Specifically, there exists a constant $\eta>0$ with the following property.
For any $\beta>0$, there exists a constant $c_{\beta} >0$ such that, 
for any $n \geq n_0$, $x \in \bb P(\bb V)$ and $t \leq n^{1/2 - \beta}$, 
\begin{align} \label{ApplCondLT-001}
&  \bb E \Big[  H_m \Big( t + \sigma (g_k \cdots g_1, x) \Big);  \tau_{x, t} > k  \Big] 
 =  \bb E \bigg[ L_m \bigg( \frac{t + \sigma (g_k \cdots g_1, x)}{ \upsilon_{\mu} \sqrt{k}} \bigg);  \tau_{x, t} > k   \bigg]   \notag \\
& \leq  \frac{2 V(x, t)}{ \sqrt{2 \pi k} \upsilon_{\mu} }  \int_{\bb R_+}  L_m(s) \phi^+ (s) ds  
 +   \frac{c_{\beta} (1 + \max\{t, 0\})}{k^{1/2+\eta}} \int_{\bb R_+}  |L_m'(s)| ds. 
 \end{align}
where $\phi^+$ is the standard Rayleigh density function defined by \eqref{Rayleigh law-001}.
The last term on the right-hand side of \eqref{ApplCondLT-001} can be bounded as follows: 
by \eqref{JJJ006}, Fubini's theorem and the change of variables $u' = \frac{u}{\sqrt{m/k}}$ and $s' = \frac{s}{\sqrt{m/k}}$, 
there exist constants $c, c'>0$ such that for any $m \geq 1$, 
\begin{align} \label{JJJ-20001}
 \int_{\mathbb{R}_{+}}  | L'_m(s) | ds  
& =  \int_{\bb P(\bb V)}  \int_{\bb R_+} 
\bigg[ \int_{\mathbb R}   G_\ee \bigg(x', \upsilon_{\mu} \sqrt{m} \frac{u}{\sqrt{m/k}} \bigg) 
   \phi' \bigg( \frac{s - u}{\sqrt{m/k}} \bigg)
  \frac{du}{\sqrt{m/k}} \bigg] \frac{ds}{\sqrt{m/k}}  \nu(dx') \notag\\ 
& =   \int_{\bb P(\bb V)}  \int_{\bb R}  \bigg[ \int_{\bb R_+}
  G_\ee \left(x',  \upsilon_{\mu} \sqrt{m} u' \right)     
  \phi'\left( s' - u' \right)  ds'  \bigg]
  du' \nu(dx')  \notag\\ 
& \leq c  \int_{\bb P(\bb V)}  \int_{\bb R}  G_\ee \left(x',  \upsilon_{\mu} \sqrt{m} u' \right)  du' \nu(dx')  \notag\\
& = \frac{c}{\upsilon_{\mu} \sqrt{m}}  \| G_\ee \|_{\nu \otimes \Leb}
\leq \frac{c'}{\sqrt{m}}  \| G \|_{\nu \otimes \Leb}, 
\end{align}
where in the last inequality we used $\| G_\ee \|_{\nu \otimes \Leb} \leq \| G \|_{\nu \otimes \Leb}$. 
For the first term on the right-hand side of \eqref{ApplCondLT-001}, 
by the change of variable $s = \frac{s'}{\sqrt{k}}$ and the fact that $L_m(s) = H_m(\upsilon_{\mu} \sqrt{k} s)$, we get 
\begin{align}\label{def-I-mk-Lm}
I_{m,k} 
: = \frac{1}{\sqrt{k}}  \int_{\bb R_+}  L_m(s) \phi^+ (s) ds  
=  \int_{\mathbb{R}_{+}}  H_m(\upsilon_{\mu} s) \phi^+ \bigg( \frac{s}{\sqrt{k}} \bigg) \frac{ds}{k}.
\end{align}
Substituting \eqref{def-I-mk-Lm} and \eqref{JJJ-20001} into \eqref{ApplCondLT-001}, 
we obtain that there exist constants $c_{\beta}, \eta>0$ such that, 
for any $n \geq n_0$, $x \in \bb P(\bb V)$ and $t \leq n^{1/2 - \beta}$, 
\begin{align} \label{Integ_3}
 \bb E \Big[  H_m \Big( t + \sigma (g_k \cdots g_1, x) \Big);  \tau_{x, t} > k  \Big]  
 \leq \frac{2 V(x, t)}{\sqrt{2\pi } \upsilon_{\mu} }   I_{m,k}  
  +  \frac{c_{\beta} (1 + \max\{t, 0\})}{ k^{\eta} \sqrt{km} }  \| G \|_{\nu \otimes \Leb}.  
\end{align}
By implementing this bound into \eqref{JJJ004}, 
we find that there exist constants $c_{\beta}, c_{\ee}, \eta >0$ such that 
for any $n \geq n_0$, $x \in \bb P(\bb V)$ and $t \leq n^{1/2 - \beta}$, 
\begin{align}\label{JJJ201aaa}
\Psi_n (x,t) 
& \leq  \frac{2 V(x, t)}{\sqrt{2\pi } \upsilon_{\mu} }  I_{m,k}  
  +   \frac{ c_{\beta}(\ee + k^{-\eta}) (1 + \max\{t, 0\}) }{\sqrt{k m }}  \| G \|_{\nu \otimes \Leb}   \notag\\
 & \quad  +  \frac{c_{\ee} (1 + \max\{t, 0\}) }{m\sqrt{k} }  \| G \|_{\scr H_{\gamma}}, 
\end{align}
where $I_{m,k}$ is defined in \eqref{def-I-mk-Lm}. 
Next, we will derive an upper bound for $I_{m,k}$ by using the convolution technique along with Lemma \ref{t-Aux lemma}. 
By the definition of $H_m$ (cf.\ \eqref{JJJ005}) and Fubini's theorem, we get 
\begin{align*}
I_{m,k} 
& =  \int_{\mathbb{R}_{+}}  
\bigg[ \int_{\bb P(\bb V)}  \int_{\bb R}    
  G_\ee \left(x', u\right)  \frac{1}{\upsilon_{\mu} \sqrt{m}} 
   \phi \bigg( \frac{u - \upsilon_{\mu} s}{\upsilon_{\mu} \sqrt{m}} \bigg)  du \, \nu(dx') \bigg] 
  \phi^+ \bigg( \frac{s}{\sqrt{k}} \bigg) \frac{ds}{k} \notag\\
& =  \int_{\mathbb{R}_{+}} \bigg[ \int_{\bb P(\bb V)}  \int_{\bb R}  
  G_\ee \left(x', \upsilon_{\mu} u\right)  \phi_{\sqrt{m}} (u - s)    du \,  \nu(dx') 
   \bigg]  \phi^+ \bigg( \frac{s}{\sqrt{k}} \bigg)  \frac{ds}{k} \nonumber \\
& =  \int_{\bb P(\bb V)}  \int_{\mathbb R}  G_\ee (x', \upsilon_{\mu} u)
\bigg[ \int_{\mathbb{R}_{+}}  \phi_{\sqrt{m}} (u - s)
       \phi^+ \left( \frac{s}{\sqrt{k}} \right)  \frac{ds}{k} \bigg] du \,  \nu(dx'), 
\end{align*}
where $\phi_{\sqrt{m}}$ is the normal density with variance $\sqrt{m}$, defined by \eqref{def-normal-density-v}. 
Denote $\delta_n = \frac{m}{n} = \frac{[\delta n]}{n}$. 
By a change of variable, it follows that 
\begin{align}\label{Equality-Imk-abc}
I_{m,k} &= \frac{1}{\sqrt{n}} \int_{\bb P(\bb V)}  \int_{\mathbb R}  G_\ee (x', \upsilon_{\mu} \sqrt{n} u)
\bigg[ \int_{\mathbb{R}_{+}}
  \phi_{\delta_n} (u - s)
\phi^+_{1 - \delta_n} \left( s \right)  ds \bigg] du \,  \nu(dx')  \notag\\
& =  \frac{1}{\sqrt{n}}  \int_{\bb P(\bb V)}  \int_{\mathbb R}  G_\ee (x', \upsilon_{\mu} \sqrt{n} u)
\phi_{\delta_n}*\phi_{1-\delta_n}^+(u) du \,  \nu(dx')   \notag\\
& = \frac{ 1 }{\upsilon_{\mu} n} 
  \int_{\bb P(\bb V)}  \int_{\mathbb R}  G_\ee (x', u)
\phi_{\delta_n}*\phi_{1-\delta_n}^+ \bigg( \frac{u}{\upsilon_{\mu} \sqrt{n}} \bigg)  du \,  \nu(dx')  \notag\\  
& = \frac{ 1 }{\upsilon_{\mu} n} 
  \int_{\bb P(\bb V)}  \int_{-\ee}^\infty  G_\ee (x', u)
\phi_{\delta_n}*\phi_{1-\delta_n}^+ \bigg( \frac{u}{\upsilon_{\mu} \sqrt{n}} \bigg)  du \,  \nu(dx'),  
\end{align}
where in the last equality we used the fact that $G_{\ee}(x', u)=0$ for any $x' \in \bb P(\bb V)$ and $u\leq -\ee$,
since $G_{\ee} (x, u) = G (x, u) \chi_{\ee} (u - \ee)$.
Now, we handle the convolution $\phi_{\delta_n}*\phi_{1-\delta_n}^+$ 
using Lemma \ref{t-Aux lemma}  together with the fact that $\delta_n=\frac{m}{n}$, $1 - \delta_n = \frac{k}{n}$
and $u \geq -\ee$:
\begin{align*} 
 \phi_{\delta_n} * \phi^+_{1- \delta_n} \bigg( \frac{u}{\upsilon_{\mu} \sqrt{n}} \bigg)  
&\leq  \sqrt{1-\delta_n} \phi^+ \bigg( \frac{u}{\upsilon_{\mu} \sqrt{n}} \bigg) +  \sqrt{\delta_n}  e^{ -\frac{u^2}{2 \upsilon_{\mu}^2 n \delta_n} } 
+  \bigg| \frac{u}{\upsilon_{\mu} \sqrt{n}} \bigg| e^{- \frac{u^2}{2 \upsilon_{\mu}^2 n}} \mathds 1_{\{ u < 0\}} \\
&=  \sqrt{\frac{k}{n}} \phi^+ \bigg( \frac{u}{\upsilon_{\mu} \sqrt{n}} \bigg) +  \sqrt{\frac{m}{n}}  e^{ -\frac{u^2}{2 \upsilon_{\mu}^2 m} } 
 +  \bigg| \frac{u}{\upsilon_{\mu} \sqrt{n}} \bigg| e^{- \frac{u^2}{2 \upsilon_{\mu}^2 n}} \mathds 1_{\{ u < 0\}} \\ 
&\leq  \sqrt{\frac{k}{n}} \phi^+ \bigg( \frac{u}{\upsilon_{\mu} \sqrt{n}} \bigg) 
 +  \sqrt{\frac{m}{n}} + \frac{\ee}{\upsilon_{\mu} \sqrt{n}} \mathds 1_{\{ u < 0\}}. 
\end{align*}
Substituting this into \eqref{Equality-Imk-abc}, and using the fact that $G_{\ee} \leq G$ and $\phi^+(u) = 0$ for $u \leq 0$, 
we arrive at  
\begin{align}\label{final-bound-I-mk}
I_{m,k} &  \leq   \frac{ \sqrt{k} }{\upsilon_{\mu} n^{3/2}}   \int_{\bb P(\bb V)}  \int_{\mathbb R_+}  G(x', u)
\phi^+ \bigg( \frac{u}{\upsilon_{\mu} \sqrt{n}} \bigg) du \,  \nu(dx')   \notag\\
& \quad  +   \frac{\sqrt{m}}{ \upsilon_{\mu} n^{3/2}} 
  \int_{\bb P(\bb V)}  \int_{\mathbb R}  G(x', u)  du \,  \nu(dx')  \notag\\
& \quad   +    
  \frac{\ee}{ \upsilon_{\mu} n^{3/2}}  \int_{\bb P(\bb V)}  \int_{-\infty}^0  G(x', u)  du \nu(dx')   
  \notag\\
 & \leq   \frac{ \sqrt{k} }{\upsilon_{\mu} n^{3/2}}
  \int_{\bb P(\bb V)}  \int_{\mathbb R_+}  G(x', u) \phi^+ \bigg( \frac{u}{\upsilon_{\mu} \sqrt{n}} \bigg) du \nu(dx')
  +  \frac{c \ee^{1/4}}{ n}  \| G \|_{\nu \otimes \Leb},
\end{align}
where in the last inequality we used the fact that $\sqrt{\frac{m}{n}}  = \sqrt{ \frac{ [\ee^{1/2} n] }{ n } } \leq \ee^{1/4}$. 
Note that it is shown in \cite[Theorem 2.2]{GLP17} that,
there exists a constant $c>0$ such that $V(x, t) \leq c (1 + \max\{t, 0\})$ for any $x \in \bb P(\bb V)$ and $t \in \bb R$. 
Using this inequality and substituting the bound \eqref{final-bound-I-mk}
 into \eqref{JJJ201aaa}, we obtain that 
 there exist constants $c_{\beta}, c_{\ee}, \eta >0$ such that,
 for any $n \geq n_0$, $x \in \bb P(\bb V)$ and $t \in \bb R$,  
\begin{align*}
 \Psi_n (x,t) 
& \leq  \frac{2 V(x, t)}{\sqrt{2\pi} \upsilon_{\mu}^2 }  \frac{ \sqrt{k} }{n^{3/2}}
  \int_{\bb P(\bb V)}  \int_{\mathbb R_+}  G(x', u)  \phi^+ \bigg( \frac{u}{\upsilon_{\mu} \sqrt{n}} \bigg) du \nu(dx')   \notag\\
& \quad + c_{\beta} (1 + \max\{t, 0\}) \bigg(  \frac{ \ee + k^{-\eta}}{ \sqrt{km} }  + \frac{\ee^{1/4}}{n} \bigg)   \| G \|_{\nu \otimes \Leb}  
  +  \frac{c_{\ee} (1 + \max\{t, 0\})}{ m \sqrt{k} }  \| G \|_{\scr H_{\gamma}}  \notag\\
& \leq  \frac{2 V(x, t)}{\sqrt{2\pi} \upsilon_{\mu}^2 n}  
  \int_{\bb P(\bb V)}  \int_{\mathbb R_+}  G(x', u)  \phi^+ \bigg( \frac{u}{\upsilon_{\mu} \sqrt{n}} \bigg) du \nu(dx')   \notag\\
& \quad +  \frac{c_{\beta} (1 + \max\{t, 0\}) }{n}  \left( \ee^{1/4}  +   \ee^{-1/4} n^{-\eta} \right)    \| G \|_{\nu \otimes \Leb}  
    +  \frac{c_{\ee} (1 + \max\{t, 0\}) }{ n^{3/2} }  \| G \|_{\scr H_{\gamma}}, 
\end{align*}
where in the last inequality we used the fact that 
$\ee^{1/2} n \geq  m\geq \frac{1}{2}\ee^{1/2} n$ and $n > k\geq \frac{1}{2}n$. 
This finishes the proof of the upper bound \eqref{eqt-A 001}. 

\subsection{Proof of Corollary \ref{Cor-Cara-Bound}}
We can employ the same argument as in the proof of the upper bound \eqref{eqt-A 001},
so we only outline the main differences. 
Since $\phi$ is bounded, by the inequality $\| G_\ee \|_{\nu \otimes \Leb} \leq \| G \|_{\nu \otimes \Leb}$,
 the function $H_m$ defined in \eqref{JJJ005} is dominated by $\frac{ c}{\sqrt{m}} \| G \|_{\nu \otimes \Leb}$. 
Therefore, 
 Corollary \ref{Cor-Cara-Bound} is obtained from the inequalities \eqref{Upper-bound-tau-xt} and \eqref{JJJ004}, 
 by following the same steps as for the bound \eqref{eqt-A 001}. 


\subsection{Proof of Theorem \ref{t-A 001}: lower bound}
We now proceed to establish the second assertion \eqref{eqt-A 002} of Theorem \ref{t-A 001}. 
This proof is more intricate than that of the upper bound \eqref{eqt-A 001}
and will rely heavily on the inequalities established in Proposition \ref{Lem_Inequality_Aoverline}. 
We shall use the same notation as that in the proof of the upper bound. 
Recall that $\ee\in (0,\frac{1}{8})$, $\delta=\sqrt{\ee}$, $m=[\delta n]$ and $k=n-m$. 
It suffices to prove \eqref{eqt-A 002} only for sufficiently large $n > n_0 := n_0(\ee)$, 
where $n_0(\ee)$ depends on $\ee$; 
otherwise, the bound becomes trivial. 
As stated in \eqref{JJJ-markov property}, by the Markov property, we have that for any $x \in \bb P(\bb V)$ and $t \in \bb R$, 
\begin{align}\label{KKK-markov property} 
\Psi_n (x,t) = \bb E \Big[ \Psi_m \Big( g_k \cdots g_1 x, t + \sigma (g_k \cdots g_1, x) \Big);  \tau_{x, t} > k  \Big], 
\end{align}
where we have denoted, for $m \geq 1$, $x \in \bb P(\bb V)$ and $t \in \bb R$,
\begin{align*}
\Psi_m (x,t) = \bb E \Big[  F \Big( g_m \cdots g_1 x, t + \sigma (g_m \cdots g_1, x) \Big);  \tau_{x, t} > m   \Big]. 
\end{align*} 
To provide a lower bound for $\Psi_m \left( g_k \cdots g_1 x, t + \sigma (g_k \cdots g_1, x) \right)$ 
in \eqref{KKK-markov property}, our strategy is to 
decompose the function $\Psi_m$ into two terms as follows: 
for any $m \geq 1$, $x \in \bb P(\bb V)$ and $t \in \bb R$, 
\begin{align}\label{KKK-1111-000}
\Psi_m (x,t)  = A_{m} (x,t) - \overline A_{m} (x,t), 
\end{align}
where 
\begin{align}
A_{m} (x,t)   &  =  \bb E \Big[  F \Big( g_m \cdots g_1 x, t + \sigma (g_m \cdots g_1, x) \Big)   \Big], \label{decompos-psi001}  \\
\overline A_{m} (x,t)  
& =  \bb E \Big[  F \Big( g_m \cdots g_1 x, t + \sigma (g_m \cdots g_1, x) \Big);  \tau_{x, t} \leq m   \Big]. \label{decompos-psi002}
\end{align}
Substituting \eqref{KKK-1111-000} into \eqref{KKK-markov property} implies that,
for any $m \geq 1$, $x \in \bb P(\bb V)$ and $t \in \bb R$, 
\begin{align} \label{Psi_nee_Decompo}
 \Psi_n(x,t)= J_n(x,t)-K_n(x,t), 
\end{align}
where  
\begin{align}
J_n(x,t)& := \bb E \Big[  A_{m} \Big( g_k \cdots g_1 x, t + \sigma (g_k \cdots g_1, x) \Big);  \tau_{x, t} > k  \Big], 
 \label{Psi_nee_Decompo001} \\
K_n(x,t)&:=  \bb E \Big[ \overline A_{m} \Big( g_k \cdots g_1 x, t + \sigma (g_k \cdots g_1, x) \Big); \tau_{x, t} > k  \Big].  \label{Psi_nee_Decompo002}
\end{align}
Below, we will demonstrate that $J_n(x,t)$ is the dominant term, while $K_n(x,t)$ is negligible in comparison. 

We proceed to utilize the effective local limit theorem (Theorem \ref{Lem_LLT_Nonasm}) 
in conjunction with the conditioned central limit theorem (Corollary \ref{Cor-CCLT-target-function}) 
to give a lower bound for the first term $J_n(x,t)$ in \eqref{Psi_nee_Decompo}.
While this can be approached similarly to the case of the upper bound, 
the situation here is more complex. According to the local limit theorem 
(see the lower bound \eqref{LLT_Lower_aa} of Theorem \ref{Lem_LLT_Nonasm}), 
there exist constants $c, c_{\ee} >0$ such that 
for any  $m \geq 1$, $x \in \bb P(\bb V)$ and $t \in \bb R$, 
and any non-negative measurable functions $F, G, H: \bb P(\bb V) \times \bb R \to \bb R$ satisfying $H \leq_{\ee} F \leq_{\ee} G$, 
\begin{align} \label{JJJJJ-1111-001B}
 A_{m}(x, t) 
 \geq  B_m(t)
     -   \frac{c \ee}{\sqrt{m}}  \| G \|_{\nu \otimes \Leb}   
       -  \frac{c_{\ee}}{ m }   \left(  \| G \|_{\scr H_{\gamma}}   +  \| H \|_{\scr H_{\gamma}} \right), 
\end{align}
where for brevity we set, for $m \geq 1$ and $t \in \bb R$, 
\begin{align}\label{JJJ005B}
B_m(t) = \int_{\bb P(\bb V)}  \int_{\bb R} H(x', u)
 \frac{1}{\upsilon_{\mu} \sqrt{m}} \phi \bigg( \frac{u - t}{\upsilon_{\mu} \sqrt{m}} \bigg)  
    du \,  \nu(dx').
\end{align}
Substituting \eqref{JJJJJ-1111-001B} into \eqref{Psi_nee_Decompo001} and using the bound \eqref{Upper-bound-tau-xt},  
we derive that there exist constants $c, c_{\ee} >0$ such that 
for any  $n \geq n_0$, $x \in \bb P(\bb V)$ and $t \in \bb R$, 
\begin{align} \label{JJJ004B}
  J_n (x,t)
& \geq   \bb E \Big[ B_{m} \Big( t + \sigma (g_k \cdots g_1, x) \Big); \tau_{x, t} > k  \Big]  \notag \\  
& \quad  -  \frac{c \ee (1 + \max\{t, 0\}) }{\sqrt{km}}   \| G \|_{\nu \otimes \Leb}  
   -  \frac{c_{\ee} (1 + \max\{t, 0\}) }{\sqrt{k} m }  \left(  \| G \|_{\scr H_{\gamma}}   +  \| H \|_{\scr H_{\gamma}} \right).  
\end{align}
For the first term on the right-hand side of \eqref{JJJ004B}, 
by following the same approach as in the proof of \eqref{Integ_3} and utilizing the lower bound in the conditioned central limit theorem
(Corollary \ref{Cor-CCLT-target-function}), 
one can verify that, there exist constants $c_{\beta}, \eta>0$ such that, 
for any $n \geq n_0$, $x \in \bb P(\bb V)$ and $t \leq n^{1/2 - \beta}$, 
\begin{align} \label{Integ_3B}
&  \bb E \Big[ B_{m} \Big( t + \sigma (g_k \cdots g_1, x) \Big); \tau_{x, t} > k  \Big]  \notag\\
& \geq \frac{2 V(x, t)}{\sqrt{2\pi } \upsilon_{\mu} } 
\int_{\mathbb{R}_{+}}  B_m(\upsilon_{\mu} s) \phi^+ \bigg( \frac{s}{ \sqrt{k}} \bigg) \frac{ds}{k}  
   -  \frac{c_{\beta} (1 + \max\{t, 0\})}{ k^{\eta}\sqrt{km} } \| H \|_{\nu \otimes \Leb}.
\end{align}
where $B_m$ is defined by \eqref{JJJ005B}. 
Substituting \eqref{Integ_3B} into \eqref{JJJ004B} 
and using the fact that $\| H \|_{\nu \otimes \Leb} \leq \| G \|_{\nu \otimes \Leb}$,
we get that there exist constants $c_{\beta}, c_{\ee}, \eta>0$ such that, 
for any $n \geq n_0$, $x \in \bb P(\bb V)$ and $t \leq n^{1/2 - \beta}$, 
\begin{align}
\label{JJJ201aaaB}
  J_n (x,t) 
&  \geq  \frac{2 V(x, t)}{\sqrt{2\pi} \upsilon_{\mu}}  
\int_{\mathbb{R}_{+}}  B_m(\upsilon_{\mu} s) \phi^+ \bigg( \frac{s}{\sqrt{k}} \bigg) \frac{ds}{k}   
 -  \frac{c_{\beta} (\ee + k^{-\eta}) (1 + \max\{t, 0\}) }{\sqrt{km}}   \| G \|_{\nu \otimes \Leb}   \notag\\
& \quad 
     -  \frac{c_{\ee} (1 + \max\{t, 0\}) }{\sqrt{k} m }  \left(  \| G \|_{\scr H_{\gamma}}   +  \| H \|_{\scr H_{\gamma}} \right).  
\end{align}
In a manner similar to the proof of \eqref{Equality-Imk-abc}, we have
\begin{align*}
 \int_{\mathbb{R}_{+}}  B_m(\upsilon_{\mu} s) \phi^+ \bigg( \frac{s}{\sqrt{k}} \bigg) \frac{ds}{k}  
& =   \frac{ 1 }{\upsilon_{\mu} n} 
  \int_{\bb P(\bb V)}  \int_{\mathbb R}  H(x', u)
 \phi_{\delta_n}*\phi_{1-\delta_n}^+ \bigg( \frac{u}{\upsilon_{\mu} \sqrt{n}} \bigg)  du \,  \nu(dx')  \notag\\
& \geq  \frac{ 1 }{\upsilon_{\mu} n}
  \int_{\bb P(\bb V)}  \int_{\bb R_+} H(x', u)
 \phi_{\delta_n}*\phi_{1-\delta_n}^+ \bigg( \frac{u}{\upsilon_{\mu} \sqrt{n}} \bigg)  du \,  \nu(dx'),  
\end{align*}
where $\delta_n = \frac{m}{n} = \frac{[\delta n]}{n}$. 
Applying the lower bound in Lemma \ref{t-Aux lemma}, we get that 
$\phi_{\delta_n}*\phi_{1-\delta_n}^+ (t) 
\geq \sqrt{1 - \delta_n} \phi^+ (t) = \sqrt{\frac{k}{n}} \phi^+ (t)$
for any $t \geq 0$. 
It follows that 
\begin{align*}
\int_{\mathbb{R}_{+}}  B_m(\upsilon_{\mu} s) \phi^+ \bigg( \frac{s}{\sqrt{k}} \bigg) \frac{ds}{k} 
 \geq  \frac{ \sqrt{k} }{\upsilon_{\mu} n^{3/2}}
  \int_{\bb P(\bb V)}  \int_{\mathbb R_+}  H(x', u)
\phi^+ \bigg( \frac{u}{\upsilon_{\mu} \sqrt{n}} \bigg) du \, \nu(dx'). 
\end{align*}
Substituting this into \eqref{JJJ201aaaB} and using the fact that 
$V(x, t) \leq c (1 + \max\{t, 0\})$ for any $x \in \bb P(\bb V)$ and $t \in \bb R$ (\cite[Theorem 2.2]{GLP17}), 
we get that there exist constants $c_{\beta}, c_{\ee}, \eta>0$ such that, 
for any $n \geq n_0$, $x \in \bb P(\bb V)$ and $t \leq n^{1/2 - \beta}$, 
\begin{align*}
  J_n (x,t) 
& \geq  \frac{2 V(x, t)}{\sqrt{2\pi} \upsilon_{\mu}^2 }  \frac{ \sqrt{k} }{n^{3/2}}
  \int_{\bb P(\bb V)}  \int_{\mathbb R_+}  H(x', u)
\phi^+ \bigg( \frac{u}{\upsilon_{\mu} \sqrt{n}} \bigg) du \,  \nu(dx')    \notag\\
& \quad  -  \frac{ c_{\beta} (\ee + k^{-\eta}) (1 + \max\{t, 0\}) }{\sqrt{km}}   \| G \|_{\nu \otimes \Leb}  \notag\\
 &  \quad   -  \frac{c_{\ee} (1 + \max\{t, 0\}) }{\sqrt{k} m }  \left(  \| G \|_{\scr H_{\gamma}}   +  \| H \|_{\scr H_{\gamma}} \right). 
\end{align*}
Since  $\sqrt{\frac{n}{k}}\leq 1+ c \ee^{1/4}$, $m\geq \frac{1}{2} \ee^{1/2} n$, $k\geq \frac{1}{2}n$ and $H \leq_{\ee} G$, 
by the fact that 
$V(x, t) \leq c (1 + \max\{t, 0\})$ for any $x \in \bb P(\bb V)$ and $t \in \bb R$ (\cite[Theorem 2.2]{GLP17}), 
there exist constants $c_{\beta}, c_{\ee}, \eta>0$ such that, 
for any $n \geq n_0$, $x \in \bb P(\bb V)$ and $t \leq n^{1/2 - \beta}$, 
\begin{align} \label{final bounf Jm(x)}
J_n (x,t)
& \geq  \frac{2 V(x, t)}{\sqrt{2\pi} \upsilon_{\mu}^2 n} 
  \int_{\bb P(\bb V)}  \int_{\mathbb R_+}  H(x', u)
\phi^+ \bigg( \frac{u}{\upsilon_{\mu} \sqrt{n}} \bigg) du \nu(dx')  \notag\\
&  \quad  -  \frac{c_{\beta} (\ee^{1/4} + n^{-\eta}) (1 + \max\{t, 0\}) }{n}    \| G \|_{\nu \otimes \Leb}   \notag\\
 & \quad  -  \frac{c_{\ee} (1 + \max\{t, 0\}) }{ n^{3/2} }  \left(  \| G \|_{\scr H_{\gamma}}   +  \| H \|_{\scr H_{\gamma}} \right), 
\end{align}
which concludes the proof of the lower bound for the first term $J_n (x,t)$ defined in \eqref{Psi_nee_Decompo001}.

Next, we shall provide an upper bound for the second term $K_n(x,t)$ defined by \eqref{Psi_nee_Decompo002}. 
Indeed, 
bounding $K_n(x,t)$ is one of the difficult parts of the paper 
and requires the application of Lemma \ref{FK-joint-inequality}, Proposition \ref{Lem_Inequality_Aoverline}
and Theorems \ref{Lem_LLT_Nonasm} and \ref{Thm-CCLT-001}. 
We start by splitting $K_n(x,t)$ into two parts 
according to whether the values of $t + \sigma (g_k \cdots g_1, x)$ 
are less than or greater than $\ee^{1/6} \upsilon_{\mu} \sqrt{n}$: for $x \in \bb P(\bb V)$ and $t \in \bb R$, 
\begin{align}\label{KKK-111-001}
K_n(x,t) = K_{n,1}(x,t) + K_{n,2}(x,t), 
\end{align}
where
\begin{align*} 
K_{n,1}(x,t)
& = \bb E \Big[  \overline A_{m} \Big( g_k \cdots g_1 x, t + \sigma (g_k \cdots g_1, x) \Big);   
  t + \sigma (g_k \cdots g_1, x) \leq \ee^{1/6} \upsilon_{\mu} \sqrt{n},  \tau_{x, t} > k   \Big],   \notag\\
K_{n,2}(x,t) 
& = \bb E \Big[  \overline A_{m} \Big( g_k \cdots g_1 x, t + \sigma (g_k \cdots g_1, x) \Big);    t + \sigma (g_k \cdots g_1, x) > \ee^{1/6} \upsilon_{\mu} \sqrt{n},  \tau_{x, t} > k   \Big]. 
\end{align*}
For $K_{n,1}(x,t)$,  applying the upper bound in the local limit theorem 
(cf.\ Theorem \ref{Lem_LLT_Nonasm}) to the function $\overline A_{m} (x, t)$ defined in \eqref{decompos-psi002},
 and taking into account that $\phi \leq 1$,  we get
\begin{align*}
\overline A_{m} (x, t) 
\leq \bb E \Big[  F \Big( g_m \cdots g_1 x, t + \sigma (g_m \cdots g_1, x) \Big)   \Big] 
\leq  \frac{L_m(\ee)}{\sqrt{m}}, 
\end{align*}
where 
\begin{align*}
L_m(\ee) =  c \| G \|_{\nu \otimes \Leb}   +  \frac{c_{\ee}}{ \sqrt{m} }  \| G \|_{\scr H_{\gamma}}. 
\end{align*}
Substituting this into $K_{n,1}(x,t)$ and using the fact that $\sqrt{\frac{n}{k}} \leq 2$, we get 
\begin{align*} 
K_{n,1}(x,t)  & \leq  \frac{L_m(\ee)}{\sqrt{m}} 
    \bb P \left(  t + \sigma (g_k \cdots g_1, x)  \leq \ee^{1/6} \upsilon_{\mu} \sqrt{n},   \tau_{x, t} > k \right)    \notag\\
& \leq  \frac{L_m(\ee)}{\sqrt{m}} \,  \bb P \bigg(  \frac{ t + \sigma (g_k \cdots g_1, x) }{\upsilon_{\mu} \sqrt{k}} 
 \leq 2 \ee^{1/6}, \tau_{x, t} > k \bigg). 
\end{align*}
By Theorem \ref{Thm-CCLT-001} and the fact that $m=[\ee^{1/2}n]$
and $V(x, t) \leq c (1 + \max\{t, 0\})$ for any $x \in \bb P(\bb V)$ and $t \in \bb R$, 
there exists a constant $\eta >0$ such that the following holds. 
For any $\beta>0$, 
there exists a constant $c_{\beta} >0$ such that for any $n \geq n_0$, $x \in \bb P(\bb V)$ and $t \leq n^{1/2 - \beta}$, 
\begin{align} \label{K1-final bound}
K_{n,1}(x,t)
& \leq  \frac{L_m(\ee)}{\sqrt{m}}  
\bigg( \frac{2 V(x, t)}{ \sqrt{2\pi k} \upsilon_{\mu}} 
  \int_{0}^{ 2 \ee^{1/6}} \phi^+ (t') dt' 
  +  \frac{c_{\beta} (1 + \max\{t, 0\}) }{k^{1/2+\eta}} \bigg) \notag \\
&\leq  \frac{c(1 + \max\{t, 0\})}{\ee^{\frac{1}{4}} n}  
 \Big(  \| G \|_{\nu \otimes \Leb}     +  \frac{c_{\ee}}{ \sqrt{m} }  \| G \|_{\scr H_{\gamma}} \Big)
 \Big(\ee^{1/3}  +  c_{\beta} n^{-\eta} \Big)  \notag \\
& \leq \frac{c (1 + \max\{t, 0\})}{n}  \Big(  \| G \|_{\nu \otimes \Leb}   +  \frac{c_{\ee}}{ \sqrt{m} }  \| G \|_{\scr H_{\gamma}} \Big)  
  \Big(c\ee^{\frac{1}{12}}  +  c_{\beta}  \ee^{-\frac{1}{4}}  n^{-\eta} \Big) \notag\\
&\leq  \frac{c (1 + \max\{t, 0\})}{n} \Big( \ee^{\frac{1}{12}} +  c_{\beta} \ee^{-\frac{1}{4}}  n^{-\eta}  \Big) \| G \|_{\nu \otimes \Leb}  \notag\\
& \quad +   \frac{c_{\beta, \ee} (1 + \max\{t, 0\})}{ n^{3/2} }  \| G \|_{\scr H_{\gamma}}.
\end{align}

We proceed to apply the upper bound \eqref{eqt-A 001}
to provide an estimate for $K_{n,2}(x,t)$, as defined in \eqref{KKK-111-001}.
Recall that the function $(x,t) \mapsto \overline A_{m} (x,t)$, which is involved in the definition of $K_{n,2}(x,t)$,
is defined by \eqref{decompos-psi002} 
and does not in general belong to the space $\scr H_{\gamma}$.  
To address this, we begin by smoothing the indicator function in 
\eqref{decompos-psi002}. 
Let $\kappa$ be a smooth and non-negative  function compactly supported in $[-1,1]$ such that
$\int_{-1}^1 \kappa(t)dt=1$ 
and set $\kappa_{\ee}(t) = \frac{1}{\ee} \kappa(\frac{t}{\ee})$ for $t \in \bb R$.
For $m \geq 1$, $x \in \bb P(\bb V)$ and $t \in \bb R$, 
we define
\begin{align*}
\overline A_{m, \ee}  (x,t) & :=   \bb E \Big[   
     G *  \kappa_{\ee/2} \Big( g_m \cdots g_1 x, t + \sigma (g_m \cdots g_1, x) \Big)  
       \overline\chi_{\ee}  \Big( t- 2 \ee + \min_{1 \leq  j \leq m}  \sigma (g_j \cdots g_1, x) \Big)  \Big],
\end{align*}
where $\chi_{\ee}$ is defined as in \eqref{Def_chiee} and $\overline\chi_{\ee} = 1 - \chi_{\ee}$. 
Note that the function $F$ is $\ee/2$-dominated by the function $G * \kappa_{\ee/2}$. 
Using the identity
\begin{align*} 
 \left\{ \tau_{x, t} > m \right\} 
=  \left\{ t + \min_{1 \leq j \leq m}  \sigma (g_j \cdots g_1, x) \geq 0 \right\},   
\end{align*}  
along with the inequalities \eqref{bounds-reversedindicators-001} and $F(x, \cdot) \leq  G * \kappa_{\ee/2} (x, \cdot)$,  
 we conclude that
 the function $\overline A_{m}$ defined by \eqref{decompos-psi002}  is $\ee/2$-dominated by 
 $\overline A_{m, \ee}$. 
Moreover, by Proposition \ref{Lem_Inequality_Aoverline}, 
there exists a constant $c_{\ee} >0$ such that for any $m \geq 1$, 
the function $\overline A_{m, \ee}$ belongs to the space $\scr H_{\gamma}$ and satisfies the following inequalities: 
\begin{align}\label{Inequalities-Am-ee}
\| \overline A_{m, \ee} \|_{ \scr H_{\gamma} }  \leq c_{\ee}  \| G\|_{ \scr H_{\gamma} },
\qquad  
\| \overline A_{m, \ee} \|_{\nu \otimes \Leb }  \leq   \| G \|_{\nu \otimes \Leb }. 
\end{align}
For $m \geq 1$, $x \in \bb P(\bb V)$ and $t \in \bb R$, we denote 
\begin{align}\label{Def_m_zt}
W_{m,\ee} (x, t) = \overline A_{m, \ee} \left( x, t  \right)   \mathds 1_{ \{ t \geq \frac{1}{2} \ee^{1/6} \upsilon_{\mu} \sqrt{n} \} }. 
\end{align}
Applying the upper bound \eqref{eqt-A 001}, and using the fact that $\phi^+ \leq 1$, $k \geq \frac{1}{2}n$
and 
$V(x, t) \leq c (1 + \max\{t, 0\})$ for any $x \in \bb P(\bb V)$ and $t \in \bb R$,  
we obtain that for any $\ee \in (0,\frac{1}{8})$ and $\beta >0$, 
there exist constants $c_{\beta}, c_{\ee} > 0$ such that for any $n \geq n_0$, $x \in \bb P(\bb V)$ and $t \leq n^{1/2 - \beta}$,  
\begin{align}\label{eqt-A 001_low_bb}
 K_{n,2}(x,t) 
& \leq  \bigg(\frac{2 V(x, t)}{\sqrt{2\pi} \upsilon_{\mu}^2 k}  
  + \frac{c_{\beta} (1 + \max\{t, 0\})}{k} \Big( \ee^{1/4}  +   \ee^{-1/4}n^{-\eta} \Big) \bigg)
   \| W_{m,\ee} \|_{\nu \otimes \Leb }  \notag\\
& \quad  +   \frac{c_{\ee} (1 + \max\{t, 0\})}{ \sqrt{n} k}  \|  W_{m,\ee} \|_{\scr H_{\gamma}} \notag\\
& \leq  \frac{c_{\beta} (1 + \max\{t, 0\})}{n} \left( 1  +   \ee^{-1/4}n^{-\eta} \right) 
   \| W_{m,\ee} \|_{\nu \otimes \Leb }  \notag\\
& \quad   +   \frac{c_{\ee} (1 + \max\{t, 0\})}{ n^{3/2} }  \|  W_{m,\ee} \|_{\scr H_{\gamma}}. 
\end{align}
For the first term on the right-hand side of \eqref{eqt-A 001_low_bb}, 
by the definition of $W_{m,\ee}$ and Fubini's theorem, we have
\begin{align} \label{first norm-001}
 \| W_{m,\ee} \|_{\nu \otimes \Leb } 
 &=   \int_{\mathbb R}  \int_{\bb P(\bb V)}  W_{m,\ee} \left(x', t \right)  \nu(dx') dt    \notag\\ 
&\leq  \int_{\bb P(\bb V)}  \int_{ \frac{1}{2} \ee^{1/6} \upsilon_{\mu} \sqrt{n}}^{\infty} 
 \bb E  \bigg[   G *  \kappa_{\ee/2} \Big(  g_m \cdots g_1 x', t + \sigma (g_m \cdots g_1, x') \Big);  \notag\\
& \qquad   t + \min_{1 \leq  j \leq m}  \sigma (g_j \cdots g_1, x')  \leq 2 \ee   \bigg]   dt  \, \nu(dx')  =: U_{m,n}. 
\end{align}
Using the change of variable $t' = t + \sigma (g_m \cdots g_1, x')$, we get
\begin{align*}
U_{m,n}  & =  \int_{\bb P(\bb V)}  \int_{\mathbb R}  
 \bb E  \bigg[   G *  \kappa_{\ee/2} \left(  g_m \cdots g_1 x',  t'  \right);  \notag\\
& \qquad\qquad\qquad  t' - \sigma (g_m \cdots g_1, x') +  \min_{1 \leq j \leq m }  \sigma (g_j \cdots g_1, x')   \leq 2 \ee,  \notag\\ 
& \qquad\qquad\qquad   t' - \sigma (g_m \cdots g_1, x') \geq  \frac{1}{2} \ee^{1/6} \upsilon_{\mu} \sqrt{n}  \bigg] dt'  \nu(dx')   \notag\\
& \leq \int_{\bb P(\bb V)}  \int_{\mathbb R}  
 \bb E  \bigg[   G *  \kappa_{\ee/2} \left(  g_m \cdots g_1 x',  t'  \right);  \notag\\
& \qquad\qquad\qquad \frac{1}{2} \ee^{1/6} \upsilon_{\mu} \sqrt{n} +  \min_{1 \leq j \leq m }  \sigma (g_j \cdots g_1, x')   \leq  2 \ee  \bigg] dt'  \nu(dx')   \notag\\
 & \leq  \int_{\bb P(\bb V)}  \int_{\mathbb R}  
 \bb E  \bigg[   G *  \kappa_{\ee/2} \left(  g_m \cdots g_1 x',  t'  \right);   \notag\\
 & \qquad\qquad\qquad   \max_{1 \leq j \leq m }  |\sigma (g_j \cdots g_1, x')|   \geq  \frac{1}{3} \ee^{1/6}  \upsilon_{\mu} \sqrt{n}  \bigg] dt'  \nu(dx'),
\end{align*}
where in the last inequality we used 
$\frac{1}{2} \ee^{1/6} \upsilon_{\mu} \sqrt{n} - 2 \ee \geq \frac{1}{3} \ee^{1/6}  \upsilon_{\mu} \sqrt{n} $. 
Applying Lemma \ref{FK-joint-inequality}, 
we derive that there exist  constants $c, c'>0$ such that, for any $m \geq 1$ and $x' \in \bb P(\bb V)$, 
\begin{align*}
& \int_{\mathbb R} \bb E  \bigg[   G *  \kappa_{\ee/2} \left(  g_m \cdots g_1 x',  t'  \right);  
    \max_{1 \leq j \leq m }  |\sigma (g_j \cdots g_1, x')|   \geq  \frac{1}{3} \ee^{1/6}  \upsilon_{\mu} \sqrt{n}  \bigg]  dt'  \notag\\
 &  \leq 2   \exp\left( - c \ee^{-1/6} \right) \| G *  \kappa_{\ee/2}  \|_{\nu\otimes \Leb}  + c' e^{-c \ee^{1/12} n^{1/6}} \| G *  \kappa_{\ee/2} \|_{\scr H_{\gamma}} \notag\\
 &  \leq 2   \exp\left( - c \ee^{-1/6} \right) \| G \|_{\nu\otimes \Leb}  + c' e^{-c \ee^{1/12} n^{1/6}} \| G \|_{\scr H_{\gamma}}. 
\end{align*}
where the last inequality holds since $\kappa$ is a smooth and non-negative function with support in $[-1,1]$, 
and  $\kappa_{\ee}(t) = \frac{1}{\ee} \kappa(\frac{t}{\ee})$.  
Consequently, there exist constants $c, c'>0$ such that, for any $m \geq 1$, 
\begin{align}\label{Bound_n_U-002}
\| W_{m,\ee} \|_{\nu \otimes \Leb }  \leq 2   \exp\left( - c \ee^{-1/6} \right) \| G \|_{\nu\otimes \Leb}  + c' e^{-c \ee^{1/12} n^{1/6}} \| G \|_{\scr H_{\gamma}}. 
\end{align}

 The norm $\|  W_{m,\ee} \|_{\scr H_{\gamma}}$ 
in the second term on the right-hand side of \eqref{eqt-A 001_low_bb}
is dominated by applying Proposition \ref{Lem_Inequality_Aoverline}. Taking into account \eqref{Def_m_zt} and \eqref{Inequalities-Am-ee},  we get
\begin{align*}
\|  W_{m,\ee} \|_{\scr H_{\gamma}}  
\leq \| \overline A_{m, \ee} \|_{ \scr H_{\gamma} }
\leq  c_\ee \|  G \|_{\scr H_{\gamma}}. 
\end{align*}
Therefore, substituting this and \eqref{Bound_n_U-002} into \eqref{eqt-A 001_low_bb},  
we derive the upper bound for $K_{n,2}(x,t)$: 
there exist constants $c_{\beta}, c_{\ee} > 0$ such that for any $n \geq n_0$, $x \in \bb P(\bb V)$ and $t \leq n^{1/2 - \beta}$,  
\begin{align}\label{Final_Bound_K2}
K_{n,2}(x,t)
& \leq c_{\beta} \frac{ (1 + \max\{t, 0\})}{n} \left( 1  +   \ee^{-1/4}n^{-\eta} \right) 
 \exp\left(- c \ee^{-1/6} \right)    \| G  \|_{\nu\otimes \Leb}   
+   \frac{c_{\ee}}{ n^{3/2} }  \|  G \|_{\scr H_{\gamma}}
 \notag \\  
& \leq c_{\beta} \frac{  \ee (1 + \max\{t, 0\})}{n} 
\| G  \|_{\nu\otimes \Leb}  +   \frac{c_{\ee}}{ n^{3/2} }  \|  G \|_{\scr H_{\gamma}}.  
\end{align}
Combining \eqref{Psi_nee_Decompo}, \eqref{final bounf Jm(x)}, \eqref{KKK-111-001}, \eqref{K1-final bound} and \eqref{Final_Bound_K2} completes the proof of the lower bound \eqref{eqt-A 002}.

\subsection{Consequences}
\begin{proof}[Proof of Theorem \ref{t-B 001}]
This follows from the more precise Theorem \ref{t-A 001}. 
\end{proof}

\begin{proof}[Proof of Corollary \ref{Cor-bound-n-001}]
This follows directly by applying Corollary \ref{Cor-Cara-Bound} to the indicator function $F = \mathds 1_{\bb P(\bb V) \times [a, b] }$
and $G = \mathds 1_{\bb P(\bb V) \times [a- \ee, b+ \ee] }$, where $a < b$ and $\ee \in (0, \frac{1}{8})$. 
\end{proof}


\section{The reversed random walk} \label{sec: reversed random walk}

In this section, we aim to establish two indispensable results which are announced in Section \ref{subsec-statement results}, 
namely Theorems \ref{Cor-CLLT-cocycle-001} and \ref{Thm-CLLT-cocycle-bound-001}.
To accomplish this, we will employ the reversal techniques that were developed in the preliminary paper \cite{GQX24a}. 
These techniques provide a powerful framework for analyzing the underlying reversed random walk that we encounter,  
the details of which will be presented in Subsections \ref{Sec-duality} and \ref{sec-reversed random walk-001}. 
Additionally, we will make use of the conditioned central limit theorems for the reversed walk 
(cf.\ Theorems \ref{Thm-CCLT-limit-000} and \ref{Thm-inte-exit-time-001}). 
The proofs of these theorems are technically complex and can be found in Appendix \ref{sec invar func}. 

\subsection{Duality}\label{Sec-duality}
The aforementioned reversal techniques rely on the use of the linear duality of group actions. In the present subsection,  
we introduce the dual cocycle
and formulate the cohomological equation which relates the dual cocycle to the original norm cocycle.

Let $\Omega = \bb G^{\bb N^*}$ be equipped with the Borel $\sigma$-algebra $\scr A$ and with the  
probability measure $\bb P = \mu^{\otimes \bb N^*}$. 
A typical element of $\Omega$ is written as $\omega = (g_1, g_2, \ldots)$.
The sequence of coordinate maps, $\omega \mapsto g_k$, $k=1,2,\ldots$ on the probability space 
$(\Omega,\mathscr A, \bb P)$ forms a sequence of independent and identically distributed elements of $\bb G$ with law $\mu$.
Let $T: \Omega \to \Omega$ be the shift map defined by $\omega = (g_1, g_2, \ldots) \mapsto T \omega = (g_2, g_3, \ldots)$.
We also define the shift map $\widetilde T $ on $\Omega\times \bb P(\bb V^*)$ as follows: 
for $\omega=(g_1, g_2, \ldots) \in \Omega$ and $y \in \bb P(\bb V^*)$, 
\begin{align} \label{def-tilde-T}
\widetilde T(\omega, y)= ( (g_2, g_3, \ldots), g^{-1}_{1} y).
\end{align}
Recall that $\Gamma_{\mu}$ is the closed subsemigroup of $\bb G$ spanned by the support of $\mu$. 
Since $\Gamma_{\mu}$ contains some proximal element and the action of $\Gamma_{\mu}$ on $\bb V$ is strongly irreducible,   
the space $\bb P(\bb V)$ carries a unique $\mu$-stationary probability measure $\nu$. 
By a classical result of Furstenberg, 
there exists a unique measurable map $\xi: \Omega \to \bb P(\bb V)$ satisfying the following equivariance property:  
for $\bb P$-almost every $\omega = ( g_1, g_2, \ldots) \in \Omega$, 
\begin{align}\label{equivariance-xi}
\xi(\omega) = g_1 \xi(T\omega),  
\end{align}
and hence, by iteration, for any $p\geq 1$,
\begin{align}\label{equivariance-xi-002}
\xi(\omega) = g_1\cdots g_p \xi(T^p \omega). 
\end{align}
Moreover, the law of the random point $\xi$ in $\bb P(\bb V)$ is the stationary measure $\nu$.
We consider the dual vector space $\bb V^*$ of $\bb V$ and denote by $\bb P(\bb V^*)$ the projective space of $\bb V^*$.
The group $\bb G$ acts on $\bb V^*$ and $\bb P(\bb V^*)$ as follows:   
for any $g \in \bb G$, $v \in \bb V$ and $\varphi \in \bb V^*$, 
\begin{align*}
g \varphi (v) = \varphi (g^{-1} v). 
\end{align*}
Recall that we have chosen a Euclidean norm $\| \cdot \|$ on $V$.  
The space $\bb V^*$ is equipped  with the Euclidean norm  which is dual to the norm $\|\cdot\|$ of  $\bb V$:
for $\varphi \in \bb V^*$, 
\begin{align*}
\| \varphi \| = \sup_{v \in \bb V \smallsetminus \{0\}} \frac{| \varphi(v) |}{\|v\|}. 
\end{align*}
We introduce a dual cocycle $\sigma^*: \bb G \times \bb P(\bb V^*) \to \bb R$ by setting, 
for any $g \in \bb G$ and $y = \bb R \varphi \in \bb P(\bb V^*)$, 
\begin{align}\label{dual cocycle-001}
\sigma^*(g, y) = \log \frac{\| g \varphi \|}{\| \varphi \|}. 
\end{align}
Consider the set 
\begin{align*} 
\Delta: = \Big\{ (x, y) \in \bb P(\bb V) \times \bb P(\bb V^*):  x = \bb R v, y = \bb R \varphi, \varphi(v) \neq 0  \Big\}. 
\end{align*}
For any $(x, y) \in \Delta$ 
with $x = \bb R v \in \bb P(\bb V)$ and $y = \bb R \varphi \in \bb P(\bb V^*)$, 
let
\begin{align} \label{def of delta func-001}
\delta(x, y) = - \log \frac{ |\varphi(v)| }{ \| \varphi \| \| v\| } \geq 0. 
\end{align}
The ideal perturbation function $f$ is defined as follows: 
for $\omega \in \Omega$ and $y\in \bb P(\bb V^*)$,  
\begin{align} \label{perturbed function in bb Y-001}
f(\omega,y) = \delta ( \xi (\omega), y ). 
\end{align}
For any $y\in \bb P(\bb V^*)$, the function $\omega \mapsto f(\omega, y)$ 
is finite almost everywhere with respect to $\bb P$ on $\Omega$.

The function $\delta$ and the cocycles $\sigma$, $\sigma^*$ are involved in the following cohomological equation:
for any $g \in \bb G$ and $(x, y) \in \Delta$, 
\begin{align*}
\delta(gx, gy) = \delta(x, y) + \sigma(g, x) + \sigma^*(g, y). 
\end{align*}
This formula can be restated in the following more intuitive way: 
\begin{align} \label{cohomological eq -vers 002}
\sigma(g, x) - \delta(gx, y)  = \sigma^*(g^{-1}, y) - \delta (x, g^{-1}y), 
\end{align}
which will be applied to establish reversal identities in the next subsection. 

\subsection{The reversal lemma} \label{sec-reversed random walk-001}
In this subsection, we shall present a reversal lemma which enables us to convert problems involving the random walks 
$(\sigma(g_n\cdots g_1,x))_{n \geq 1}$, conditioned to stay non-negative, 
into corresponding problems for the random walk $-\sigma^*(g^{-1}_n\cdots g^{-1}_1,y)$ with some perturbations depending on the future coordinates. 
This step is essential in the proofs of Theorems \ref{Cor-CLLT-cocycle-001} and \ref{Thm-CLLT-cocycle-bound-001}.
It is also necessary when applying the excursion decomposition approach in the proof of 
Theorem \ref{Thm-CLLT-cocycle} which states our conditioned local limit theorem. 

Since the semigroup $\Gamma_{\mu}^{-1}$ is proximal and strongly irreducible in $\bb V^*$, 
the probability measure $\mu^{-1}$ on $\Gamma_{\mu}^{-1}$ admits a unique stationary probability measure $\nu^*$ on $\bb P(\bb V^*)$. 
For $x \in \bb P(\bb V)$, denote
\begin{align*}
\Delta_x: = \Big\{ y \in \bb P(\bb V^*):  \bb P( \delta( g_n \cdots g_1 x, y) < \infty ) = 1, \  \forall n \geq 1  \Big\}. 
\end{align*}
By Lemma 5.1 of \cite{GQX24a}, we have that $\nu^*(\Delta_x)  =1$ for any $x \in \bb P(\bb V)$. 

In order to establish Theorem \ref{Cor-CLLT-cocycle-001}, 
we will first relate the integral on the left-hand side of \eqref{Asym-rho-001} 
to the equivalent expression in terms of the following array of reversed random walks:  
for $\omega=(g_1,g_2,\ldots) \in \Omega,$  
 $x\in\bb P(\bb V)$, $y \in \Delta_x$ 
and $0 \leq n \leq m$, 
\begin{align} \label{reversed RWfor products-001}
 \tilde S^{x,m}_{n}(\omega, y) 
  = -\sigma^*( g^{-1}_{n}\cdots g^{-1}_1,y)  
 +   \delta(g_{n+1}\cdots g_m x, g^{-1}_{n} \cdots g^{-1}_{1}y) - \delta(g_1\cdots g_m x,y), 
\end{align}
where $\sigma^*$ and $\delta$ are defined by \eqref{dual cocycle-001} and \eqref{def of delta func-001}, repectively. 
Besides, we adopt the convention that 
\begin{align} \label{reversed RWfor products-001bbb}
 \tilde S^{x,m}_{0}(\omega, y) = 0
\end{align}
and 
\begin{align}\label{convention-S-xn-y}
 \tilde S^{x,n}_{n}(\omega, y) 
 & = -\sigma^*( g^{-1}_{n}\cdots g^{-1}_1,y) +   \delta(x, g^{-1}_{n} \cdots g^{-1}_{1}y) - \delta(g_1\cdots g_n x,y) \notag\\
 & = - \sigma(g_1 \cdots g_n, x), 
\end{align}
where the last equality holds due to \eqref{cohomological eq -vers 002}. 
In these definitions, the additional parameter $m$ is interpreted as the range of dependence, 
indicating how far into the future the perturbation of the random walk $-\sigma^*( g^{-1}_{n}\ldots g^{-1}_1,y)$
can extend. 
For $\omega=(g_1,g_2,\ldots) \in \Omega,$  
 $y \in\bb P(\bb V^*)$ 
and $n \geq 1$,  we also define 
\begin{align*}
\tilde S_n(\omega,y) = - \sigma^*( g^{-1}_n\cdots g^{-1}_1, y)  + f\circ \widetilde{T}^n (\omega, y) - f(\omega, y). 
\end{align*}
where $f$ and $\widetilde{T}$ are defined by \eqref{perturbed function in bb Y-001} and \eqref{def-tilde-T}, respectively. 
Formally, this corresponds to the definition in \eqref{reversed RWfor products-001} in the case where $m = \infty$. 

The 
left-hand side of \eqref{Asym-rho-001} is connected to
the array $(\tilde S^{x,n}_{k}(\cdot, y))_{1\leq k\leq n}$ through the following reversal lemma, in analogy to Lemma 3.3 of \cite{GQX24a}.

\begin{lemma} \label{lemma-duality-for-product-001}
Let $x\in \bb P(\bb V)$ and $y\in \Delta_x$. For any $n\geq 1$, any non-negative measurable function $h$ on
$\bb P(\bb V)\times \bb R$
and any non-negative measurable function $\psi$ on $\bb R$, we have 
\begin{align} \label{duality id-002}
& \int_{\bb R}  \bb E \Big[ h\Big( g_n\cdots g_1 x, t+ \sigma (g_n\cdots g_1,x) \Big);  \tau_{x,t} > n - 1 \Big] \psi(t) dt \notag\\
 &= \int_{\bb R} \bb E\Big[    h(g_1\cdots g_n x, t) \psi \Big( t + \tilde S^{x,n}_{n}(\omega, y)  \Big);  
  t+ \tilde S^{x,n}_{k}(\cdot,y) \geq 0, 1 \leq k \leq n-1  \Big] dt 
\end{align}
and 
\begin{align} \label{duality id-001}
& \int_{\bb R}  \bb E \left[ h\Big( g_n\cdots g_1 x, t+ \sigma (g_n\cdots g_1,x) \Big);  \tau_{x,t} > n \right] \psi(t) dt \notag\\
 &= \int_{\bb R_+} \bb E \Big[   h(g_1\cdots g_n x, t) \psi \Big( t + \tilde S^{x,n}_{n}(\omega, y)  \Big);  
   t+ \tilde S^{x,n}_{k}(\cdot,y) \geq 0, 1 \leq k \leq n-1  \Big]  dt. 
\end{align}
\end{lemma}
\begin{proof}
The identity \eqref{duality id-001} follows directly from \eqref{duality id-002}, so we will only prove \eqref{duality id-002}. 

For brevity, we set $S^x_0=0$ and $S^x_n= \sigma(g_n\cdots g_1, x)$, for $x\in \bb P(\bb V)$ and $n\geq 1$. 
Using the change of variable $u = t + S^x_n$, the left-hand side of \eqref{duality id-002} 
can be expressed as follows: 
\begin{align} \label{After a change of variable-001}
J&:=\int_{\bb R}  \bb E \Big[ h  \Big( g_n\cdots g_1 x, t + S^x_n \Big);  \tau_{x,t} > n - 1 \Big] \psi(t) dt \notag\\
&= \int_{\bb R}  \bb E \Big[ h  \Big(g_n\cdots g_1 x, t + S^x_n\Big) \psi(t);  
t +S^x_k \geq 0, 1\leq k\leq n-1  \Big]  dt \notag\\
 &= \int_{\bb R} \bb E\Big[   h(g_n\cdots g_1 x, u) \psi(u - S^x_n);    u - S^x_n + S^x_k \geq 0, 1 \leq k\leq n - 1 \Big] du.
\end{align}
By the cocycle property, we have that for any $x\in \bb P(\bb V)$ and $1 \leq k\leq n-1$,
\begin{align*} 
S^x_n - S^x_k = \sigma(g_{n}\cdots g_{k+1}, g_{k}\cdots g_{1} x ). 
\end{align*}
We shall make use of the following cohomological equation, which is obtained from \eqref{cohomological eq -vers 002}:
for $g \in \bb G$, $a \in \bb P(\bb V)$ and $b \in \bb P(\bb V^*)$, 
\begin{align*} 
\sigma(g, a)    =  \sigma^*(g^{-1}, b)  - \delta (a, g^{-1}b) +  \delta(g a, b). 
\end{align*}
Applying this with $g =g_{n}\cdots g_{k+1}$, $a =g_{k}\cdots g_{1} x$ and $b=y$, we get that, for any $1\leq k\leq n-1$,
$x\in \bb P(\bb V)$ and $y\in \Delta_x$, 
\begin{align} \label{cohomological eq -vers 004}
S^x_n - S^x_k 
& = \sigma(g_{n}\cdots g_{k+1}, g_{k}\cdots g_{1} x )  \notag \\
& = \sigma^*((g_{n}\cdots g_{k+1})^{-1}, y)  
 - \delta (g_{k}\cdots g_{1} x, (g_{n}\cdots g_{k+1})^{-1}y)  +  \delta(g_{n}\cdots g_{1} x, y).
\end{align}
Substituting the identity \eqref{cohomological eq -vers 004} into \eqref{After a change of variable-001},
 we obtain
\begin{align*} 
& J= \int_{\bb R} \bb E\Big[    h(g_n\cdots g_1 x, u) \psi( u - \sigma(g_n \cdots g_1, x) );  \notag\\
&  u -\sigma^*( g^{-1}_{k+1}\cdots g^{-1}_n,y) +  \delta(g_k\cdots g_1 x, g^{-1}_{k+1} \cdots g^{-1}_{n}y)-\delta(g_n\cdots g_1x,y)
\geq 0,  \notag\\
&\qquad\qquad\qquad\qquad 1 \leq k\leq n-1  \Big] du.
\end{align*}
The random elements  $g_1,\ldots,g_n$ are exchangable, which is justified by the assumption that these  elements 
are independent and identically distributed. 
Therefore, we can reverse the order of elements $g_1,\ldots,g_n$, 
which gives 
\begin{align*} 
J  &= \int_{\bb R}  \bb E\Big[ 
   h(g_1\cdots g_n x, u) \psi( u - \sigma(g_1 \cdots g_n, x) );  \notag \\
&  u -\sigma^*( g^{-1}_{k}\cdots g^{-1}_1,y) 
 +   \delta(g_{k+1}\cdots g_n x, g^{-1}_{k} \cdots g^{-1}_{1}y) - \delta(g_1\cdots g_nx,y) \geq 0, 
\notag\\
&\qquad\qquad\qquad\qquad  1 \leq k\leq n-1  \Big] du \notag\\ 
 &= \int_{\bb R} \bb E \Big[ 
   h(g_1\cdots g_n x, u)  \psi \left( u +  \tilde S^{x,n}_{n}(\omega, y) \right);  
   u + \tilde S^{x,n}_{k}(\cdot, y) \geq 0,  
   1 \leq k \leq n -1  \Big] du, 
\end{align*}
where in the last equality we used \eqref{reversed RWfor products-001} and \eqref{convention-S-xn-y}. 
This finishes the proof of \eqref{duality id-002}. 
\end{proof}

From Lemma \ref{lemma-duality-for-product-001}, by integrating with respect to $\nu^*(dy)$ and using Fubini's theorem, we obtain
the following alternative representation of the integral on the left-hand side of \eqref{duality id-001}:
\begin{align} \label{duality id-integr-version-002}
& \int_{\bb R}  \bb E \Big[ h\Big( g_n\cdots g_1 x, t+ \sigma (g_n\cdots g_1,x) \Big);  \tau_{x,t} > n \Big] dt   \notag\\
 &=  \int_{\bb R_+} \int_{\bb P(\bb V^*)} \bb E\Big[ 
 h(g_1\cdots g_n x, t);   t+ \tilde S^{x,n}_{k}(\cdot,y) \geq 0, 1\leq k\leq n - 1 \Big] \nu^*(dy) dt.
\end{align}
To apply the results from Appendix \ref{sec invar func},
we need to rewrite the perturbation functions in the reversed random walk \eqref{reversed RWfor products-001} 
in a more convenient form.
For any $\omega = (g_1, g_2, \ldots) \in \Omega$, $x \in \bb P(\bb V)$ and $y \in \Delta_x$, we denote 
\begin{align} \label{more general perturb-001}
f_n^{x, m}(\omega, y) 
= \delta(g_1 \cdots g_{m-n} x, y),\ 0 \leq n \leq m, 
\end{align}
where $\delta$ is defined by \eqref{def of delta func-001}. 
From \eqref{more general perturb-001} and \eqref{def-tilde-T}, 
we get that, for any $0 \leq n \leq m$, $\omega = (g_1, g_2, \ldots) \in \Omega$, $x \in \bb P(\bb V)$ and $y \in \Delta_x$, 
\begin{align*} 
f_n^{x, m} \circ \widetilde T^n(\omega, y)  = \delta \left( g_{n+1}\cdots g_m x, g^{-1}_{n} \cdots g^{-1}_{1}y \right) 
\end{align*}
with the convention $f_m^{x, m} \circ \widetilde T^m(\omega, y)  = \delta (x, g^{-1}_{m} \cdots g^{-1}_{1}y )$ 
and $f_0^{x, m} (\omega, y) = \delta (g_{1}\cdots g_m x, y )$. 
With this notation, we can rewrite 
$(\tilde S^{x,m}_{n}(\cdot,y))_{1\leq n \leq m}$ defined by \eqref{reversed RWfor products-001} as 
\begin{align} \label{another form of the perturbRW-001}
\tilde S^{x,m}_{n}(\omega,y) 
= -\sigma^*( g^{-1}_{n}\cdots g^{-1}_1,y) +  f_n^{x, m} \circ \widetilde T^n(\omega, y)  - f_0^{x, m} \left(\omega, y \right),
\end{align}
where $\omega$ stands for the infinite sequence $(g_1, g_2, \ldots) \in \Omega$.

\subsection{Proof of Theorem \ref{Cor-CLLT-cocycle-001}} \label{sec proof T1.6}

Observe that the second assertion \eqref{Asym-rho-001} follows from the first one \eqref{Asym-rho-002} by applying it to the function
$(x, t) \mapsto  h(x, t) \mathds 1_{\{t \geq 0\}}$, along with an approximation argument to eliminate the discontinuity at $0$. 


Conversely, the first assertion \eqref{Asym-rho-002} is a consequence of the second one \eqref{Asym-rho-001}, the definition 
\eqref{operator R 001} of the operator $R$, 
and the invariance property of the Radon measure $\rho$ (see Corollary 1.4 of \cite{GQX24a}): 
\begin{align*}
\int_{\bb P(\bb V) \times \bb R_+}  \bb E h \Big( g_1 x, t + \sigma(g_1, x)  \Big)  \rho(dx, dt)
&= \int_{\bb P(\bb V) \times \bb R}  R h(x, t)  \rho(dx, dt) \\
 &= \int_{\bb P(\bb V) \times \bb R}  h(x, t)  \rho(dx, dt). 
\end{align*}
Therefore, we only need to prove the second assertion \eqref{Asym-rho-001}.

At this point, we assume that $h(x, t) = \varphi(x) \psi(t)$ for $x \in \bb P(\bb V)$ and $t \in \bb R$, 
where $\varphi$ is a non-negative H\"older continuous function on $\bb P(\bb V)$
and $\psi$ is a compactly supported non-negative continuous function on $\bb R$. 
For $x \in \bb P(\bb V)$, $t\in \bb R$ and $1 \leq n \leq m$, we set 
\begin{align} \label{def of U^xm_n with phi-001}
U^{x,m,\varphi}_{n}(t) 
& =  \int_{\bb P(\bb V^*)}
\bb E \Big[  \Big( t +\tilde S^{x,m}_{n}(\cdot,y) \Big) \varphi(g_1\cdots g_m x);    t + \tilde S^{x,m}_{k}(\cdot,y) \geq 0, 1\leq k\leq n  \Big]  \nu^*(dy), 
\end{align}
where $( \tilde S^{x,m}_{k}(\cdot,y) )_{1\leq k\leq m}$ is given in \eqref{another form of the perturbRW-001}. 
By Corollary 2.3, Proposition 2.4 in \cite{GQX24a} 
and the strong approximation result (6.5) of \cite{GLP17},  the assumptions of Theorem \ref{Thm-CCLT-limit-000} are satisfied
by the sequence of perturbations $\mathfrak f^{x, n} = (f^{x, n}_k)_{0 \leq k \leq n}$
and the twist function $\omega = (g_1, g_2, \ldots) \mapsto \varphi(g_1\cdots g_n x)$. 
As in Subsection \ref{subsec invar func fixed}, 
for any $y \in \bb P(\bb V^*)$ and $t\in \bb R$, denote by $\tau_{y, t}^{\mathfrak f^{x, n}}$ the first time
when the perturbed random walk $( t + \tilde S^{x,n}_{k}(\cdot, y) ) _{k \geq 1}$ exits $\mathbb{R}_{+}= [ 0,\infty),$ 
\begin{align} \label{def-stop time with preturb-001-bb}
\tau_{y,t}^{\mathfrak f^{x, n}} = \min \left\{ k\geq 1: t + \tilde S^{x,n}_{k}(\cdot, y) < 0\right\}. 
\end{align}
Therefore, applying Theorem 4.1 from \cite{GQX24a} and Theorem \ref{Thm-CCLT-limit-000} from Appendix \ref{sec invar func}, 
we obtain that uniformly in $x \in \bb P(\bb V)$, 
\begin{align*}
& \lim_{n\to\infty} \bigg( \frac{ \upsilon_{\mu} \sqrt{2 \pi n} }{ 2 } 
\int_{\bb P(\bb V^*)}  \int_{\bb R_+} \psi(t)
\bb E \left( \varphi(g_1\cdots g_n x);   \tau_{y, t}^{\mathfrak f^{x, n}} >n - 1   \right)   dt \,  \nu^*(dy)  \notag\\
& \qquad\qquad\qquad  -  \int_{\bb R_+} \psi(t) U^{x,n,\varphi}_{[n/2]}(t)  dt \bigg) = 0.  
\end{align*}
Using Proposition 3.1 and Corollary 3.6 from \cite{GQX24a}, we get that uniformly in $x \in \bb P(\bb V)$, 
\begin{align*}
\lim_{n\to\infty} \int_{\bb R_+} \psi(t) U^{x,n,\varphi}_{[n/2]}(t)  dt  
 = \lim_{n\to\infty} \int_{\bb R_+} \psi(t) U^{\varphi}_{[n/2]}(t)  dt  = \int_{\bb P(\bb V) \times \bb R_+}  \varphi(x) \psi(t) \rho(dx,dt) 
\end{align*}
where $U^{\varphi}_{[n/2]}$ is defined by (3.3) in \cite{GQX24a}, and 
 $0$ does not belong to the support of $\psi$. If $0$ actually belongs to the support of $\psi$, the above convergence property
can also be shown to hold by using standard approximation arguments 
alongside the uniform bound provided by Corollary 4.3 of \cite{GQX24a}. 
Therefore, by applying Lemma \ref{lemma-duality-for-product-001}, the assertion \eqref{Asym-rho-001} follows
for functions of the form $h(x, t) = \varphi(x) \psi(t)$,
with $\varphi$ non-negative H\"older continuous on $\bb P(\bb V)$
and $\psi$ compactly supported non-negative continuous on $\bb R$.
Using again Corollary 4.3 of  \cite{GQX24a}, 
we can extend this result to any continuous compactly supported function
 $h$ on $\bb P(\bb V) \times \bb R$ through standard approximation techniques. 


\subsection{Proof of Theorem \ref{Thm-CLLT-cocycle-bound-001}} \label{sec proof of Th1.8}
The proof relies on the conditioned local limit theorem of order $\frac{1}{n}$ (Corollary \ref{Cor-Cara-Bound}), 
bounds for H\"older norms (Proposition \ref{Lem_Inequality_Aoverline}), 
the reversal lemma (Lemma \ref{lemma-duality-for-product-001}) and estimates for random walks with perturbations depending on future coordinates 
(Theorem \ref{Thm-inte-exit-time-001}). 

We fix a constant $\ee>0$ and a non-negative smooth compactly supported function $\kappa$ on $\bb R$ such that 
$\kappa = 1$ on $[-\ee, \ee]$. 
Let $a<b$ be real numbers, and set $F = \mathds 1_{[a, b]}$ and $G = F * \kappa$ so that in particular, $F \leq_{\ee} G$. 
For $n \geq 1$,  $x \in \bb P(\bb V)$ and $t \in \bb R$, 
we set 
\begin{align*}
\Psi_n (x,t): & = \bb P \Big(  t + \sigma(g_n \cdots g_1, x) \in [a, b],   \tau_{x, t} > n -1 \Big) \notag\\
& = \bb E \Big(  F \left( t + \sigma(g_n \cdots g_1, x)  \right);   \tau_{x, t} > n -1 \Big). 
\end{align*}
Set $m=\left[ n/2 \right]$ and $k = n-m.$   
By the Markov property, we get that for any $x \in \bb P(\bb V)$ and $t \in \bb R$, 
\begin{align*}
\Psi_n (x,t) = \bb E \Big[ \Psi_m \Big( g_k \cdots g_1 x, t + \sigma(g_k \cdots g_1, x) \Big);  \tau_{x, t} > k   \Big]. 
\end{align*}
We also define, for any $m \geq 1$, $x' \in \bb P(\bb V)$ and $t' \in \bb R$, 
\begin{align}\label{def-Psi-x-t-00a}
\overline \Psi_m(x', t')  :  = 
\bb E \Big[ G \Big(  t' + \sigma(g_m \cdots g_1, x') \Big)  
  \chi_{\ee}  \Big( t' + \ee + \min_{1 \leq  j \leq m-1}  \sigma(g_j \cdots g_1, x')  \Big) \Big], 
\end{align}
where $\chi_{\ee}$ is defined by \eqref{Def_chiee}. 
Since $F \leq_{\ee} G$, it follows that $\Psi_m \leq_{\ee}  \overline \Psi_m$. 
By reasoning in the same way as in the proof of Proposition \ref{Lem_Inequality_Aoverline}, 
one can verify that the function $\overline \Psi_m$ belongs to the space $\scr H_{\gamma}$,
and that $\| \overline \Psi_m \|_{ \scr H_{\gamma} }  \leq \frac{c}{\ee}  \| G\|_{ \scr H_{\gamma} } \leq  c'(b-a+1)$. 
Therefore, by Corollary \ref{Cor-Cara-Bound}, 
there exists a constant $c>0$ such that for any $n \geq 2$, $x \in \bb P(\bb V)$, $t \in \bb R$ and $a<b$, 
\begin{align}\label{Cara-Bound008}
& \Psi_n (x,t)  = \bb E \left[ \Psi_m \Big( g_k \cdots g_1 x, t + \sigma(g_k \cdots g_1, x)  \Big) 
  \mathds 1_{ \left\{ \tau_{x, t} > k \right\}}  \right]  \notag\\
 & \leq  \frac{ c}{n} \left( 1 + \max\{t, 0\} \right) 
\bigg( \int_{\bb P(\bb V)}  \int_{\mathbb R}  \overline \Psi_m \left(x', t' \right)  dt' \nu(dx') 
    +   \frac{b-a +1}{ \sqrt{n}}  \bigg).  
\end{align}
Now we provide an upper bound for the integral in \eqref{Cara-Bound008}. 
By \eqref{def-Psi-x-t-00a} and \eqref{Def_chiee}, we have 
\begin{align*}
\overline \Psi_m \left(x', t' \right) 
\leq \bb E \Big[ G \Big( t' + \sigma(g_m \cdots g_1, x') \Big);   
  \tau_{x', t' + 2 \ee} > m - 1 \Big]. 
\end{align*}
Using a change of variable and the reversal identity \eqref{duality id-002} of Lemma \ref{lemma-duality-for-product-001}, 
we derive that there exists a constant $c>0$ such that for any $m \geq 1$, $x' \in \bb P(\bb V)$ and $y\in \Delta_{x'}$, 
\begin{align*}
\int_{\mathbb R}  \overline \Psi_m \left(x', t' \right)  dt'  
& \leq  \int_{\mathbb R}  \bb E  \Big[ G \Big( t' - 2 \ee + \sigma(g_m \cdots g_1, x') \Big);   
  \tau_{x', t'} > m - 1 \Big]   dt'  \notag\\
& = \int_{\bb R} \bb E\Big[    G(u - 2 \ee);   u + \tilde S^{x',m}_{j}(\cdot,y) \geq 0, 1 \leq j \leq m-1  \Big] du  \notag\\
& \leq c \int_{a - c}^{b+c} \bb P \left(  u + \tilde S^{x',m}_{j}(\cdot,y) \geq 0, 1 \leq j \leq m-1  \right)  du. 
\end{align*}
Integrating over $\bb P(\bb V) \times \bb P(\bb V^*)$ with respect to $\nu \otimes \nu^*$,
we get 
\begin{align*}
& \int_{\bb P(\bb V)} \int_{\mathbb R}  \overline \Psi_m \left(x', t' \right)  dt'   \nu(dx') \notag\\
& \leq c \int_{\bb P(\bb V)}  \int_{a - c}^{b+c} 
\int_{\bb P(\bb V^*)}  \bb P \Big(  u + \tilde S^{x',m}_{j}(\cdot,y) \geq 0, 1 \leq j \leq m-1  \Big)  \nu^*(dy) du  \nu(dx'). 
\end{align*}
Applying Corollary 2.3 and Proposition 2.4 in \cite{GQX24a}, 
and the strong approximation result (6.5) of \cite{GLP17},  the assumptions of Theorem \ref{Thm-inte-exit-time-001} are satisfied
by the sequence of perturbations $\mathfrak f^{x, n} = (f^{x, n}_m)_{0 \leq m \leq n}$.  
By Theorem \ref{Thm-inte-exit-time-001}, it follows that there exist constants $c', c'' >0$ such that for any $m \geq 1$ and $a<b$, 
\begin{align*}
\int_{\bb P(\bb V)} \int_{\mathbb R}  \overline \Psi_m \left(x', t' \right)  dt'   \nu(dx')
 \leq  c'  \int_{a - c}^{b+c}  \frac{1 + \max\{u, 0\}}{\sqrt{n}} du 
  \leq  \frac{c''}{\sqrt{n}} (1+ b - a) (1 + \max\{b, 0\}). 
\end{align*}
Substituting this into \eqref{Cara-Bound008} concludes the proof of Theorem \ref{Thm-CLLT-cocycle-bound-001}.

\subsection{Proof of Corollary \ref{Corol-local prob for exit time}}
For $x \in \bb P(\bb V)$ and $t\in \bb R$, set
$F(x,t) = \bb P\left( t+\sigma(g_1,x)<0 \right).$
By Markov's inequality and the moment assumption \eqref{Exponential-moment}, 
there exist constants $\alpha, c > 0$ such that for any $x \in \bb P(\bb V)$ and $t \in \bb R$, 
\begin{align*} 
F(x,t) \leq c e^{-\alpha t}. 
\end{align*}
Using the Markov property, we derive that, for any $n\geq 2$, $x \in \bb P(\bb V)$ and $t\in \bb R$,
\begin{align*} 
 \bb P\left( \tau_{x,t}=n \right) 
&= \int_0^{\infty} \bb E \Big[ F(g_{n-1}\cdots g_1 x, t');  t+\sigma(g_{n-1}\cdots g_1,x)\in dt', \tau_{x,t} > n-1 \Big] \notag\\
&\leq c\int_0^{\infty} e^{-\alpha t'} \bb P\Big(  t+\sigma(g_{n-1}\cdots g_1,x)\in dt', \tau_{x,t} > n-1 \Big) \notag\\
&\leq  c \sum_{k=0}^{\infty} e^{-\alpha k} \bb P\Big(  t+\sigma(g_{n-1}\cdots g_1,x)\in [k,k+1), \tau_{x,t} > n-1 \Big) \notag\\
&\leq c \sum_{k=0}^{\infty} e^{-\alpha k} (1+\max\{t,0\})(k+1)(n-1)^{-3/2} \notag\\
&\leq c (1+\max\{t,0\}) n^{-3/2},
\end{align*}
where in the fourth line we used Theorem \ref{Thm-CLLT-cocycle-bound-001}.
The conclusion of the corollary follows.


\section{Proof of Theorem \ref{Thm-CLLT-cocycle}} \label{Proof of main Theorem} 

In this final section of the paper, we establish our main results, Theorem \ref{Thm-CLLT-cocycle} and its Corollary \ref{Thm-CLLT-cocycle-002}.

The main steps of the proof are outlined as follows.  
We shall prove a version of Theorem \ref{Thm-CLLT-cocycle} that provides upper and lower bounds 
instead of the asymptotic formulation, utilizing the excursion decomposition approach. 
To this aim, the first half of the trajectory is controlled by using the Caravenna-type local limit theorem stated in Theorem \ref{t-A 001}.
The second half of the trajectory is reversed by means of Lemma \ref{lemma-duality-for-product-001} and then 
analyzed through the conditioned central limit theorems for random walks with perturbations developed in Theorem \ref{Thm-CCLT-limit}.

\subsection{Regularity bounds for the expectations}

In the proof of Theorem \ref{Thm-CLLT-cocycle}, we will make use of a technical result stated below, which 
is similar to Proposition \ref{Lem_Inequality_Aoverline}. 
This result will allow us to smooth certain functions appearing in the proof of Theorem \ref{Thm-CLLT-cocycle}
in order to be able to apply Theorem \ref{t-A 001}. 

For $\gamma \in (0,1]$, we say that a function $G$ on $\bb P(\bb V) \times \bb R$
is $\gamma$-regular if there is a constant $c$ such that for any $(x, t)$ and $(x', t')$ in $\bb P(\bb V) \times \bb R$,
we have $|G(x,t) - G(x',t')| \leq c( |t-t'| + d(x,x')^{\gamma} )$. 
In other words, a function is $\gamma$-regular if and only if it is Lipschitz continuous on $\bb P(\bb V) \times \bb R$
when $\bb R$ is equipped with the standard distance and $\bb P(\bb V)$ is equipped with the distance 
$(x, x') \mapsto d(x,x')^{\gamma}$. 
Recall that the function $\chi_{\ee}$ is defined by \eqref{Def_chiee}. 

\begin{proposition}\label{Lem_HolderNormPsi}
Assume that $\Gamma_{\mu}$ is proximal and strongly irreducible, $\mu$ admits finite exponential moments
and that the Lyapunov exponent $\lambda_{\mu}$ is zero. 
For any small enough $\gamma >0$, there exist constants $\beta, c >0$ with the following property. 
Let $G$ be a $\gamma$-regular function with compact support on $\bb P(\bb V) \times \bb R$.
For $(x,t) \in \bb P(\bb V) \times \bb R$, $m \geq 1$ and $\ee \in (0, 1)$, define 
\begin{align*}
\overline \Psi_{m,\ee}(x, t)  :  = \bb E \Big[   G \Big( g_m \cdots g_1 x, t + \sigma (g_m \cdots g_1, x) \Big) 
  \times  \chi_{\ee}  \Big( t + \ee + \min_{1 \leq  j \leq m-1}  \sigma (g_j \cdots g_1, x) \Big)  \Big]. 
\end{align*}
Then we have $\overline \Psi_{m,\ee} \in \scr H_{\gamma}$ and  $\| \overline \Psi_{m,\ee} \|_{\scr H_{\gamma}} \leq \frac{c}{\ee m^{\beta}}.$
\end{proposition}

In the proof of Proposition \ref{Lem_HolderNormPsi}, 
we will use the following lemma, the proof of which is based on Theorem \ref{Thm-CLLT-cocycle-bound-001}.

\begin{lemma}\label{Lem-inte-prob-001}
Assume that $\Gamma_{\mu}$ is proximal and strongly irreducible, $\mu$ admits finite exponential moments
and that the Lyapunov exponent $\lambda_{\mu}$ is zero. 
Then, for every $1 \leq p < 2$, there exist constants $c, \eta, \varkappa >0$ such that for any $n \geq 1$ and $- \infty < a <b \leq n^{\varkappa}$, 
\begin{align*}
\int_{\bb R} \sup_{x \in \bb P(\bb V)} \bb P^{1/p} \Big( t + \sigma(g_n \cdots g_1, x) \in [a, b], \tau_{x, t} > n-1 \Big) dt
\leq \frac{c}{n^{\eta}} ( 1+ b-a )^{1/p}.  
\end{align*} 
\end{lemma}

\begin{proof}
We choose $\frac{1}{2} < \beta <\frac{3}{2(p+1)}$, which is possible since $p<2$. Then, we write
\begin{align*}
 \int_{\bb R} \sup_{x \in \bb P(\bb V)} \bb P^{1/p} \Big( t + \sigma(g_n \cdots g_1, x) \in [a, b], \tau_{x, t} > n-1 \Big) dt  
 =: I_1(n) + I_2(n) + I_3(n), 
\end{align*}
where 
\begin{align*}
I_1(n) & = \int_{-\infty}^{-n^\beta} \sup_{x \in \bb P(\bb V)} \bb P^{1/p} \Big( t + \sigma(g_n \cdots g_1, x) \in [a, b], \tau_{x, t} > n-1 \Big) dt, \notag\\
I_2(n) & = \int_{-n^\beta}^{n^\beta}  \sup_{x \in \bb P(\bb V)} \bb P^{1/p} \Big( t + \sigma(g_n \cdots g_1, x) \in [a, b], \tau_{x, t} > n-1 \Big) dt, \notag\\ 
I_3(n) & = \int_{n^\beta}^{\infty} \sup_{x \in \bb P(\bb V)} \bb P^{1/p} \Big( t + \sigma(g_n \cdots g_1, x) \in [a, b], \tau_{x, t} > n-1 \Big) dt. 
\end{align*}
For the first term $I_1(n)$, we use the following domination: for $n \geq 2$, $x \in \bb P(\bb V)$ and $t<0$,
\begin{align*}
\mathbb P(\tau_{x,t}>n-1)\leq \mathbb P(\sigma(g_1,x)\geq -t)\leq 
\mathbb P(\log\|g_1\|\geq -t). 
\end{align*}
Hence, as the measure $\mu$ has finite exponential moments, there exist constants $\alpha, \eta, c>0$ such that for any $n \geq 2$ and $a<b$, 
\begin{align*}
I_1(n) \leq c \int_{-\infty}^{-n^\beta} e^{\alpha t/p}dt 
\leq c n^{-\eta}. 
\end{align*}
For the second term $I_2(n)$, we dominate it by using Theorem \ref{Thm-CLLT-cocycle-bound-001}, which gives
that there exists a constant $c>0$ such that for any $n \geq 1$ and $a<b$, 
\begin{align*}
I_2(n) & \leq  c n^{-3/2p} ( 1+ b-a )^{1/p} (1 + \max\{b, 0\} )^{1/p}
\int_{-n^\beta}^{n^\beta}(1+\max\{t, 0 \} )^{1/p}dt \notag\\
&\leq 
c ( 1+ b-a )^{1/p} (1+\max\{b, 0\})^{1/p} n^{\beta(1+1/p)-3/2p}. 
\end{align*}
Since $\beta < \frac{3}{2(p+1)}$, by choosing $\varkappa >0$ small enough, the right-hand side of the above bound 
is dominated by $c ( 1+ b-a )^{1/p} n^{-\eta}$ as $b \leq n^{\varkappa}$ for some constant $\eta >0$. 


Finally, for the last term $I_3(n)$, we choose some $q> 1$ to be determined later. 
Using Chebyshev's  inequality, when $t\geq  n^{\beta}> b+1$, we write
\begin{align*}
\mathbb P \Big( t+\sigma(g_n\cdots g_1,x) \in [a, b] \Big)
 \leq (t-b)^{-q} \mathbb E \Big( |\sigma(g_n\cdots g_1,x)|^{q} \Big) 
 \leq c t^{-q} \mathbb E \Big( |\sigma(g_n\cdots g_1,x)|^{q} \Big). 
\end{align*}
By Lemma 10.18 in \cite{BQ16b}, the sequence $(\sigma(g_n\cdots g_1,x))_{n \geq 1}$
is the sum of a martingale and a bounded term, 
so by Burkholder's inequality, we obtain 
\begin{align*}
\mathbb P \Big( t+\sigma(g_n\cdots g_1,x) \in [a, b] \Big) \leq  c t^{-q}n^{q/2}. 
\end{align*}
Hence, for $p<q$,
$$I_3(n) \leq c n^{q/2p}\int_{n^\beta}^\infty \frac{dt}{t^{q/p}}\leq 
c n^{q/2p}n^{\beta (1-q/p)}.$$
Since $\beta > 1/2$, by taking $q$ large enough, we have 
$q/2p+\beta (1-q/p)<0$, which in turn implies that $I_3(n)$ is dominated by $c n^{-\eta}$ for some constant $\eta>0$.
The assertion follows. 
\end{proof}

\begin{proof}[Proof of Proposition \ref{Lem_HolderNormPsi}]
The proof is based on Lemma \ref{Lem-inte-prob-001} and the contraction property of the random walk (Lemma \ref{Lem-Holder-cocycle}). 
It is enough to prove the proposition for a non-negative function $G$. Recall that 
\begin{align}\label{Psi-m-H-alpha-001}
\|  \overline \Psi_{m,\ee} \|_{\scr H_{\gamma}}
  =  \int_{\bb R}   \sup_{x \in \bb P(\bb V)}  | \overline \Psi_{m,\ee} \left(x, t \right)|  dt 
 +  \int_{\bb R} \sup_{x, x' \in \bb P(\bb V): x \neq x'}  
    \frac{| \overline \Psi_{m,\ee} \left(x, t \right)  - \overline \Psi_{m,\ee} \left(x', t \right)  |}{ d(x,x')^{\gamma} } dt. 
\end{align}

We start to bound the first term in \eqref{Psi-m-H-alpha-001}.
For $t \in \bb R$, set $H(t) = \sup_{x \in \bb P(\bb V)} |G(x, t)|$. 
We fix some interval $[a, b]$ such that the function $H$ has support in $[a, b]$. 
Then, by Lemma \ref{Lem-inte-prob-001}, there exist constants $c, c', \eta >0$ such that for any $m \geq 1$, 
\begin{align}\label{decom-Integral-sup-001-aa}
& \int_{\bb R}   \sup_{x \in \bb P(\bb V)}  | \overline \Psi_{m,\ee} \left(x, t \right)|  dt \notag\\
& \leq  \int_{\bb R}   \sup_{x \in \bb P(\bb V)}  \bb E \Big[   H \Big(t + \sigma (g_m \cdots g_1, x) \Big)  
   \chi_{\ee}  \Big( t + \ee + \min_{1 \leq  j \leq m-1}  \sigma (g_j \cdots g_1, x) \Big)  \Big] dt \notag\\
& \leq c \int_{\bb R}   \sup_{x \in \bb P(\bb V)}  \bb P \Big(   t + \sigma (g_m \cdots g_1, x) \in [a, b], \tau_{x, t + 2\ee} > m -1 \Big) dt \notag\\
& \leq c' m^{-\eta}. 
\end{align}

For the second term in \eqref{Psi-m-H-alpha-001}, 
we define, for $m \geq 1$, $x, x' \in \bb P(\bb V)$ and $t \in \bb R$, 
\begin{align*}
& I_1(x, x', m, t) := \bb E \bigg[   G \Big( g_m \cdots g_1 x, t + \sigma (g_m \cdots g_1, x) \Big)      \notag\\
 & \qquad  \times  \bigg| \chi_{\ee}  \Big( t + \ee + \min_{1 \leq  j \leq m-1}  \sigma (g_j \cdots g_1, x) \Big)  
  - \chi_{\ee}  \Big( t + \ee + \min_{1 \leq  j \leq m-1}  \sigma (g_j \cdots g_1, x') \Big) \bigg|  \bigg],   \notag\\
 & I_2(x, x', m, t) := \bb E \bigg[  \chi_{\ee}  \Big( t + \ee + \min_{1 \leq  j \leq m-1}  \sigma (g_j \cdots g_1, x') \Big)  \notag\\
 & \qquad \times   \Big| G \Big( g_m \cdots g_1 x, t + \sigma (g_m \cdots g_1, x) \Big)  
  - G \Big( g_m \cdots g_1 x', t + \sigma (g_m \cdots g_1, x') \Big) \Big|     \bigg]. 
\end{align*}
For $I_1(x, x', m, t)$, 
we let $\varkappa>0$ as in Lemma \ref{Lem-inte-prob-001}, and define, 
for $x, x' \in \bb P(\bb V)$ and $m \geq 1$, 
\begin{align}\label{def-D-x1-x2-m}
D_{x, x', m} = \Big\{ \Big| \sigma (g_k \cdots g_1, x) - \sigma (g_k \cdots g_1, x') \Big| \leq m^{\varkappa}, \forall 1 \leq k \leq m  \Big\}. 
\end{align}
Note that, 
 for any $x, x' \in \bb P(\bb V)$, $t \in \bb R$ and $m \geq 1$, on the set $D_{x, x', m}$,  
 we have that if 
 $$t  + 2\ee + \min_{1 \leq  j \leq m-1}  \sigma (g_j \cdots g_1, x) < - m^{\varkappa},$$ then 
 $t  + 2\ee + \min_{1 \leq  j \leq m-1}  \sigma (g_j \cdots g_1, x') < 0$
 and therefore, 
 \begin{align*}
 \chi_{\ee}  \Big( t + \ee + \min_{1 \leq  j \leq m-1}  \sigma (g_j \cdots g_1, x) \Big) 
 = \chi_{\ee}  \Big( t  + \ee+ \min_{1 \leq  j \leq m-1}  \sigma (g_j \cdots g_1, x') \Big) = 0. 
\end{align*}
In other words, for any $x, x' \in \bb P(\bb V)$, $t \in \bb R$ and $m \geq 1$, on the set $D_{x, x', m}$, 
\begin{align*}
& \left| \chi_{\ee}  \Big( t + \ee + \min_{1 \leq  j \leq m-1}  \sigma (g_j \cdots g_1, x) \Big) 
  - \chi_{\ee}  \Big( t  + \ee+ \min_{1 \leq  j \leq m-1}  \sigma (g_j \cdots g_1, x') \Big)  \right|  \notag\\
& = \left|   \chi_{\ee}  \Big( t + \ee + \min_{1 \leq  j \leq m-1}  \sigma (g_j \cdots g_1, x) \Big) 
  - \chi_{\ee}  \Big( t  + \ee+ \min_{1 \leq  j \leq m-1}  \sigma (g_j \cdots g_1, x') \Big)  \right|  \notag\\
& \qquad \times  \mathds 1_{\{ t  + 2\ee + \min_{1 \leq  j \leq m-1}  \sigma (g_j \cdots g_1, x) \geq - m^{\varkappa} \}}   \notag\\
& \leq \frac{1}{\ee} 
\left|  \min_{1 \leq  j \leq m-1}  \sigma (g_j \cdots g_1, x) - \min_{1 \leq  j \leq m-1}  \sigma (g_j \cdots g_1, x') \right| 
\mathds 1_{\{ \tau_{x, t  + 2 \ee + m^{\varkappa}} > m-1  \}}  \notag\\ 
& \leq \frac{1}{\ee} 
 \max_{1 \leq  j \leq m-1} \Big|    \sigma (g_j \cdots g_1, x) -  \sigma (g_j \cdots g_1, x') \Big| 
\mathds 1_{\{ \tau_{x, t  + 2 \ee + m^{\varkappa}} > m-1  \}}
\end{align*}
where the first inequality holds since $\chi_{\ee}$ is $1/\ee$-Lipschitz continuous on $\bb R$, by \eqref{Def_chiee}. 
Since $|\chi_{\ee}| \leq 1$ and $H(t) = \sup_{x \in \bb P(\bb V)} |G(x, t)|$ for $t \in \bb R$, 
it follows that
\begin{align}\label{expect-G-Holder-contin-001}
I_1(x, x', m, t)   
 & \leq  \bb E \bigg[   H \Big( t + \sigma (g_m \cdots g_1, x) \Big)   
  \bigg| \chi_{\ee}  \Big( t + \ee + \min_{1 \leq  j \leq m-1}  \sigma (g_j \cdots g_1, x) \Big)   \notag\\
 & \qquad\qquad\qquad\qquad\qquad  - 
 \chi_{\ee}  \Big( t + \ee + \min_{1 \leq  j \leq m-1}  \sigma (g_j \cdots g_1, x') \Big) \bigg|; D_{x, x', m}  \bigg]   \notag\\
  & \quad + 2  \bb E \Big[   H \Big( t + \sigma (g_m \cdots g_1, x) \Big);  D_{x, x', m}^c  \Big]   \notag\\
& \leq I_{11}(x, x', m, t) + I_{12}(x, x', m, t),  
\end{align}
where 
\begin{align*}
& I_{11}(x, x', m, t)  =  \frac{1}{\ee}  \bb E \bigg[   H \Big( t + \sigma (g_m \cdots g_1, x) \Big)  \notag\\
& \qquad
 \times \max_{1 \leq  j \leq m-1}  \Big|   \sigma (g_j \cdots g_1, x) -   \sigma (g_j \cdots g_1, x') \Big|; 
\tau_{x, t  + 2 \ee + m^{\varkappa}} > m-1   \bigg] \notag\\
& I_{12}(x, x', m, t)  =  2  \bb E \Big[   H \Big( t + \sigma (g_m \cdots g_1, x) \Big);  D_{x, x', m}^c  \Big]. 
\end{align*}
We choose $1 < q <2$ and set $q' = \frac{q}{q-1}$. Recall also that the function $H$ has support in $[a, b]$. 
For the first term in \eqref{expect-G-Holder-contin-001}, 
we apply H\"older's inequality and Lemma \ref{Lem-Holder-cocycle} to get 
\begin{align*}
 I_{11}(x, x', m, t)  
& \leq  \frac{1}{\ee}  \bb E^{1/q} \bigg[   H \Big( t + \sigma (g_m \cdots g_1, x) \Big)^q;
\tau_{x, t  + 2 \ee + m^{\varkappa}} > m-1   \bigg] \notag\\
& \quad \times \bb E^{1/q'} \left( \max_{1 \leq  j \leq m-1} \left|    \sigma (g_j \cdots g_1, x) -   \sigma (g_j \cdots g_1, x') \right|^{q'} \right)
\notag\\
& \leq  \frac{c}{\ee}  \bb P^{1/q} \Big(   t + \sigma (g_m \cdots g_1, x) \in [a, b],  \tau_{x, t  + 2 \ee + m^{\varkappa}} > m-1   \Big)   \notag\\
& \quad \times
 \bb E^{1/q'} \left(  \max_{1 \leq  j \leq m-1}  \left|  \sigma (g_j \cdots g_1, x) -  \sigma (g_j \cdots g_1, x') \right|^{q'} \right)  \notag\\
 & \leq  \frac{c}{\ee}  d(x, x')^{\gamma} \bb P^{1/q} \Big(   t + \sigma (g_m \cdots g_1, x) \in [a, b],  \tau_{x, t  + 2 \ee + m^{\varkappa}} > m -1  \Big). 
\end{align*}
Applying Lemma \ref{Lem-inte-prob-001}, we derive that there exists a constant $\eta >0$ such that for any $m \geq 1$, 
\begin{align}\label{bound-I-11-x-m-t}
\int_{\bb R} \sup_{x, x' \in \bb P(\bb V): x \neq x'}  \frac{I_{11}(x, x', m, t) }{d(x, x')^{\gamma}} dt
\leq  \frac{c}{\ee} m^{-\eta}. 
\end{align}
For the second term in \eqref{expect-G-Holder-contin-001}, we fix $p > \max\{1, \frac{6}{\varkappa} \}$. 
Then, by \eqref{def-D-x1-x2-m} and Lemma \ref{Lem-Holder-cocycle}, 
we get that, for any small enough $\gamma > 0$, 
\begin{align*}
\bb P \left( D_{x, x', m}^c \right) 
& \leq m^{- p \varkappa} \bb E \left( \max_{1 \leq j \leq m-1} |\sigma(g_j \cdots g_1,x)-\sigma(g_j\cdots g_1,x')|^p \right) \notag\\
& \leq   c m^{- p \varkappa} d(x,x')^{2 \gamma}. 
\end{align*}
Since the function $H$ has support in $[a, b]$, we have
\begin{align*}
 & \bb E^{1/2} \left[ H \Big(t + \sigma (g_m \cdots g_1, x) \Big)^{2} \right] \notag\\
 & \leq \|H\|_{\infty} \bb P^{1/2} \Big(  |t + \sigma (g_m \cdots g_1, x)| \leq |a| + |b| \Big) \notag\\
 & \leq c \bb P^{1/2} \Big(  t \in \Big[ -\log \|g_m \cdots g_1\| - |a| - |b|, \log \|(g_m \cdots g_1)^{-1}\| + |a| + |b| \Big]  \Big). 
\end{align*}
Integrating over $t \in \bb R$ and using the Cauchy-Schwarz inequality,  we get 
\begin{align*}
& \int_{\bb R}  \sup_{x, x' \in \bb P(\bb V): x \neq x'}  \frac{I_{12}(x, x', m, t) }{d(x, x')^{\gamma}}  dt  \notag\\
& = \int_{\bb R}   \sup_{x, x' \in \bb P(\bb V): x \neq x'} 
\frac{1}{d(x, x')^{\gamma}}  \bb E \Big[   H \Big(t + \sigma (g_m \cdots g_1, x) \Big);  D_{x, x', m}^c \Big] dt \notag\\
& \leq  c m^{- p\varkappa/2}  
 \int_{\bb R} \bb P^{1/2} \Big(  t \in \Big[ -\log \|g_m \cdots g_1\| - |a| - |b|, \log \|(g_m \cdots g_1)^{-1}\| + |a| + |b| \Big]  \Big) dt \notag\\
& \leq  c m^{- p\varkappa/2}   
 \int_0^{\infty}  \bb P^{1/2} \Big( |a| + |b| + \max \Big\{ \log \|g_m \cdots g_1\|,  \log \|(g_m \cdots g_1)^{-1}\|  \Big\}  \geq t \Big) dt. 
\end{align*}
By Chebyshev's and Minkowski's inequalities, for $t \geq 2 ( |a| + |b|)$, we have 
\begin{align*}
& \bb P \left( |a| + |b|  + \max \left\{ \log \|g_m \cdots g_1\|,  \log \|(g_m \cdots g_1)^{-1}\|  \right\}  \geq t \right) \notag\\
& \leq  (t- |a| - |b| )^{- 3}  \bb E  \left(  \max \left\{ \log \|g_m \cdots g_1\|, \log \|(g_m \cdots g_1)^{-1}\|  \right\}^{3}  \right) \notag\\
& \leq c m^3 (t- |a| - |b| )^{- 3}. 
\end{align*}
Since $p >\frac{6}{\varkappa}$, we obtain 
\begin{align}\label{decom-Integral-sup-003-aa} 
& \int_{\bb R}  \sup_{x, x' \in \bb P(\bb V): x \neq x'}  \frac{I_{12}(x, x', m, t) }{d(x, x')^{\gamma}}  dt   \notag\\
 & \leq  c m^{- p \varkappa/2} \bigg( 2 (|a| + |b|) + m^{3/2} \int_{2 (|a| + |b|)}^{\infty}  (t- |a| - |b|)^{- 3/2} dt \bigg) \notag\\
 & \leq c' m^{- 3/2}. 
\end{align}
Substituting \eqref{bound-I-11-x-m-t} and \eqref{decom-Integral-sup-003-aa} into \eqref{expect-G-Holder-contin-001}, we get
\begin{align}\label{bound-I-1-x-m-t}
\int_{\bb R}  \sup_{x, x' \in \bb P(\bb V): x \neq x'}  \frac{I_{1}(x, x', m, t) }{d(x, x')^{\gamma}}  dt
\leq \frac{c}{\ee} m^{- \eta} +  c' m^{- 3/2}  \leq  \frac{c''}{\ee} m^{- \eta}. 
\end{align} 

For $I_2(x, x', m, t)$, 
we note that, since $G$ is $\gamma$-regular,
we may find a constant $c>0$ such that, for any $x, x' \in \bb P(\bb V)$ and $t, t' \in \bb R$, 
\begin{align*}
|G(x,t) - G(x',t')| \leq c( |t-t'| + d(x,x')^{\gamma} ) (\mathds 1_{ [a, b] }(t) + \mathds 1_{ [a, b] }(t')). 
\end{align*}
Note also that $\chi_{\ee}  ( t + \ee + \min_{1 \leq  j \leq m-1}  \sigma (g_j \cdots g_1, x') ) 
\leq \mathds 1_{ \{  \tau_{x', t + 2\ee} > m -1 \} }$. 
Now we write 
\begin{align*}
 I_2(x, x', m, t) 
 & \leq \bb E \bigg[   \Big| G \Big( g_m \cdots g_1 x, t + \sigma (g_m \cdots g_1, x)  \Big)  
  \notag \\  
&\qquad\quad  
 - G \Big( g_m \cdots g_1 x', t + \sigma (g_m \cdots g_1, x') \Big) \Big|;   \tau_{x', t + 2\ee} > m -1   \bigg] \notag\\
 & \leq  
 \bb E \bigg[  \Big| \sigma (g_m \cdots g_1, x) -   \sigma (g_m \cdots g_1, x') \Big|
  + d(g_m \cdots g_1 x, g_m \cdots g_1 x')^{\gamma};  \notag\\
& \qquad\qquad\qquad  t + \sigma (g_m \cdots g_1, x') \in [a, b],    \tau_{x', t + 2\ee} > m-1    \bigg]  \notag\\
& \quad + \bb E \bigg[  \Big| \sigma (g_m \cdots g_1, x) -   \sigma (g_m \cdots g_1, x') \Big| 
 + d(g_m \cdots g_1 x, g_m \cdots g_1 x')^{\gamma};  \notag\\
& \qquad\qquad\qquad  t + \sigma (g_m \cdots g_1, x) \in [a, b],    \tau_{x', t + 2\ee} > m-1    \bigg] \notag\\
& =: I_{21}(x, x', m, t) + I_{22}(x, x', m, t). 
\end{align*}
For $I_{21}(x, x', m, t)$, recall that $1 < q <2$ and $q' = \frac{q}{q-1}$. 
We use H\"older's inequality, Lemma \ref{Lem-Holder-cocycle}
and the contraction property of the random walk (see Bougerol and Lacroix \cite[Chapter V, Proposition 2.3]{Boug-Lacr85}) to get 
\begin{align*}
 I_{21}(x, x', m, t) 
& \leq \bb E^{1/q'}  \Big(  \Big| \sigma (g_m \cdots g_1, x) -   \sigma (g_m \cdots g_1, x') \Big| 
+ d(g_m \cdots g_1 x, g_m \cdots g_1 x')^{\gamma}  \Big)^{q'} 
\notag\\
&  \quad \times \bb P^{1/q} \bigg( t + \sigma (g_m \cdots g_1, x') \in [a, b],    \tau_{x', t + 2\ee} > m-1    \bigg) \notag\\
& \leq c d(x, x')^{\gamma}  \bb P^{1/q} \Big( t + \sigma (g_m \cdots g_1, x') \in [a, b],    \tau_{x', t + 2\ee} > m -1   \Big). 
\end{align*}
By Lemma \ref{Lem-inte-prob-001}, there exist constants $c, \eta >0$ such that for any $m \geq 1$, 
\begin{align}\label{bound-I-21-x-m-t}
\int_{\bb R}  \sup_{x, x' \in \bb P(\bb V): x \neq x'}  \frac{I_{21}(x, x', m, t) }{d(x, x')^{\gamma}}  dt
\leq c m^{-\eta}.  
\end{align}
For $I_{22}(x, x', m, t)$, we still use H\"older's inequality with $1 < q <2$ and $q' = \frac{q}{q-1}$, 
Lemma \ref{Lem-Holder-cocycle} and the contraction property of the random walk (\cite[Chapter V, Proposition 2.3]{Boug-Lacr85}) to get 
\begin{align*}
 I_{22}(x, x', m, t) 
& \leq \bb E^{1/q'}  \Big(  \Big| \sigma (g_m \cdots g_1, x) -   
 \sigma (g_m \cdots g_1, x') \Big| 
 + d(g_m \cdots g_1 x, g_m \cdots g_1 x')^{\gamma} \Big)^{q'} 
\notag\\
&  \quad \times \bb P^{1/q} \bigg( t + \sigma (g_m \cdots g_1, x) \in [a, b],    \tau_{x', t + 2\ee} > m-1    \bigg) \notag\\
& \leq c d(x, x')^{\gamma}  \bb P^{1/q} \Big( t + \sigma (g_m \cdots g_1, x) \in [a, b],    \tau_{x', t + 2\ee} > m-1    \Big) \notag\\
& = c d(x, x')^{\gamma}  \bb P^{1/q} \Big( t + \sigma (g_m \cdots g_1, x) \in [a, b],    \tau_{x', t + 2\ee} > m-1, D_{x, x', m}    \Big)  \notag\\
& \quad +  c d(x, x')^{\gamma}  \bb P^{1/q} \Big( t + \sigma (g_m \cdots g_1, x) \in [a, b],    \tau_{x', t + 2\ee} > m-1, D_{x, x', m}^c   \Big) \notag\\
& \leq c d(x, x')^{\gamma}  \bb P^{1/q} \Big( t + \sigma (g_m \cdots g_1, x) \in [a, b],    \tau_{x, t + m^{\varkappa} + 2\ee} > m-1   \Big)  \notag\\
& \quad +  c d(x, x')^{\gamma}  \bb P^{1/q} \Big( t + \sigma (g_m \cdots g_1, x) \in [a, b],   D_{x, x', m}^c    \Big), 
\end{align*}
where in the last inequality we used the definition of $D_{x, x', m}$ (see \eqref{def-D-x1-x2-m}). 
By Lemma \ref{Lem-inte-prob-001}, we have 
\begin{align*}
\int_{\bb R} \sup_{x \in \bb P(\bb V)} 
\bb P^{1/q} \Big( t + \sigma (g_m \cdots g_1, x) \in [a, b],    \tau_{x, t + m^{\varkappa} + 2\ee} > m -1  \Big) dt
\leq c m^{-\eta}.  
\end{align*}
Besides, following the reasoning used to dominate $I_{12}(x, x', m, t)$, we get
\begin{align*}
\int_{\bb R} \sup_{x, x' \in \bb P(\bb V)} 
\bb P^{1/q} \Big( t + \sigma (g_m \cdots g_1, x) \in [a, b],   D_{x, x', m}^c    \Big) dt 
\leq c m^{-\frac{3}{2q}}.  
\end{align*}
Therefore, we obtain
\begin{align}\label{bound-I-22-x-m-t}
\int_{\bb R}  \sup_{x, x' \in \bb P(\bb V): x \neq x'}  \frac{I_{22}(x, x', m, t) }{d(x, x')^{\gamma}}  dt
\leq c m^{- \eta} +  c m^{-\frac{3}{2q}}  \leq  c' m^{- \eta}. 
\end{align} 
The proposition is established by combining \eqref{decom-Integral-sup-001-aa},
\eqref{bound-I-1-x-m-t}, \eqref{bound-I-21-x-m-t} and \eqref{bound-I-22-x-m-t}. 
\end{proof}

\subsection{A technical version of Theorem \ref{Thm-CLLT-cocycle}} \label{sec a technical version}

Now we present a version of Theorem \ref{Thm-CLLT-cocycle} which is stated in terms of upper and lower bounds. 
The proof is involved and paricularly relies on Theorem \ref{Cor-CLLT-cocycle-001}, Theorem \ref{t-A 001}, 
Lemma \ref{lemma-duality-for-product-001}, 
Proposition \ref{Lem_HolderNormPsi}, 
 Theorem \ref{Thm-CCLT-limit} as well as preparatory results in \cite{GQX24a}. 

\begin{proposition} \label{t-BB001}
Assume that $\Gamma_{\mu}$ is proximal and strongly irreducible, 
the measure $\mu$ admits an exponential moment and the Lyapunov exponent $\lambda_{\mu}$ is zero. 
Then, there exists $\gamma_0>0$ such that, 
for any $\gamma \in (0, \gamma_0)$, $\ee \in (0, \frac{1}{8})$ and $t \in \bb R$,  
and for any non-negative  function $F$ 
and non-negative $\gamma$-regular compactly supported functions $G, H$ 
satisfying $H \leq_{\ee} F \leq_{\ee} G$, 
we have that, uniformly in $x \in \bb P(\bb V)$, 
\begin{align}\label{eqt-BB001}
& \limsup_{n \to \infty} n^{3/2} \bb E \Big[ F \Big(g_n \cdots g_1 x,  t + \sigma(g_n \cdots g_1, x) \Big);  \tau_{x, t} > n -1 \Big]  \nonumber\\
& \qquad\qquad\qquad\qquad \leq   \frac{2 V(x, t)}{ \sqrt{2\pi} \upsilon_{\mu}^3 }  
  \int_{\bb P(\bb V) \times \bb R} G (x',t') \rho(dx',dt')  
\end{align}
and
\begin{align} \label{eqt-BB002}
&  \liminf_{n \to \infty}  n^{3/2} \bb E \Big[ F \Big( g_n \cdots g_1 x,  t + \sigma(g_n \cdots g_1, x) \Big);  \tau_{x, t} > n -1 \Big]  \nonumber\\
& \qquad\qquad\qquad\qquad \geq   \frac{2 V(x, t)}{ \sqrt{2\pi} \upsilon_{\mu}^3}  
  \int_{\bb P(\bb V) \times \bb R} H (x',t' + 2\ee) \rho(dx',dt'). 
\end{align}
\end{proposition}

\begin{proof}
We first prove \eqref{eqt-BB001}.  
As in \eqref{def-Psi-x-t-001}, denote, for $n \geq 2$, $x \in \bb P(\bb V)$ and $t \in \bb R$,  
\begin{align*}
\Psi_n (x,t) =  \bb E \Big[  F \Big( g_n \cdots g_1 x, t + \sigma (g_n \cdots g_1, x) \Big);  \tau_{x, t} > n -1    \Big]. 
\end{align*}
Set $m=\left[ n/2 \right]$ and $k = n-m.$   
By the Markov property, we have that for any $x \in \bb P(\bb V)$ and $t \in \bb R$, 
\begin{align*}
\Psi_n (x,t) = \bb E \Big[ \Psi_m \Big( g_k \cdots g_1 x, t + \sigma (g_k \cdots g_1, x) \Big);  \tau_{x, t} > k  \Big]. 
\end{align*}
For any $x' \in \bb P(\bb V)$ and $t' \in \bb R$, 
we set 
\begin{align*} 
\overline \Psi_m(x', t')   :  = 
\bb E \bigg[ G \Big( g_m \cdots g_1 x', t' + \sigma(g_m \cdots g_1, x') \Big)   
 \times \chi_{\ee}  \Big( t' + \ee + \min_{1 \leq  j \leq m-1} \sigma(g_j \cdots g_1, x')  \Big)  \bigg], 
\end{align*}
where $\chi_{\ee}$ is defined by \eqref{Def_chiee}. 
By using $F \leq_{\ee} G$, we get that $\Psi_m \leq_{\ee}  \overline \Psi_m$.
Note that, by Proposition \ref{Lem_HolderNormPsi}, 
the function $\overline \Psi_m$ belongs to the space $\scr H_{\gamma}$,
so that we are exactly in the setting of Theorem \ref{t-A 001}.  
Therefore, applying the upper bound \eqref{eqt-A 001} of Theorem \ref{t-A 001}, we get that, uniformly in $x \in \bb P(\bb V)$, 
\begin{align*}
\Psi_n (x,t) 
& \leq  \frac{2 V(x, t)}{ \sqrt{2\pi} \upsilon_{\mu}^2 k } 
  \int_{\bb P(\bb V)}  \int_{\mathbb R}  \overline \Psi_m (x', t')
\phi^+ \bigg( \frac{t'}{\upsilon_{\mu} \sqrt{k}} \bigg) dt' \nu(dx')   \nonumber\\
& \quad + \frac{c}{k} \left( \ee^{1/4}  +  \ee^{-1/4} k^{-\eta} \right)  \| \overline \Psi_m \|_{\nu \otimes \Leb }  
    +   \frac{c_{\ee}}{ k^{3/2} }  \|  \overline \Psi_m \|_{\scr H_{\gamma}}   \notag\\
&  =: J_1 + J_2 + J_3. 
\end{align*}

We shall apply the duality lemma (Lemma \ref{lemma-duality-for-product-001}), Theorem \ref{Thm-CCLT-limit}
and results in \cite{GQX24a} to handle $J_1$. 
Note that $\chi_{\ee}  ( t' + \ee + \min_{1 \leq  j \leq m-1} \sigma(g_j \cdots g_1, x') ) 
\leq \mathds 1_{\{\tau_{x', t' + 2\ee} > m - 1\}}$. 
Using \eqref{duality id-002} of Lemma \ref{lemma-duality-for-product-001} 
and integrating over $y \in \bb P(\bb V^*)$ with respect to $\nu^*$, we deduce that
\begin{align*}
&  \int_{\bb P(\bb V)}  \int_{\mathbb R}  \overline \Psi_m(x', t')
\phi^+ \bigg( \frac{t'}{\upsilon_{\mu} \sqrt{k}} \bigg) dt' \nu(dx')  \\
& \leq   \int_{\bb P(\bb V)} \int_{\bb R}
\bb E \bigg[ G \Big( g_m \cdots g_1 x', t' + \sigma(g_m \cdots g_1, x') \Big); \tau_{x', t' + 2\ee} > m - 1 \bigg]   
  \phi^+ \bigg( \frac{t'}{\upsilon_{\mu} \sqrt{k}} \bigg)  dt' \nu(dx') \notag\\
  & =  \int_{\bb P(\bb V)} \int_{\bb P(\bb V^*)} \int_{\bb R} 
  \bb E \bigg[ G \left( g_1 \cdots g_m x',  u - 2\ee \right) 
   \phi^+ \bigg( \frac{u + \tilde S^{x,m}_{m}(\omega, y) - 2\ee }{\upsilon_{\mu} \sqrt{k}} \bigg);  
   \tau^{\mathfrak f^{x, m}}_{y, u} > m-1 \bigg]  \notag\\
 & \qquad\qquad\qquad\qquad\qquad    du \,  \nu^*(dy) \nu(dx'), 
\end{align*}
where $\tilde S^{x,m}_{m}(\omega, y)$ and $\tau^{\mathfrak f^{x, m}}_{y, u}$ 
are defined by \eqref{convention-S-xn-y} and \eqref{def-stop time with preturb-001-bb}, respectively. 
At this point, we assume that $G(x, t) = \varphi(x) \psi(t)$ for $x \in \bb P(\bb V)$ and $t \in \bb R$, 
where $\varphi$ is a non-negative $\gamma$-H\"older continuous function on $\bb P(\bb V)$
and $\psi$ is a compactly supported non-negative Lipschitz continuous function on $\bb R$. 
By Corollary 2.3 and Proposition 2.4 in \cite{GQX24a}, 
and the strong approximation result (6.5) from \cite{GLP17},  the assumptions of Theorem \ref{Thm-CCLT-limit} are satisfied
by the sequence of perturbations $\mathfrak f^{x, m} = (f^{x, m}_k)_{0 \leq k \leq m}$
and the twist function $\omega = (g_1, g_2, \ldots) \mapsto \varphi(g_1\cdots g_m x)$. 
Therefore, applying Theorem 4.1 of \cite{GQX24a} and Theorem \ref{Thm-CCLT-limit}, 
we obtain that, uniformly in $x' \in \bb P(\bb V)$, 
\begin{align*}
& \lim_{m\to\infty} \Bigg( \frac{ \upsilon_{\mu} \sqrt{2 \pi m} }{ 2 } 
\int_{\bb P(\bb V^*)} \int_{\bb R} 
  \bb E \bigg[ G \left( g_1 \cdots g_m x',  u - 2\ee \right)   \notag \\  
&\qquad\qquad\qquad\qquad  \times 
\phi^+ \bigg( \frac{u + \tilde S^{x,m}_{m}(\omega, y) - 2\ee }{\upsilon_{\mu} \sqrt{k}} \bigg);
   \tau^{\mathfrak f^{x, m}}_{y, u} > m-1 \bigg]   du  \nu^*(dy) \notag\\
& \qquad\qquad  -    \int_{\bb R} \psi(u-2\ee) U^{x',m,\varphi}_{[m/2]}(u) du  \int_{\bb R_+} (\phi^+(u'))^2  du' \Bigg) = 0,   
\end{align*}
where $U^{x',m,\varphi}_{[m/2]}$ is defined by \eqref{def of U^xm_n with phi-001}. 
Using Proposition 3.1 and Corollary 3.6 of \cite{GQX24a}, we get that, uniformly in $x' \in \bb P(\bb V)$, 
\begin{align*}
\lim_{m\to\infty} \int_{\bb R} \psi(u-2\ee) U^{x',m,\varphi}_{[m/2]}(u) du  
& = \lim_{m\to\infty} \int_{\bb R} \psi(u -2 \ee) U^{\varphi}_{[m/2]}(u)  du \notag\\
& = \int_{\bb P(\bb V) \times \bb R}  \varphi(x) \psi(u-2\ee) \rho(dx,du) \notag\\
& = \int_{\bb P(\bb V) \times \bb R}  G(x, u-2\ee) \rho(dx,du),  
\end{align*}
where $U^{\varphi}_{[n/2]}$ is defined by (3.3) in \cite{GQX24a}, 
and $\rho$ is a Radon measure on $\bb P(\bb V) \times \bb R$, as given in \eqref{exist of measure rho-001}. 
Note that, by \eqref{Rayleigh law-001}, we have $\int_{\bb R_+} (\phi^+(u'))^2  du' = \frac{\sqrt{\pi}}{4}$.
Therefore, we conclude that, when the function $G$ has the product form $(x, t) \mapsto \varphi(x) \psi(t)$, 
uniformly in $x' \in \bb P(\bb V)$, 
\begin{align*}
& \lim_{m\to\infty}  \frac{ \upsilon_{\mu} \sqrt{2 \pi m} }{ 2 } 
\int_{\bb P(\bb V^*)} \int_{\bb R} 
  \bb E \bigg[  G \left( g_1 \cdots g_m x',  u - 2\ee \right)  
   \notag\\
  & \qquad\qquad\qquad\quad \times 
  \phi^+ \bigg( \frac{u + \tilde S^{x,m}_{m}(\omega, y) - 2\ee }{\upsilon_{\mu} \sqrt{k}} \bigg);
   \tau^{\mathfrak f^{x, m}}_{y, u} > m-1 \bigg]   du  \,  \nu^*(dy) \notag\\
& =  \frac{\sqrt{\pi}}{4}  \int_{\bb P(\bb V) \times \bb R}  G(x, u-2\ee) \rho(dx,du).  
\end{align*}
Using again Corollary 4.3 of  \cite{GQX24a}, 
 the extension to any continuous compactly supported function $G$ on $\bb P(\bb V) \times \bb R$
 is made by standard approximation arguments. Thus, we obtain that, uniformly in $x' \in \bb P(\bb V)$, 
\begin{align*}
\lim_{n \to \infty} n^{3/2} J_1
=  \frac{2 V(x, t)}{ \sqrt{2\pi}  \upsilon_{\mu}^3}  
  \int_{\bb P(\bb V) \times \bb R}  G(x, u-2\ee) \rho(dx,du). 
\end{align*}

For $J_2$, noting that $\chi_{\ee}  ( t' + \ee + \min_{1 \leq  j \leq m-1} \sigma(g_j \cdots g_1, x') ) 
\leq \mathds 1_{\{\tau_{x', t' + 2\ee} > m - 1\}}$
and applying Theorem \ref{Cor-CLLT-cocycle-001}, we get that there exists a constant $c>0$ such that for any $m \geq 1$,  
\begin{align*} 
 \| \overline \Psi_m \|_{\nu \otimes \Leb }  
& \leq \int_{\bb P(\bb V)} \int_{\bb R} \bb E \Big[ G \Big( g_m \cdots g_1 x', t' + \sigma(g_m \cdots g_1, x') \Big);  
 \tau_{x', t' + 2\ee} > m-1  \Big] dt' \nu(dx')
\notag\\
& \leq \frac{c}{\sqrt{m}}. 
\end{align*}
Therefore, taking into account that $m=[n/2]$ and $k=n-m, $ we obtain $\limsup_{n \to \infty} n^{3/2} J_2   \leq c \ee^{1/4}.$

For $J_3$, by Proposition \ref{Lem_HolderNormPsi}, there exist constants $\beta, c >0$ such that
$\| \overline \Psi_{m,\ee} \|_{\scr H_{\gamma}} \leq \frac{c}{\ee m^{\beta}}.$
Thus, we have $\lim_{n \to \infty} n^{3/2} J_3 = 0$. 

Consequently, combining the above estiamtes, we obtain that, uniformly in $x \in \bb P(\bb V)$, 
\begin{align*}
& \limsup_{n \to \infty} n^{3/2} \bb E \Big[ F \Big( g_n \cdots g_1 x,  t + \sigma(g_n \cdots g_1, x) \Big);  \tau_{x, t} > n -1 \Big]  \nonumber\\
& \qquad\qquad\qquad \leq   \frac{2 V(x, t)}{ \sqrt{2\pi} \upsilon_{\mu}^3 }  
  \int_{\bb P(\bb V) \times \bb R} G (x',t' - 2 \ee) \rho(dx',dt') + c \ee^{1/4}. 
\end{align*}
Applying the above with $\ee' \in (0, \ee)$ instead of $\ee$,
and taking $\ee' \to 0$ yields the upper bound \eqref{eqt-BB001}. 
The proof of the lower bound \eqref{eqt-BB002} can be carried out in the same way. 
\end{proof}

\subsection{Proofs of Theorem \ref{Thm-CLLT-cocycle} and Corollary \ref{Thm-CLLT-cocycle-002}}

From Proposition \ref{t-BB001}, we obtain Theorem \ref{Thm-CLLT-cocycle} through a standard approximation procedure. 
\begin{lemma}\label{Lem_Approximation_FGH}
Fix $\gamma \in (0,1)$. 
Let $F$ be a non-negative continuous compactly supported function on $\bb P(\bb V) \times \bb R$.
Then, there exist a decreasing sequence $(G_k)_{k \geq 1}$ and an increasing sequence $(H_k)_{k \geq 1}$
of compactly supported $\gamma$-regular functions
such that $H_k \leq_{1/k} F \leq_{1/k} G_k$ for any $k \geq 1$, and 
$G_k$ and $H_k$ converge uniformly to $F$ as $k \to \infty$. 
\end{lemma}

The proof of this lemma is standard and therefore will not be detailed here.

\begin{proof}[Proof of Theorem \ref{Thm-CLLT-cocycle}]
This follows directly from Proposition \ref{t-BB001} and Lemma \ref{Lem_Approximation_FGH}.  
\end{proof}

\begin{proof}[Proof of Corollary \ref{Thm-CLLT-cocycle-002}]
Let $\varphi$ be a non-negative continuous function on $\bb P(\bb V)$.  Let $t\in \bb R$.
On the one hand, due to Corollary \ref{Coro-CLLT-cocycle-001}, 
for any $a\geq 0$, we have that, uniformly in $x \in \bb P(\bb V)$, 
\begin{align}\label{lower-bound-pf-Thm1}
&\liminf_{n\rightarrow\infty}n^{3/2} \bb E \Big( \varphi(g_n\cdots g_1 x); \tau_{x,t}=n \Big) \notag\\
&\geq \lim_{n\rightarrow\infty} n^{3/2} \bb E \Big(\varphi(g_n\cdots g_1 x); t+\sigma(g_n\cdots g_1,x)\in [-a,0), \tau_{x,t}=n \Big) \notag\\
&= \frac{2 V(x, t)}{ \sqrt{2 \pi} \upsilon_{\mu}^3 } \int_{\mathbb P(\mathbb V) \times [-a,0)} \varphi(x') \rho(dx',dt').
\end{align}
Since this is true for all $a\geq 0$, we obtain the following lower bound: uniformly in $x \in \bb P(\bb V)$, 
\begin{align}\label{lower-bound-Thm1-00a}
 \liminf_{n\rightarrow\infty}n^{3/2}\mathbb E \Big( \varphi(g_n\cdots g_1 x); \tau_{x,t}=n \Big) 
 \geq \frac{2 V(x,t)}{\sqrt{2\pi}\upsilon_\mu^3}
\int_{\mathbb P(\mathbb V) \times (-\infty,0)} \varphi(x') \rho(dx',dt').
\end{align}

To conclude the proof of the upper bound, we first show that, for any $\ee>0$, there exists $a>0$ such that, for any $x$ in $\mathbb P(\mathbb V)$ and $t\in \mathbb R$, 
we have
\begin{align}\label{bound-n-32-ee}
\mathbb P \Big( t+\sigma(g_n\cdots g_1,x) \in (-\infty,-a],\tau_{x,t}=n \Big)
\leq \frac{c \ee}{n^{3/2}} (1+\max\{t,0\}).
\end{align}
Indeed, for $n \geq 1$, $x \in \bb P(\bb V)$ and $t \in \bb R$, 
denote by $F_{n,x,t}$ the function defined for $s$ in $\mathbb R_+$: 
\begin{align*}
F_{n,x,t}(s) & : = \mathbb P \Big( t+\sigma(g_n\cdots g_1,x)\in [0, s], \tau_{x,t}>n \Big)  \notag\\
& \leq c n^{-3/2}(1+s)^2(1 + \max\{t,0\}),
\end{align*}
where the latter inequality holds due to Theorem \ref{Thm-CLLT-cocycle-bound-001}.
We also set, still for $x \in \bb P(\bb V)$ and $s\geq 0$, 
\begin{align*}
H_{x}(s): = \mathbb P\Big( \sigma(g_1,x)\leq -s \Big)\leq ce^{-\alpha s}, 
\end{align*}
where in the last inequality we used the exponential moment assumption \eqref{Exponential-moment}. 
Therefore, we derive that, by the Markov property, and then by integration by parts,
\begin{align*}
&\mathbb P \Big( t+\sigma(g_n\cdots g_1,x)\in (-\infty,-a], \tau_{x,t}=n \Big) \\
&=\mathbb E \Big( H_{g_{n-1}\cdots g_1 x} \big( a+t+\sigma(g_{n-1}\cdots g_1,x) \big); \tau_{x,t}>n-1 \Big) \\
&\leq c \, \mathbb E \Big[ \exp \Big( -\alpha(a+t+\sigma(g_{n-1}\cdots g_1,x)) \Big); \tau_{x,t}>n-1 \Big]\\
&= c\int_{[0,\infty)} e^{-\alpha (a+s)}dF_{n-1,x,t}(s)\\
&= c \alpha \int_0^\infty e^{-\alpha(a+s)}F_{n-1,x,t}(s) ds\\
&\leq cn^{-3/2}(1+\max\{t, 0\}) \int_0^\infty e^{-\alpha(a+s)}(1+s)^2 ds
\end{align*}
and the inequality \eqref{bound-n-32-ee} follows.

Now, we conclude the argument by showing that, uniformly in $x \in \bb P(\bb V)$, 
\begin{align}\label{upper-bound-Thm1-00a}
 \limsup_{n\rightarrow\infty} n^{3/2} \mathbb E \Big( \varphi(g_n\cdots g_1 x); \tau_{x,t}=n \Big) 
 \leq 
\frac{2 V(x,t)}{\sqrt{2\pi}\upsilon_\mu^3}\int_{\mathbb P(\mathbb V) \times (-\infty,0)} \varphi(x') \rho(dx',dt').
\end{align}
We fix $\ee>0$ and choose $a >0$ as above. Then, using \eqref{lower-bound-pf-Thm1} and \eqref{bound-n-32-ee}, 
we obtain that, uniformly in $x \in \bb P(\bb V)$,
\begin{align*} 
&\limsup_{n\rightarrow\infty}n^{3/2}\mathbb E \Big( \varphi(g_n\cdots g_1 x); \tau_{x,t}=n \Big) \\
&\leq \lim_{n\rightarrow\infty}n^{3/2} \mathbb E \Big( \varphi(g_n\cdots g_1 x);  t+\sigma(g_n\cdots g_1,x)\in [-a,0),\tau_{x,t}=n) \Big)\\
&\quad + \| \varphi \|_{\infty} \limsup_{n\rightarrow\infty}n^{3/2}\mathbb P \Big( t+\sigma(g_n\cdots g_1,x) \in (-\infty,-a),\tau_{x,t}=n \Big)\\
&\leq \frac{2 V(x,t) }{\sqrt{2\pi}\upsilon_\mu^3}
\int_{\mathbb P(\mathbb V) \times [-a,0)} \varphi(x') \rho(dx',dt') 
+  c \ee (1+\max\{t,0\}) \| \varphi \|_{\infty}\\
&\leq  \frac{2 V(x,t)}{\sqrt{2\pi}\upsilon_\mu^3}
\int_{\mathbb P(\mathbb V) \times (-\infty,0)} \varphi(x') \rho(dx',dt')
+   c \ee (1+\max\{t,0\}) \| \varphi \|_{\infty}.
\end{align*}
Since this inequality holds for all $\ee>0$, we get \eqref{upper-bound-Thm1-00a}. 
Combining \eqref{lower-bound-Thm1-00a} and \eqref{upper-bound-Thm1-00a} 
concludes the proof of Corollary \ref{Thm-CLLT-cocycle-002}. 
\end{proof}

\appendix

\section{Random walks with perturbations depending on the future}\label{sec invar func}

In this section, we establish several conditioned limit theorems 
within the abstract framework which we introduced in our preliminary paper \cite[Sections 4-7]{GQX24a}.
This framework enables us to analyze the behavior of random walks with perturbations that depend on the future, 
providing a comprehensive understanding of their limiting behavior. 
The results established in this Appendix are utilized 
 in the proofs of Theorems \ref{Thm-CLLT-cocycle}, \ref{Cor-CLLT-cocycle-001} and \ref{Thm-CLLT-cocycle-bound-001}
through the reversal Lemma \ref{lemma-duality-for-product-001}.

\subsection{Setting and assumptions}\label{subsec set assump}
We now briefly recall the setting of random walks with perturbations from Section 4 in \cite{GQX24a}.
Assume that $\bb G$ is a general second countable locally compact group and $\mu $ is a Borel probability measure on $\bb G$.
We endow the measurable space $\Omega= \bb G^{\bb N^*}$ with the product measure $\bb P = \mu^{\otimes \bb N^*}$.
The sequence $g_1,g_2,\ldots$ represents the coordinate maps on $\Omega$, 
forming a sequence of  independent identically distributed elements of $\bb G$ with law  $\mu$.
The expectation corresponding to $\bb P$ is denoted by $\bb E$. 
We fix a second countable locally compact space $\bb X$ equipped with a continuous action of the group $\bb G$,
along with a $\mu$-stationary Borel probability measure $\nu$ on $\bb X$. 
Let $\sigma: \bb G \times \bb X \to \bb R$ be a continuous cocycle satisfying that, 
for any $g_1, g_2 \in \bb G$ and $x\in \bb X$,
\begin{align*} 
\sigma(g_2 g_1, x) = \sigma(g_2,  g_1 x) + \sigma(g_1, x). 
\end{align*}
Assume that $\sigma$ has an exponential moment with respect to $\mu$: there exists $\alpha>0$ such that 
\begin{align} \label{exp mom for f 001}
\int_{\bb G} e^ {\alpha \sup_{x\in \bb X} | \sigma (g,x) |} \mu (dg) < \infty. 
\end{align}
Moreover, assume that $\sigma$ is centered in the following strong sense: for any $x\in \bb X$,
\begin{align} \label{centering-001}
\int_{\bb G} \sigma (g,x) \mu (dg) = 0. 
\end{align}
For any $x\in \bb X$, consider the random walk defined by
\begin{align*} 
S^x_n: = \sigma(g_n  \cdots  g_1, x )= \sum_{i=1}^n \sigma\left(g_i, g_{i-1}\cdots g_1 x\right), \quad n\geq 1.
\end{align*}
The space $\Omega\times \bb X$ is endowed with the shift map $T$ defined as follows: 
for $x\in \bb X$ and  $\omega = (g_1, g_2, \ldots) =(g_{i})_{i \geq 1} \in \Omega$,
\begin{align*}
T(\omega,x): = ((g_2, g_3, \ldots), g_1x) = ((g_{i+1})_{i \geq 1}, g_1x).
\end{align*}
We will deal with the random walk $S_n^x$ with perturbations that depend on the future. 
Specifically,  given a sequence $\mathfrak f = (f_n)_{n \geq 0}$ of 
measurable functions $f_n: \Omega\times \bb X \to \bb R$, we will study 
the properties of the perturbed random walk
\begin{align} \label{the perturbed random walk 001}
S^x_n(\omega) + f_n \circ T^n(\omega, x) - f_0(\omega, x), \quad n\geq 1.  
\end{align}
As in \cite{GQX24a}, the challenge arises from the fact that the functions 
$f_n$ may depend on the whole sequence $\omega=(g_{i})_{i \geq 1}$, including the future coordinates $(g_{i})_{ i \geq n }$.
To manage this dependence, we will introduce a finite-dimensional approximation property, which asserts that 
functions depending on infinitely many coordinates 
can be effectively approximated by functions relying on finitely many coordinates.
This strategy is one of the crucial aspects of the paper's approach. 

First, we assume that the sequence $\mathfrak f = (f_n)_{n \geq 0}$ of 
measurable functions $f_n: \Omega\times \bb X \to \bb R$ admits a uniform exponential moment: 
there exists a constant $\alpha > 0$ such that 
\begin{align} \label{exp mom for g 002} 
C_{\alpha}(\mathfrak f) 
= \sup_{n \in \bb N} \int_{\bb X} \int_{\Omega} e^{ \alpha |f_n(\omega,x)| } \bb P(d\omega)  \nu(dx)
< \infty.
\end{align}
For any integer $p\geq 0$, let $\mathscr A_p$ be the $\sigma$-algebra on $\Omega$ 
spanned by the coordinate functions $\omega \in \Omega  \mapsto \omega_{i} \in \bb G$ for $1\leq i \leq p$. 
The finite-dimensional approximation property of the sequence $\mathfrak f$ is stated as follows: 
 there are constants $\alpha, \beta >0$ such that 
\begin{align} \label{approxim rate for gp-002} 
D_{\alpha,\beta}(\mathfrak f) 
 = \sup_{n \in \bb N} \sup_{p \in \bb N}   e^{\beta p} 
\left( \int_{\bb X}  \int_{\Omega}  
 e^{\alpha |f_n(\omega, x) - \bb E (f_n(\cdot, x) | \scr A_p)(\omega) |} \bb P(d\omega) \nu(dx)  - 1\right)  
  <\infty.  
\end{align}
We have already seen that this property is satisfied in the case of the random walks with preturbations considered
in Subsections \ref{sec proof T1.6}, \ref{sec proof of Th1.8} and \ref{sec a technical version}.

We will make the following strong approximation assumption for the process $(S^x_{[nt]})_{t \in [0, 1]}$:
up to enlarging the probability space $(\Omega, \scr F, \bb P)$, 
there exist a standard Wiener process $(B_t)_{t \geq 0}$ on $(\Omega, \scr F, \bb P)$ and 
a constant $\bf{v} > 0$ with the following property. 
For any $\ee \in (0, \frac{1}{4})$, there exists $c_{\ee} > 0$ such that, for any $x\in \mathbb{X}$ and $n\geq 1$,  
\begin{align}\label{KMTbound001}
\mathbb{P} \bigg( \sup_{0\leq t\leq 1} \left\vert S^x_{\left[ nt\right] } - \bf{v} B_{nt}\right\vert >  n^{1/2 - \ee } \bigg) 
\leq c_{\ee}  n^{- \ee}. 
\end{align}

\subsection{Main results}\label{subsec invar func fixed}
In order to state our main results, 
we need to introduce some additional notation.
For $x\in \bb X$ and $t\in \bb R$, denote by $\tau^{\mathfrak f}_{x,t}$ the first time
when the perturbed random walk $( t + S^x_{k} + f_k\circ T^k -f_0 ) _{k\geq 1}$ defined by \eqref{the perturbed random walk 001}
 exits $\mathbb{R}_{+}= [ 0,\infty),$ 
\begin{align} \label{def-stop time with preturb-001}
\tau_{x,t}^{\mathfrak f} = \min \left\{ k\geq 1: t+S^x_{k} + f_k( T^k(\omega,x)) - f_0(\omega,x) < 0\right\}. 
\end{align}
For $n \geq 1$ and $t \in \bb R$, let
\begin{align}\label{def-U-varphi-001}
U^{\mathfrak f}_n(t) 
= \int_{\bb X} \bb E \left( t + S^{x}_n + f_n(T^n (\omega,x))- f_0(\omega,x);  \tau^{\mathfrak f}_{x,t} > n \right) \nu(dx). 
\end{align}
In the preliminary paper \cite{GQX24a}, we have established 
asymptotic bounds for $U^{\mathfrak f}_n(t)$. 
Actually we also gave a control over a twisted version of this function to be defined below.
For an integrable and non-negative function $\theta$ on $\Omega$, 
 denote the essential supremum as $\|\theta\|_{\infty} = {\rm esssup} |\theta|$. 
Assume that $\theta$ satisfies the following condition: 
there exists a constant $\beta >0$ such that 
\begin{align} \label{approx property of theta-001}
N_{\beta}(\theta) : = \sup_{p \geq 0} e^{\beta p} \bb E | \theta -  \bb E(\theta | \scr A_p) | < \infty. 
\end{align}
In analogy with \eqref{def-U-varphi-001},
for $t \in \bb R$ and $n \geq 1$, we set
\begin{align}\label{def-U-f-n-t-theta-001}
U^{\mathfrak f, \theta}_n(t)
= \int_{\bb X} \bb E \left((t + S^{x}_n + f_n(T^n (\omega,x))- f_0(\omega,x))  \theta(\omega) ;  \tau^{\mathfrak f}_{x,t} > n \right) \nu(dx). 
\end{align}
It follows from \cite[Corollary 4.2]{GQX24a} 
that there exists an increasing function $U^{\mathfrak f, \theta}: \bb R \to \bb R_+$ such that,
for any continuous compactly supported function $\psi$ on $\bb R$, we have 
\begin{align}\label{def-U-f-theta-001}
\lim_{n \to \infty} \int_{\bb R} U^{\mathfrak f, \theta}_n(t) \psi(t) dt = \int_{\bb R} U^{\mathfrak f, \theta}(t) \psi(t) dt. 
\end{align} 
Our goal in this section is to utilize the function $U^{\mathfrak f, \theta}$ to prove the following conditioned limit theorems. 
Specifically, when $\theta=1$, our first theorem provides an estimate for the persistence probability $\bb P(\tau_{x, t}^{\mathfrak f} >n)$. 

\begin{theorem}\label{Thm-CCLT-limit-000}
Suppose that the cocycle $\sigma$ admits finite exponential moments \eqref{exp mom for f 001}
and is centered \eqref{centering-001}. 
We also suppose that the strong approximation property \eqref{KMTbound001}  is satisfied.
Fix $\beta > 0$, $B \geq 1$ and a continuous compactly supported test function $\varphi$ on $\bb R$. 
Then, uniformly over the sequence $\mathfrak f = (f_n)_{n \geq 0}$ of  measurable functions on $\Omega \times \bb X$
satisfying the moment condition \eqref{exp mom for g 002} and the approximation property \eqref{approxim rate for gp-002}
with $C_{\alpha}(\mathfrak f) \leq B$ and $D_{\alpha,\beta}(\mathfrak f) \leq B$, 
and uniformly for non-negative function 
$\theta \in L^{\infty}(\Omega, \bb P)$ with $\|\theta\|_{\infty} \leq B$ and $N_{\beta}(\theta) \leq B$, 
we have
\begin{align*}
\lim_{n\to\infty} \frac{ \bf{v} \sqrt{2 \pi n} }{ 2 } 
\int_{\bb X}  \int_{\bb R} \varphi(t)
\bb E \left( \theta(\omega);   \tau_{x, t}^{\mathfrak f} >n   \right)  \nu(dx) dt  
=   \int_{\bb R} \varphi(t) U^{\mathfrak f, \theta}(t) dt. 
\end{align*}
\end{theorem}

We also obtain the following bound, which holds uniformly for all $n \geq 1$ and $t\in \bb R$,
and which does not follow from the previous asymptotic result.  

\begin{theorem}\label{Thm-inte-exit-time-001}
Suppose that the cocycle $\sigma$ admits finite exponential moments \eqref{exp mom for f 001}
and is centered \eqref{centering-001}. 
We also suppose that the strong approximation property \eqref{KMTbound001}  is satisfied.
Fix $\beta > 0$ and $B \geq 1$. 
Then, there exists a constant $c>0$ such that for any $t \in \bb R$,
uniformly over the sequence $\mathfrak f = (f_n)_{n \geq 0}$ of measurable functions on $\Omega \times \bb X$
satisfying the moment condition \eqref{exp mom for g 002} and the approximation property \eqref{approxim rate for gp-002}
with $C_{\alpha}(\mathfrak f) \leq B$ and $D_{\alpha,\beta}(\mathfrak f) \leq B$, 
we have
\begin{align*}
\int_{\bb X}   \bb P \left(   \tau_{x, t}^{\mathfrak f} >n   \right)  \nu(dx) 
\leq c \frac{1 + \max\{t, 0\}}{\sqrt{n}}. 
\end{align*}
\end{theorem}

When $\theta=1$, our final result in this section 
is a central limit theorem for the process $t + \sigma(g_n \cdots g_1, x)$ conditioned on the event $\{ \tau_{x, t}^{\mathfrak f} >n \}$.
Recall that $\phi^+$ is the standard Rayleigh density function on $\bb R$, as defined by \eqref{Rayleigh law-001}. 

\begin{theorem}\label{Thm-CCLT-limit}
Suppose that the cocycle $\sigma$ admits finite exponential moments \eqref{exp mom for f 001}
and is centered \eqref{centering-001}. 
We also suppose that the strong approximation property \eqref{KMTbound001}  is satisfied.
Fix $\beta > 0$, $B \geq 1$, $b_0 >0$ and a continuous compactly supported test function $\varphi$ on $\bb R$. 
Then, uniformly for $b \leq b_0$, and for a sequence $\mathfrak f = (f_n)_{n \geq 0}$ of  measurable functions on $\Omega \times \bb X$
satisfying the moment condition \eqref{exp mom for g 002} and the approximation property \eqref{approxim rate for gp-002}
with $C_{\alpha}(\mathfrak f) \leq B$ and $D_{\alpha,\beta}(\mathfrak f) \leq B$; 
and uniformly for non-negative function $\theta \in L^{\infty}(\Omega, \bb P)$ with $\|\theta\|_{\infty} \leq B$ and $N_{\beta}(\theta) \leq B$, 
we have
\begin{align*}
 \lim_{n\to\infty} \frac{ \bf{v} \sqrt{2 \pi n} }{ 2 } 
\int_{\bb X}  \int_{\bb R} \varphi(t) 
& \bb E \left( \theta(\omega);  \frac{t + \sigma(g_n \cdots g_1, x)}{ \bf{v} \sqrt{n}} \leq b,  \tau_{x, t}^{\mathfrak f} >n   \right) 
\nu(dx) dt
\nonumber \\
& \qquad 
=   \int_{\bb R} \varphi(t) U^{\mathfrak f, \theta}(t) dt  \int_{0}^{\max\{b, 0\}} \phi^+(t')  dt'. 
\end{align*}
\end{theorem}

As in \cite[Sections 4, 6 and 7]{GQX24a}, 
we will first establish these theorems for perturbations that depend on finitely many coordinates, 
utilizing the theory of Markov chains developed in \cite{GQX24a}. 
The proofs will be modeled on \cite{GLP17}, for which we need to make the statements effective. 
The general case will be addressed through an approximation procedure. 

\subsection{Perturbation depending on finitely many coordinates}\label{susec: perturb finite-001}

In this subsection, we focus on the scenario where 
the perturbations $\mathfrak f$ depend only on finitely many coordinates. 
We recall the relevant concepts introduced for this case, as detailed in \cite[Section 6.1]{GQX24a}. 

Assume that $\mathfrak f=(f_n)_{n\geq0}$ is a sequence of measurable functions
$f_n: \bb G^{\{1,\ldots,p \}}\times \bb X \to \bb R$. 
In particular, for any $n\geq 0$ and  $x\in \bb X$,  
the function $f_{n}(\cdot, x)$ 
depends only on the first $p$ coordinates $(g_{1},\ldots, g_{p})$. 
This implies that $f_{n}(\cdot, x)$ is 
$\mathscr A_p$ measurable, where 
 $\mathscr A_p$ 
is the $\sigma$-algebra generated by $g_{1},\ldots,g_{p}$. 
We use the same notation $f_n$ to denote 
the function on $\Omega \times \bb X$
defined by $(\omega,x) \mapsto f_n(g_1,\ldots, g_p,x)$.
A similar remark applies to the function $(g_0,\ldots, g_p,x) \mapsto f_n(g_1,\ldots, g_p,x)$ 
defined on $\bb G^{\{0,\ldots,p \}} \times \bb X$.

We now define a Markov chain.
For any $p \geq 0$, set $\bb A_{p} =\bb G^{\{0,\ldots,p \}} \times \bb X \times \bb N$. 
Consider the family of transition probabilities $\{ P_a: a \in \bb A_p \}$ defined as follows: for any $a=(g_0,\ldots,g_{p},x,q)\in \bb A_p$
and any nonnegative measurable function $\varphi: \bb A_p \to \bb R$, 
\begin{align*} 
P_a (\varphi) = \int_{\bb G} \varphi \left( g_{1},\ldots, g_{p}, g, g_1 x, q+1 \right) \mu(dg).  
\end{align*}
For any $a\in \bb A_p$, we write $\bb P_a$ for the probability measure 
on $\Omega'_p = \bb A_p^{\bb N}$ associated with the transition probability $P_a$.
The expectation with respect to $\bb P_a$ is denoted by $\bb E_a$.
 For any 
 $\omega' \in \bb A_p^{\bb N}$ and $i \geq 0$, 
 we denote by $\xi_i(\omega')$ the $i$-th coordinate map of $\omega'$ on $\bb A_p^{\bb N}$.
%
%
Consider the $\sigma$-algebras $\mathscr G_n= \sigma \{ \xi_0,\ddd,\xi_{n} \}$ for $n\geq 0$.
Additionally, we define the function $\sigma_p:  \bb A_p \to  \bb R$ by setting, for any $a=(g_0,\ldots,g_{p},x,q) \in \bb A_p$, 
\begin{align}\label{def-sigma-p-00a}
\sigma_p (g_0, \ldots, g_{p}, x,q) = \sigma(g_0,g_0^{-1}x)=-\sigma(g_0^{-1},x) . 
\end{align}

With the definitions introduced above, we can express $\sigma (g_n \cdots g_1,x)$ as an additive functional
of the Markov chain $(\xi_i)_{i \geq 0}$ as follows:   
for any $a=(g_0,\ldots,g_{p},x,q) \in \bb A_p$ and any $n\geq 1$,
for $\bb P_a$-almost every $\xi=(\xi_0,\xi_1,\ldots)\in \Omega'_p = \bb A_p^{\bb N}$, we have
\begin{align} \label{sigma-Gn-001}
\sigma (g_n \cdots g_1,x) = S_n^x := \sum_{i=1}^{n} \sigma_p(\xi_i),  
\end{align}
where, for  $j>p$, the element $g_j$ is the unique element of $\bb G$ such that 
\begin{align*} 
\xi_{j-p} = (g_{j-p},\ldots, g_j, g_{j-p} \cdots g_1 x, q+ j-p).
\end{align*}
Define a measurable perturbation function $\tilde f$ on $\bb A_p$ by setting, for any $a=(g_0,\ldots,g_{p},x,q) \in \bb A_p$, 
\begin{align*} 
\tilde f(a) = f_q(g_1,\ldots, g_p,x).
\end{align*}
This function captures the dependence of the perturbation 
on the specified finite number of coordinates in the Markov chain framework.
For any $k \in \bb N$ and $a \in \bb A_p$, we set 
\begin{align}\label{def-Fka}
\mathcal F_k(a) = \bb E_a e^{\alpha |\tilde f(\xi_k) |},
\end{align}
where $\alpha>0$ is a constant given in \eqref{exp mom for g 002}. 
It follows from \eqref{exp mom for g 002} that for any $q \in \bb N$, 
we have $\mathcal F_k(g_0, \ldots, g_p, x, q) < \infty$ for $\bb P \otimes \nu$-almost all $g_0, \ldots, g_p, x$. 
 For any $t \in \bb R$, 
 define the function 
$\tilde \tau^{\mathfrak f}_{t}: \bb A_p^{\bb N} \to \bb N \cup \{\infty \} $ by setting 
\begin{align}\label{def-tau-f-y}
 \tilde\tau^{\mathfrak f}_{t} = \min \left\{ k\geq 1:  t + \sum_{i=1}^{k} \sigma_p(\xi_i) + \tilde f(\xi_k)-\tilde f(\xi_0) < 0 \right\}.
\end{align}
It is straightforward to see that $\tilde\tau^{\mathfrak f}_{t}$ is a $(\mathscr G_n)_{n\geq 0}$-stopping time on the space $\bb A_p^{\bb N}$. 

\subsection{Comparison of preharmonic functions}
As a preliminary step in the proofs of Theorems \ref{Thm-CCLT-limit-000}, \ref{Thm-inte-exit-time-001} and \ref{Thm-CCLT-limit},
we investigate the relationships between functions that generalize the function $U^{\mathfrak f}_{n}(t)$ from \eqref{def-U-varphi-001},
where the deterministic integer $n$ is replaced by a stopping time.


Let $p\geq 1$. We still assume that the perturbation $\mathfrak f = (f_n)_{n \geq 0}$ forms a sequence of $\scr A_p$-measurable functions.
Let $\eta: \bb A_p^{\bb N} \to \bb N^*$ be a bounded stopping time, meaning that for any $n \geq 1$, the set $\{ \eta \leq n \}$ belongs to the $\sigma$-algebra $\mathscr G_n$ generated by the coordinate maps $\{\xi_0, \ldots, \xi_n\}$. 
For any  $t\in \bb R$, denote 
\begin{align}\label{def-U-f-n-t-001}
U^{\mathfrak f}_{\eta}(t) 
= \int_{\bb A_p} \mathbb{E}_{a} 
\bigg( t + \sum_{i=1}^{\eta} \sigma_p(\xi_i) + \tilde f(\xi_{\eta}) - \tilde f(\xi_0);   \tilde{\tau}_{t}^{\mathfrak f} > \eta \bigg)
\mu_p(da), 
\end{align}
where $\sigma_p$ is defined by \eqref{def-sigma-p-00a}, and for $a = (g_0, \ldots, g_p, x,q)$, we have denoted 
\begin{align}\label{def-mu-p-da}
\mu_p(da) = \mathds 1_{\{q=0\}} \mu(dg_0) \ldots \mu(dg_p) \nu(dx).  
\end{align}
More generally, if $\theta$ is a non-negative and bounded $\mathscr A_p$-measurable function on $\Omega$, 
for $a = (g_0, \ldots, g_p, x,q) \in \bb A_p$, setting
\begin{align}\label{def-tilde-theta}
\tilde \theta(a) = \theta(g_1, \ldots, g_p), 
\end{align} 
we can extend our definition \eqref{def-U-f-n-t-001} to include this function, leading to: 
\begin{align}\label{def-U-f-n-t-002}
U^{\mathfrak f, \theta}_{\eta}(t) 
= \int_{\bb A_p} \tilde \theta(a) \mathbb{E}_{a} 
\bigg( t + \sum_{i=1}^{\eta} \sigma_p(\xi_i) + \tilde f(\xi_{\eta}) - \tilde f(\xi_0);   \tilde{\tau}_{t}^{\mathfrak f} > \eta \bigg)
\mu_p(da). 
\end{align}
This formulation incorporates the influence of the function $\tilde \theta$
 on the expectation, allowing for a broader analysis of the expectations under consideration. 
 
The goal of this subsection is to establish the following monotonicity property for the map $\eta \mapsto U^{\mathfrak f, \theta}_{\eta}$. 
This property indicates that increasing the stopping time leads to an increase in the expected value of the perturbed process. 

\begin{proposition}\label{Prop-U-f-t-increase-001}
Suppose that the cocycle $\sigma$ admits finite exponential moments \eqref{exp mom for f 001}
and is centered \eqref{centering-001}.
Let $1 < \gamma < 8$. 
There exist constants $\beta, c > 0$ with the following property. 
Let $ n \geq 1$, $p\leq n^{\beta}$, $t\in \bb R$, 
$\mathfrak f = (f_n)_{n \geq 0}$ be a sequence of $\scr A_p$-measurable functions, 
 $\eta_1, \eta_2: \bb A_p^{\bb N} \to \bb N^*$ be bounded stopping times 
 with $n \leq \eta_1 \leq \eta_2 \leq n^{\gamma}$, 
and  $\theta$ be a non-negative and bounded $\mathscr A_p$-measurable function on $\Omega$. Then, we have 
\begin{align*}
U^{\mathfrak f, \theta}_{\eta_1}(t) \leq U^{\mathfrak f, \theta}_{\eta_2}(t) 
+ \frac{c}{n^{\beta}}   \left( 1 + \max \{t,0\} \right)  
\|\theta\|_{\infty} C_{\alpha}(\mathfrak f). 
\end{align*}
\end{proposition}

The proof of Proposition \ref{Prop-U-f-t-increase-001}
 will essentially follows from the optional stopping time theory, which we will apply to a certain martingale 
defined through the Markov chain on the space $\bb A_p$.
To draw this connection, 
we introduce the following auxiliary function. Let $\eta: \bb A_p^{\bb N} \to \bb N^*$ be a bounded stopping time. 
For $a  \in \bb A_p$ and $t \in \bb R$, define
\begin{align} \label{EXPECT-E_x-001}
W^{\mathfrak f}_{\eta}(a, t)
= \bb E_a \bigg( t + \sum_{i=1}^{\eta + p} \sigma_p(\xi_i); \tilde \tau^{\mathfrak f}_{t} > \eta \bigg). 
\end{align}

\begin{lemma}\label{Lem-Optional-thm-W-001}
Suppose that the cocycle $\sigma$ admits finite exponential moments \eqref{exp mom for f 001}
and is centered \eqref{centering-001}.
Let $p\geq 1$. Let $\mathfrak f = (f_n)_{n \geq 0}$ be a sequence of $\scr A_p$-measurable functions 
and $\eta: \bb A_p^{\bb N} \to \bb N^*$ be a bounded stopping time. 
Then, for any $a = (g_0, \ldots, g_p, x,q) \in \bb A_p$ and  $t \in \bb R$, we have
\begin{align*}
W^{\mathfrak f}_{\eta}(a, t)
= t + \sigma(g_p \cdots g_1, x)
 - \bb E_a \bigg( t + \sum_{i=1}^{\tilde \tau^{\mathfrak f}_{t} + p} \sigma_p(\xi_i); \tilde \tau^{\mathfrak f}_{t} \leq \eta \bigg). 
\end{align*}
\end{lemma}

\begin{proof}
As $\eta$ is bounded and the sequence $(\sum_{i =p+1}^{p+n} \sigma_p (\xi_i))_{n \geq 0}$ is a martingale with respect to the probability measure $\bb P_a$, 
we have 
\begin{align*}
\bb E_a \bigg( t + \sum_{i=1}^{\eta + p} \sigma_p(\xi_i)  \bigg) = t + \sigma(g_p \cdots g_1, x). 
\end{align*}
Therefore, 
\begin{align*}
W^{\mathfrak f}_{\eta}(a, t)
= t + \sigma(g_p \cdots g_1, x)
 - \bb E_a \bigg( t + \sum_{i=1}^{\eta + p} \sigma_p(\xi_i); \tilde \tau^{\mathfrak f}_{t} \leq \eta \bigg). 
\end{align*}
The conclusion follows from the optional stopping theorem. 
\end{proof}

From Lemma \ref{Lem-Optional-thm-W-001}, we get the following corollary. 

\begin{corollary}\label{Coro-Optional-thm-W-002}
Suppose that the cocycle $\sigma$ admits finite exponential moments \eqref{exp mom for f 001}
and is centered \eqref{centering-001}.
Let $p\geq 1$. Let $\mathfrak f = (f_n)_{n \geq 0}$ be a sequence of $\scr A_p$-measurable functions
satisfying the moment assumption \eqref{exp mom for g 002}.
Let  
 $\eta_1, \eta_2: \bb A_p^{\bb N} \to \bb N^*$ be bounded stopping times with $\eta_1 \leq \eta_2$. 
Then, for any $a  \in \bb A_p$ and $t \in \bb R$, we have 
\begin{align*}
W^{\mathfrak f}_{\eta_1}(a, t) \leq W^{\mathfrak f}_{\eta_2}(a, t) 
+  \bb E_a \left( \tilde f(a) - \tilde f (  \xi_{\tilde \tau^{\mathfrak f}_{t}}); \eta_1 < \tilde \tau^{\mathfrak f}_{t} \leq \eta_2 \right). 
\end{align*}
\end{corollary}

\begin{proof}
By Lemma \ref{Lem-Optional-thm-W-001}, 
for $a = (g_0, \ldots, g_p, x,q) \in \bb A_p$ and $t \in \bb R$, we have
\begin{align*}
W^{\mathfrak f}_{\eta_2}(a, t) - W^{\mathfrak f}_{\eta_1}(a, t)
& = - \bb E_a \bigg( t + \sum_{i=1}^{\tilde \tau^{\mathfrak f}_{t} + p} \sigma_p(\xi_i); \eta_1 < \tilde \tau^{\mathfrak f}_{t} \leq \eta_2 \bigg) \notag\\
& = - \bb E_a \bigg( t + \sum_{i=1}^{\tilde \tau^{\mathfrak f}_{t} + p} \sigma_p(\xi_i) + \tilde f (\xi_{\tilde \tau^{\mathfrak f}_{t}}) - \tilde f(a); \eta_1 < \tilde \tau^{\mathfrak f}_{t} \leq \eta_2 \bigg) \notag\\ 
& \quad + \bb E_a \left(  \tilde f (\xi_{\tilde \tau^{\mathfrak f}_{t}}) - \tilde f(a); \eta_1 < \tilde \tau^{\mathfrak f}_{t} \leq \eta_2 \right) \notag\\
& \geq  \bb E_a \left(  \tilde f (\xi_{\tilde \tau^{\mathfrak f}_{t}}) - \tilde f(a); \eta_1 < \tilde \tau^{\mathfrak f}_{t} \leq \eta_2 \right), 
\end{align*}
where the last inequality holds by the definition of $\tilde \tau^{\mathfrak f}_{t}$ (cf.\  \eqref{def-tau-f-y}). 
The result follows. 
\end{proof}

The next lemma is proved in \cite[Lemma 6.3]{GQX24a}.
Recall that $\mathcal F_{\ell}(a)$ is defined by \eqref{def-Fka}. 

\begin{lemma}\label{Lem-tau-prior}
Suppose that the cocycle $\sigma$ admits finite exponential moments \eqref{exp mom for f 001}
and is centered \eqref{centering-001}.
Then, there exist constants $\beta, c>0$ 
such that for any $a= (g_0,\ldots,g_{p},x,q) \in \bb A_p$, $t \in \bb R$, $n\geq 1$ and $1\leq p < n $,
and any sequence
$\mathfrak f = (f_n)_{n \geq 0}$ of $\scr A_p$-measurable functions,   
\begin{align*}
\mathbb{P}_a \left( \tilde \tau^{\mathfrak f}_{t} \geq n \right)  
 \leq
c \frac{\max \{t,0\} +  |\sigma(g_p\cdots g_1,x)|   +|\tilde{f}(a)|+ \log n}{(n-p)^{\beta}}  
 + \frac{c }{n^8} \sum_{k = p+1}^n \mathcal F_k(a).   
\end{align*}
\end{lemma}

\begin{lemma} \label{Lemma martingale-001}
Suppose that the cocycle $\sigma$ admits finite exponential moments \eqref{exp mom for f 001}
and is centered \eqref{centering-001}.
Let $1 < \gamma < 8$. 
Then, there exist constants $\beta, c > 0$ such that for any $ n \geq 1$, $p\leq n^{\beta}$, 
$a=(g_0,\ldots,g_{p},x, q) \in \bb A_p$, $t\in \bb R$,
 any bounded stopping time $\eta: \bb A_p^{\bb N} \to \bb N^*$ with $\eta \in [n, n^{\gamma}]$, 
and any sequence
$\mathfrak f = (f_n)_{n \geq 0}$ of $\scr A_p$-measurable functions,   
\begin{align}\label{sum-hp-nu-nt-001}
&\bb E_a \bigg( \bigg| \sum_{i= \eta -p+1}^{\eta} h_p(\xi_i) \bigg|; \; \tilde \tau^{\mathfrak f}_{t} > \eta  \bigg)  \notag\\
&\leq 
\frac{c}{n^{\beta/2} } \bigg( \max \{t,0\} + | \sigma(g_p\cdots g_1, x)| 
 + |\tilde f(a)| + \frac{1}{n} \sum_{k = p+1}^n \mathcal F_k(a) \bigg)  
\end{align}
and
\begin{align}\label{sum-hp-nu-nt-002}
 \bb E_a \left( |\tilde f(\xi_{\eta})|; \; \tilde \tau^{\mathfrak f}_{t} > \eta  \right) 
 \leq 
\frac{c}{n^{\beta/2}}   \left( 1 + \max \{t,0\} + | \sigma(g_p\cdots g_1, x)| + |\tilde f(a)|  \right)  
 \frac{1}{n^{\gamma} }  \sum_{k = p+1}^{n^{\gamma}} \mathcal F_k(a). 
\end{align}
\end{lemma}

\begin{proof} 
We first prove \eqref{sum-hp-nu-nt-001}. 
By H\"older's inequality, with $s >1$ and $\frac{1}{s} +\frac{1}{s'}=1$,
\begin{align} \label{app-lem-10-001}
 \bb E_a \bigg( \bigg| \sum_{i= \eta -p+1}^{\eta} h_p(\xi_i) \bigg|; \; \tilde \tau^{\mathfrak f}_{t} > \eta  \bigg)
 \leq 
  \bb E_a^{1/s} \bigg( \bigg| \sum_{i= \eta -p+1}^{\eta} h_p(\xi_i) \bigg|^s  \bigg) 
   \bb P_a^{1/s'} \left(  \tilde \tau^{\mathfrak f}_{t} > \eta  \right).
\end{align}
For the first expectation on the right-hand side, we have
\begin{align*} 
\bb E_a \bigg( \bigg| \sum_{i= \eta -p+1}^{\eta} h_p(\xi_i) \bigg|^s  \bigg)
= \sum_{\ell = n}^{n^{\gamma}} 
\bb E_a \bigg( \bigg| \sum_{i=\ell-p+1}^{\ell} h_p(\xi_i)  \bigg|^s;  \eta =\ell  \bigg)
\leq  c p^s n^{\gamma}. 
\end{align*}
Substituting this into \eqref{app-lem-10-001} and using Lemma \ref{Lem-tau-prior}, 
we get that there exist constants $c, \beta, \ee > 0$  such that, 
for any $a \in \bb A_p$, $t \in \bb R$, $n\geq 1$ and $p\leq n^{\ee} $,  
\begin{align*} 
& \bb E_a \bigg( \bigg| \sum_{i= \eta -p+1}^{\eta} h_p(\xi_i) \bigg|; \; \tilde \tau^{\mathfrak f}_{t} > \eta  \bigg) \\
 &\leq  c \left( p^{s} n^{\gamma}  \right)^{1/s}  \bb P_a^{1/s'} \left(  \tilde \tau^{\mathfrak f}_{t} > n  \right)\\
 &\leq   c   p n^{\gamma/s}    
 \bigg(  \frac{\max \{t,0\} +  |\sigma(g_p\cdots g_1,x)|   +|\tilde{f}(a)|+ \log n}{(n-p)^{\beta}}
 + \frac{1}{n^8} \sum_{k = p+1}^n \mathcal F_k(a) \bigg)^{1/s'}   \notag\\
 &\leq  
 \frac{c}{n^{\beta/2} } \bigg( \max \{t,0\} + | \sigma(g_p\cdots g_1, x)| + |\tilde f(a)| + \frac{1}{n} \sum_{k = p+1}^n \mathcal F_k(a) \bigg), 
\end{align*}
where $\mathcal F_{\ell}(a)$ is defined by \eqref{def-Fka}, 
and in the last inequality we choose $\beta > 2 \ee$ and take $s' > 1$ to be sufficiently close to $1$.
This concludes the proof of \eqref{sum-hp-nu-nt-001}. 

We next prove \eqref{sum-hp-nu-nt-002}. 
By H\"older's inequality, with $s >1$ and $\frac{1}{s} +\frac{1}{s'}=1$,
\begin{align} \label{app-lem-second-ine001}
 \bb E_a \left( |\tilde f(\xi_{\eta})|; \; \tilde \tau^{\mathfrak f}_{t} > \eta  \right) 
 \leq   \bb E_a^{1/s} \left( |\tilde f(\xi_{\eta})|^{s}   \right)
   \bb P_a^{1/s'} \left(  \tilde \tau^{\mathfrak f}_{t} > \eta  \right).
\end{align}
In view of \eqref{def-Fka}, we write, with $c_s > 0$ depending on $s$, 
\begin{align*} 
\bb E_a \left( |\tilde f(\xi_{\eta})|^{s}   \right)
= \sum_{\ell =n}^{n^{\gamma}} 
\bb E_a \left( |\tilde f(\xi_{\ell})|^s;  \eta = \ell  \right)
\leq  c_s \sum_{\ell =n}^{n^{\gamma}}  \mathcal F_{\ell}(a). 
\end{align*}
Substituting this into \eqref{app-lem-second-ine001}, using Lemma \ref{Lem-tau-prior}
and taking into account the fact that $\eta \in [n, n^{\gamma}]$, we get
\begin{align*} 
&  \bb E_a \left( |\tilde f(\xi_{\eta})|; \; \tilde \tau^{\mathfrak f}_{t} > \eta  \right)  
\leq   \bigg(  c_s \sum_{\ell =n}^{n^{\gamma}}  \mathcal F_{\ell}(a)  \bigg)^{1/s}  
  \bb P_a^{1/s'} \left(  \tilde \tau^{\mathfrak f}_{t} > \eta  \right)  \notag\\
 &\leq  c_s'  \bigg(  \sum_{\ell =n}^{n^{\gamma}}  \mathcal F_{\ell}(a)  \bigg)^{1/s}  
  \bigg(  \frac{\max \{t,0\} +  |\sigma(g_p\cdots g_1,x)|   +|\tilde{f}(a)|+ \log n}{(n-p)^{\beta}}
 + \frac{1}{n^8} \sum_{k = p+1}^n \mathcal F_k(a) \bigg)^{1/s'}. 
 \end{align*}
By taking $s'$ to be sufficiently close to $1$, we get
\begin{align*} 
 \bb E_a \left( |\tilde f(\xi_{\eta})|; \; \tilde \tau^{\mathfrak f}_{t} > \eta  \right)   
   \leq  
 \frac{c}{n^{\beta/2}}   \left( 1 + \max \{t,0\} + | \sigma(g_p\cdots g_1, x)| + |\tilde f(a)|  \right)  
 \frac{1}{n^{\gamma} }  \sum_{k = p+1}^{n^{\gamma}} \mathcal F_k(a). 
\end{align*}
This concludes the proof of the second inequality \eqref{sum-hp-nu-nt-002}.
\end{proof}

\begin{lemma} \label{Lemma martingale-002}
Suppose that the cocycle $\sigma$ admits finite exponential moments \eqref{exp mom for f 001} and is centered \eqref{centering-001}.
Let $1 < \gamma < 8$. 
Then, there exist constants $\beta, c > 0$ such that 
for any $ n \geq 1$, $p\leq n^{\beta}$, $a=(g_0,\ldots,g_{p},x, q) \in \bb A_p$, $t\in \bb R$, 
and any sequence $\mathfrak f = (f_n)_{n \geq 0}$ of $\scr A_p$-measurable functions,   
\begin{align*}
& \bb E_a \left( |\tilde f(\xi_{ \tilde \tau^{\mathfrak f}_{t} })|; \; n \leq  \tilde \tau^{\mathfrak f}_{t}  \leq n^{\gamma}  \right)  \notag\\
& \leq 
\frac{c}{n^{\beta / 2}}   \left( 1 + \max \{t,0\} + | \sigma(g_p\cdots g_1, x)| + |\tilde f(a)|  \right)  
 \frac{1}{n^{\gamma} }  \sum_{k = p+1}^{n^{\gamma}} \mathcal F_k(a). 
\end{align*}
\end{lemma}

\begin{proof}
By H\"older's inequality, with $s >1$ and $\frac{1}{s} +\frac{1}{s'}=1$,
\begin{align*} 
 \bb E_a \left( |\tilde f(\xi_{ \tilde \tau^{\mathfrak f}_{t} })|; \; n \leq  \tilde \tau^{\mathfrak f}_{t}  \leq n^{\gamma}  \right)
 \leq 
 \bb E_a^{1/s} \left( |\tilde f(\xi_{ \tilde \tau^{\mathfrak f}_{t} })|^{s}; \tilde \tau^{\mathfrak f}_{t}  \leq n^{\gamma}  \right)
   \bb P_a^{1/s'} \left(  \tilde \tau^{\mathfrak f}_{t} \geq  n  \right).
\end{align*}
We write, with $c_s > 0$ depending on $s$, 
\begin{align*} 
\bb E_a \left( |\tilde f(\xi_{ \tilde \tau^{\mathfrak f}_{t} })|^{s}; \tilde \tau^{\mathfrak f}_{t}  \leq n^{\gamma}  \right)
= \sum_{\ell = 1}^{n^{\gamma}} 
\bb E_a \left( |\tilde f(\xi_{\ell})|^s;  \tilde \tau^{\mathfrak f}_{t} = \ell  \right)
\leq  c_s \sum_{\ell =n}^{n^{\gamma}}  \mathcal F_{\ell}(a). 
\end{align*}
The remainder of the proof can be completed in a manner similar to that of \eqref{sum-hp-nu-nt-002}.  
\end{proof}

\begin{proof}[Proof of Proposition \ref{Prop-U-f-t-increase-001}]
By \eqref{def-U-f-n-t-001}, \eqref{def-U-f-n-t-002} and Corollary \ref{Coro-Optional-thm-W-002}, we get 
\begin{align*}
U^{\mathfrak f, \theta}_{\eta_1}(t)
& \leq U^{\mathfrak f, \theta}_{\eta_2}(t)  
 +   \int_{\bb A_p} \tilde \theta(a) \mathbb{E}_{a} 
\bigg( - \sum_{i=\eta_1 + 1}^{\eta_1 + p} \sigma_p(\xi_i) 
 + \tilde f(\xi_{\eta_1}) - \tilde f(\xi_0);   \tilde{\tau}_{t}^{\mathfrak f} > \eta_1 \bigg)
\mu_p(da)  \notag\\
& \quad -   \int_{\bb A_p}  \tilde \theta(a)  \mathbb{E}_{a} 
\bigg( - \sum_{i=\eta_2 + 1}^{\eta_2 + p} \sigma_p(\xi_i) 
 + \tilde f(\xi_{\eta_2}) - \tilde f(\xi_0);   \tilde{\tau}_{t}^{\mathfrak f} > \eta_2 \bigg)
\mu_p(da) \notag\\
& \quad +   \int_{\bb A_p}  \tilde \theta(a) \mathbb{E}_{a} 
 \left(  \tilde f(\xi_0) - \tilde f (   \xi_{\tilde \tau^{\mathfrak f}_{t}}); \eta_1 < \tilde \tau^{\mathfrak f}_{t} \leq \eta_2 \right)
\mu_p(da)  \notag\\
& = U^{\mathfrak f, \theta}_{\eta_2}(t) 
+  \int_{\bb A_p} \tilde \theta(a) \mathbb{E}_{a} 
\bigg( - \sum_{i=\eta_1 + 1}^{\eta_1 + p} \sigma_p(\xi_i) 
 + \tilde f(\xi_{\eta_1});   \tilde{\tau}_{t}^{\mathfrak f} > \eta_1 \bigg)
\mu_p(da)  \notag\\
& \quad -  \int_{\bb A_p} \tilde \theta(a)  \mathbb{E}_{a} 
\bigg( - \sum_{i=\eta_2 + 1}^{\eta_2 + p} \sigma_p(\xi_i) 
 + \tilde f(\xi_{\eta_2});   \tilde{\tau}_{t}^{\mathfrak f} > \eta_2 \bigg)
\mu_p(da) \notag\\
& \quad -  \int_{\bb A_p} \tilde \theta(a) \mathbb{E}_{a} 
 \left(  \tilde f (   \xi_{\tilde \tau^{\mathfrak f}_{t}}); \eta_1 < \tilde \tau^{\mathfrak f}_{t} \leq \eta_2 \right)
\mu_p(da). 
\end{align*}
Now the conclusion follows from Lemmas \ref{Lemma martingale-001} and \ref{Lemma martingale-002},
by integrating over $a \in \bb A_p$ with respect to $\nu_p$ and replacing $\beta/4$ by $\beta$.  
\end{proof}


\subsection{Estimates of the persistence probability}

We continue to assume that the sequence of perturbations $\mathfrak f = (f_n)_{n \geq 0}$
depends only on finitely many coordinates, and we keep the notations introduced in Subsection \ref{susec: perturb finite-001}.
In this context, Theorem \ref{Thm-CCLT-limit-000} will be derived from the following effective result. 
Recall that $U_{n}^{\mathfrak f, \theta}$ is defined by \eqref{def-U-f-n-t-theta-001}. 
 

\begin{proposition}\label{Prop-CCLT-001-a}
Suppose that the cocycle $\sigma$ admits finite exponential moments \eqref{exp mom for f 001} and is centered \eqref{centering-001}.
Let $\ee \in (0, \frac{1}{8})$. 
Then, there exist constants $\beta, c >0$ such that for any $n \geq 1$, $1 \leq p \leq n^{\ee}$, $|t| \leq n^{\ee}$, 
any sequence $\mathfrak f = (f_n)_{n \geq 0}$ of $\scr A_p$-measurable functions, 
and any non-negative and bounded $\mathscr A_p$-measurable function $\theta$ on $\Omega$, 
\begin{align}\label{CCLT-001-a}
 \int_{\bb X} \bb E \left( \theta(\omega);   \tau_{x, t}^{\mathfrak f} > n \right) \nu(dx)  
  \leq  \frac{2}{  \sqrt{2\pi n} \bf{v} }  
  U_{n}^{\mathfrak f, \theta}(t) 
  +  \frac{c (1 + \max \{t,0\}) }{n^{1/2 + \beta}}   \|\theta\|_{\infty} C_{\alpha}(\mathfrak f)
\end{align}
and 
\begin{align}\label{CCLT-002-a}
\int_{\bb X} \bb E \left( \theta(\omega);   \tau_{x, t}^{\mathfrak f} > n \right) \nu(dx)  
 \geq  \frac{2}{  \sqrt{2\pi n} \bf{v} }  
  U_{n^{1/2 - 2\ee}}^{\mathfrak f, \theta}(t) 
  -  \frac{c (1 + \max \{t,0\})}{n^{1/2 + \beta}}  \|\theta\|_{\infty}  C_{\alpha}(\mathfrak f). 
\end{align}
\end{proposition}

As a corollary of the proof of Proposition \ref{Prop-CCLT-001-a}, we obtain the following upper bound,
which implies Theorem \ref{Thm-inte-exit-time-001} for perturbations depending on finitely many coordinates. 

\begin{corollary}\label{Cor-CCLT-001-a}
Suppose that the cocycle $\sigma$ admits finite exponential moments \eqref{exp mom for f 001} and is centered \eqref{centering-001}. 
Let $\ee \in (0, \frac{1}{8})$. 
Then, there exists a constant $c >0$ such that for any $n \geq 1$, $1 \leq p \leq n^{\ee}$, $t \in \bb R$, 
any sequence $\mathfrak f = (f_n)_{n \geq 0}$ of $\scr A_p$-measurable functions, 
\begin{align}\label{Corollary-of-CCLT-001-a}
 \int_{\bb X} \bb P \left( \tau_{x, t}^{\mathfrak f} > n \right) \nu(dx)  
 \leq  \frac{c (1 + \max \{t,0\})}{\sqrt{n}} C_{\alpha}(\mathfrak f). 
\end{align}
\end{corollary}

The proof of Proposition \ref{Prop-CCLT-001-a} relies on the Markov property together with the 
following lemma for the asymptotic of persistence probability when the starting point $t$ is large.

\begin{lemma}\label{Lemma Weak Conv-a}
Suppose that the cocycle $\sigma$ admits finite exponential moments \eqref{exp mom for f 001} and is centered \eqref{centering-001}.
Let $\ee \in (0, \frac{1}{8})$. 
Then there exists a constant $c>0$ with the following property: for any $n \geq 1$, $1 \leq p \leq n^{\ee}$, 
$n^{1/2 - \ee} \leq t \leq n^{1/2}$, 
any sequence $\mathfrak f = (f_n)_{n \geq 0}$ of $\scr A_p$-measurable functions, 
and any $a = (g_0, \ldots, g_p, x,q) \in \mathbb{A}_p$ satisfying $|\sigma(g_i, g_{i-1} \cdots g_1 x)| \leq n^{\ee}$ for $1 \leq i \leq p$ and $\tilde f(a) \leq n^{\ee}$,
we have 
\begin{align*}
 \left\vert 
\mathbb{P}_{a} \left( \tilde{\tau}_{t}^{\mathfrak f} > n  \right)
 - \frac{2t}{  \sqrt{2\pi n} \bf{v} }    \right\vert  
 \leq  \frac{ct^2}{n} + c e^{-\alpha n^{\ee} } \sum_{i = 1}^{n} \mathcal{F}_i(a).  
\end{align*}
%
%
\end{lemma}

\begin{proof}
Let us fix $a = (g_0, \ldots, g_p, x,q) \in \mathbb{A}_p$, $n \geq 1$ and $t \geq n^{1/2-\ee}$.  
Denote $t^{+} = t  + \sigma(g_p \cdots g_1, x) - \tilde f(a) + n^{1/2-2\ee },$ 
$t^{-} = t + \sigma(g_p \cdots g_1, x) - \tilde f(a) -n^{1/2-2\ee },$ 
$b^{+}=b + 2 \bf{v} n^{-2\ee}$ and $b^{-} = b -  2\bf{v} n^{-2\ee }.$ 
Set 
\begin{align}\label{def-Ana-00a}
A_{n, a}  = \bigg\{ \sup_{0\leq s \leq 1} \bigg\vert \sum_{i=p +1}^{p+ [(n-p)s]} \sigma_p(\xi_i) 
  - \bf{v} B_{(n-p)s} \bigg\vert \leq  \frac{1}{2} n^{1/2-2\ee} \bigg\} 
 \bigcap \left\{ |\tilde f(\xi_i) | \leq n^{\ee}, \forall i \in [1, n] \right\}.
\end{align}
For $u \in \bb R$, define $\tau_{u}^{bm} = \inf \{ s \geq 0:  u  + B_s \leq 0 \}$, 
where $(B_s)_{s \geq 0}$ is the standard Brownian motion on $\bb R$ with $B_0 = 0$. 
Then, by \eqref{def-tau-f-y}, it holds that, $\bb P_a$-almost surely, for $n$ large enough, 
\begin{align*}
\left\{ \tilde{\tau}_{t}^{\mathfrak f} > n  \right\}  \cap A_{n, a}
\subseteq \left\{ \tau_{ \bf{v}^{-1} t^{+} }^{bm}> n  - p \right\}. 
\end{align*}
It follows that 
\begin{align}\label{WK001-a}
\mathbb{P}_{a} \left(   \tilde{\tau}_{t}^{\mathfrak f} >n  \right) 
& \leq  \bb P_a \left(    \tilde{\tau}_{t}^{\mathfrak f} >n , A_{n,a} \right)
 + \bb P_a (A_{n,a}^c)  \notag\\
&  \leq  \bb P \left(  \tau_{ \bf{v}^{-1} t^{+}}^{bm}> n  - p \right)
 + \bb P_a (A_{n,a}^c). 
\end{align}
In the same way, by \eqref{def-Ana-00a}, it holds that, $\bb P_a$-almost surely, 
\begin{align*}
\left\{ \tilde{\tau}_{t}^{\mathfrak f} > n  \right\} 
\supseteq   \left\{ \tau_{ \bf{v}^{-1} t^{-}}^{bm}> n  - p \right\} \cap  \{ \tilde{\tau}_t^{\mathfrak f} > p \} \cap A_{n, a}. 
\end{align*}
Besides, notice that on the set $A_{n, a}$, for $1 \leq i \leq p$, we have
\begin{align*}
t + \sigma(g_i \cdots g_1, x) + \tilde f(\xi_i) - \tilde f(a) \geq n^{1/2 - \ee} - (p+2) n^{\ee} > 0
\end{align*}
for $n$ large enough. Therefore, we have $\tilde{\tau}_t^{\mathfrak f} > p$ and  
\begin{align*}
\left\{ \tilde{\tau}_{t}^{\mathfrak f} > n  \right\} 
\supseteq   \left\{ \tau_{ \bf{v}^{-1} t^{-}}^{bm}> n  - p \right\}  \cap A_{n, a}, 
\end{align*}
so that 
\begin{align}\label{WK002-a}
\mathbb{P}_{a} \left(   \tilde{\tau}_{t}^{\mathfrak f} >n  \right)  
\geq  \bb P \left(  \tau_{\bf{v}^{-1} t^{-}}^{bm}> n  - p \right)    - \bb P_a (A_{n,a}^c).  
\end{align}
Now we deal with the first probability on the right-hand side of \eqref{WK001-a}. 
By L\'evy \cite{Levy37} (Theorem 42.I, pp. 194-195), we have 
\begin{align}\label{WK004-a}
\bb P \left(   \tau_{ \bf{v}^{-1} t^{+}}^{bm}> n  - p \right) 
 = \frac{2}{\sqrt{2\pi (n -p)}  } \int_{0}^{ \bf{v}^{-1} t^{+} } e^{- \frac{ s^{2}}{2  (n - p)   } } ds  
 \leq  \frac{2 t^{+} }{\sqrt{2\pi (n -p)} \bf{v}}. 
\end{align}
Since $p \leq n^{\ee}$ and $t \geq n^{1/2-\ee},$ we have, as $n\rightarrow \infty,$ 
\begin{align}\label{WK005-a}
(n-p) /n = 1- O\left( n^{-\ee }\right)  
\end{align}
and 
\begin{align}\label{WK006-a}
\frac{t^{+}}{t} = \frac{ t  + \sigma(g_p \cdots g_1, x) - \tilde f(a) + n^{1/2-2\ee } }{ t}
= 1 + O\left( n^{-\ee}\right).  
\end{align}
Therefore, as we assumed $t \geq n^{1/2 - \ee}$, 
combining \eqref{WK004-a}, \eqref{WK005-a} and \eqref{WK006-a}, 
we get an upper bound for the first probability on the right-hand side of \eqref{WK001-a}: 
\begin{align}\label{WK009-a}
\bb P \left(   \tau_{ \bf{v}^{-1} t^{+}}^{bm}> n  - p \right) 
\leq \frac{2t}{\sqrt{2\pi n} \bf{v} }  + \frac{ct}{n^{1/2 + \ee}} 
\leq  \frac{2t}{\sqrt{2\pi n} \bf{v} }  + \frac{ct^2}{n}.  
\end{align}
Next we provide a lower bound for the first probability on 
the right-hand side of \eqref{WK002-a}. 
Using again L\'evy \cite{Levy37} (Theorem 42.I, pp. 194-195), 
the inequality $|e^z - 1| \leq |z| e^{|z|}$ with $z = - \frac{ s^{2}}{2  (n - p) }$ and the fact that $t^- \leq t +  n^{3\ee} \leq c \sqrt{n}$, we get 
\begin{align}\label{lower-bound-proba-001}
 \bb P \left(  \tau_{\bf{v}^{-1} t^{-}}^{bm}> n  - p \right) 
& =  \frac{2}{\sqrt{2\pi (n -p)}  } \int_{0}^{ \bf{v}^{-1} t^{-} } e^{- \frac{ s^{2}}{2  (n - p)   } } ds  \notag\\ 
& \geq  \frac{2 t^-}{\sqrt{2\pi (n -p)} \bf{v} } -  \frac{2}{\sqrt{2\pi (n -p)}  }  
\int_{0}^{ \bf{v}^{-1} t^{-} }  \frac{ s^{2}}{2  (n - p)   }  e^{ \frac{ s^{2}}{2  (n - p)   } } ds \notag\\
& \geq  \frac{2 t^-}{\sqrt{2\pi (n -p)} \bf{v} } - c \bigg( \frac{t^-}{\sqrt{n - p}} \bigg)^3,  
\end{align}
where in the last inequality we used $|e^{ \frac{ s^{2}}{2  (n - p)   } }| \leq c$ for any $s \in [0, \bf{v}^{-1} t^{-}]$. 
Similarly to \eqref{WK006-a}, it holds that 
\begin{align}\label{lower-bound-proba-002}
\frac{t^{-}}{t}=\frac{ t  + \sigma(g_p \cdots g_1, x) - \tilde f(a) - n^{1/2-2\ee } }{ t}
=1+O\left( n^{-\ee} \right).  
\end{align}
Combining \eqref{lower-bound-proba-001} and \eqref{lower-bound-proba-002}, 
and taking into account that $p \leq n^{\ee}$ and $n^{1/2 - \ee} \leq t \leq n^{1/2}$, 
we get
\begin{align}\label{WK010-a}
\bb P \left(  \tau_{\bf{v}^{-1} t^{-}}^{bm}> n  - p \right)
\geq  \frac{2t}{\sqrt{2\pi n} \bf{v} }  -  \frac{ct}{n^{1/2 + \ee}}  - c \left( \frac{t}{\sqrt{n}} \right)^3
\geq \frac{2t}{\sqrt{2\pi n} \bf{v} }  - \frac{ct^2}{n}.  
\end{align}
By \eqref{def-Ana-00a}, \eqref{KMTbound001} and \eqref{def-Fka}, 
we have 
\begin{align}
\bb P_a (A_{n,a}^c) \leq c n^{-2 \ee} + \sum_{i = 1}^{n} \bb P_a (| \tilde f(\xi_i) | > n^{\ee} )
\leq  c n^{-2 \ee} + e^{-\alpha n^{\ee} } \sum_{i = 1}^{n} \mathcal{F}_i(a). 
\label{WK003-a}
\end{align}
From \eqref{WK001-a}, \eqref{WK002-a}, \eqref{WK009-a}, \eqref{WK010-a} and \eqref{WK003-a},  
we conclude the proof of the lemma. 
\end{proof}


We now proceed to deducing Proposition \ref{Prop-CCLT-001-a} from Lemma \ref{Lemma Weak Conv-a}.
To this aim, we will stop the random walk when it reaches a certain high level.
By adapting the approach developed in \cite[Section 7.1]{GQX24a}, 
for any $a \in \bb A_p$, $n \geq 1$ and $t\in \bb R$, we define a stopping time on $\bb A_p^{\bb N}$ by setting 
\begin{align} \label{nu n02}
\tilde \nu_{n,t} = \min \bigg\{ k\geq p+1: \bigg| t  + \sum_{i=1}^k \sigma_p(\xi_i) -\tilde f(\xi_0)  \bigg|  \geq 2 n^{1/2-\ee } \bigg\},  
\end{align}
where $\sigma_p$ is defined by \eqref{def-sigma-p-00a}. 
For $n \geq 1$ and $t \in \bb R$, we also define the following subset of $\bb A_p^{\bb N}$, 
which will be used in the proof of Proposition \ref{Prop-CCLT-001-a} as well as in Proposition \ref{Prop-CCLT-001}:  
\begin{align}\label{def-Bnt-001}
B_{n, t} 
& = \left\{ |\tilde f(\xi_{\tilde \nu_{n,t}})| \leq n^{\ee}, 
|\sigma_p(\xi_{\tilde \nu_{n,t} + i})| \leq n^{\ee}, i \in [1, p] \right\} \notag\\
&  \bigcap 
\bigg\{ t + \sum_{i=1}^{\tilde \nu_{n,t}} \sigma_p(\xi_i)  
 - \tilde f(\xi_0) \leq n^{(1-\ee)/2},  n^{1/2 - 2\ee} \leq  \tilde \nu_{n,t} \leq n^{1-\ee} \bigg\}. 
\end{align}
The next lemma shows that the integral of the probability of the complement of the set $B_{n,t}$ is small
with respect to the measure $\mu_p$ defined by \eqref{def-mu-p-da}. 

\begin{lemma}\label{Lem-Bnt-001}
Suppose that the cocycle $\sigma$ admits finite exponential moments \eqref{exp mom for f 001} and is centered \eqref{centering-001}.
Let $\ee \in (0, \frac{1}{8})$. 
Then, there exists a constant $c >0$ such that for any $n \geq 1$, $1 \leq p \leq n^{\ee}$, $|t| \leq n^{\ee}$, 
and any sequence $\mathfrak f = (f_n)_{n \geq 0}$ of $\scr A_p$-measurable functions, 
\begin{align*}
\int_{\bb A_p}  \bb P_a \big( B_{n, t}^c  \big) \mu_p(da) 
\leq  
c e^{- \frac{\alpha}{2} n^{\ee}} C_{\alpha}(\mathfrak f). 
\end{align*}
\end{lemma}

\begin{proof}
For $n\geq 1$, $t\in \bb R$ and $a \in \bb A_p$, define 
\begin{align*}
A_{n, t} = \left\{ \tilde \nu_{n,t} \leq n^{1-\ee} \right\}.
\end{align*}
By Lemma 7.3 of \cite{GQX24a}, there exists a constant $\beta >0$ 
such that for any $1 \leq p \leq n^{\ee}$,  
\begin{align}\label{inequa-Ant-000}
\mathbb{P}_a \big( A_{n, t}^c \big) 
\leq  2 \exp \left(- \beta n^{\ee}\right).
\end{align}
By the moment assumptions \eqref{exp mom for f 001} and \eqref{exp mom for g 002},  
still for $1 \leq p \leq n^{\ee}$, 
we get 
\begin{align}\label{inequa-Ant-001}
\bb P_a \left( \left\{ |\tilde f(\xi_{\tilde \nu_{n,t}})| \leq n^{\ee}, 
|\sigma_p(\xi_{\tilde \nu_{n,t} + i})| \leq n^{\ee}, i \in [1, p] \right\}^c 
 \cap A_{n, t}  \right)
 \leq c e^{- \frac{\alpha}{2} n^{\ee}} C_{\alpha}(\mathfrak f),  
\end{align}
for some constant $\alpha > 0$. 
For $n \geq 1$, define
\begin{align*}
F_n = \left\{ a = (g_0, \ldots, g_p, x,q) \in \bb A_p:  |\sigma(g_p \cdots g_1, x) - \tilde f(a) | \leq n^{\ee}  \right\}. 
\end{align*}
By the moment assumptions \eqref{exp mom for f 001} and \eqref{exp mom for g 002},  
still for $1 \leq p \leq n^{\ee}$, 
we have 
\begin{align}\label{inequa-mu-p-Fn}
\mu_p (F_n^c) \leq  c e^{- \frac{\alpha}{2} n^{\ee}} C_{\alpha}(\mathfrak f). 
\end{align}
For $a \in F_n$, by Lemma 7.4 of \cite{GQX24a},
if $p \leq n^{\ee}$ and $|t| \leq n^{\ee}$,  
\begin{align}\label{inequa-Ant-002}
\mathbb{P}_a \left( \tilde \nu_{n,t} \leq n^{1/2- 2\ee }  \right) 
\leq c \exp \left( -c_{\ee }n^{\ee/2}   \right).
\end{align}
Finally, we denote 
\begin{align*}
G_{n,t} = \bigg\{ t + \sum_{i=1}^{\tilde \nu_{n,t}} \sigma_p(\xi_i)  
 - \tilde f(\xi_0) \leq n^{(1-\ee)/2} \bigg\}.
\end{align*}
Notice that on the set $G_{n, t}^c \cap A_{n, t}$, by \eqref{nu n02}, we have 
\begin{align*}
 t + \sum_{i=1}^{\tilde \nu_{n,t}} \sigma_p(\xi_i)   - \tilde f(\xi_0) > n^{(1-\ee)/2}
 \quad  \mbox{and} \quad 
 t + \sum_{i=1}^{\tilde \nu_{n,t} - 1} \sigma_p(\xi_i)   - \tilde f(\xi_0) \leq 2 n^{1/2 - \ee}, 
\end{align*}
which implies that, for some constant $c>0$, 
\begin{align*}
 \sigma_p(\xi_{ \tilde \nu_{n,t} }) >  n^{(1-\ee)/2} - 2 n^{1/2 - \ee} > c n^{(1-\ee)/2}. 
\end{align*}
Therefore, as $\tilde \nu_{n,t} \leq n^{1-\ee}$ (by the definition of $A_{n, t}$), 
by the moment assumption \eqref{exp mom for f 001}, 
we obtain
\begin{align}\label{inequa-Ant-003}
\bb P_a \big( G_{n, t}^c \cap A_{n,t} \big) 
\leq \bb P_a \left( \sigma_p(\xi_{ \tilde \nu_{n,t} }) > c n^{(1-\ee)/2} \right)
\leq  n^{1 - \ee} e^{- c \alpha n^{(1-\ee)/2}} 
\leq c e^{- n^{\ee}}. 
\end{align}
Using \eqref{inequa-mu-p-Fn}, we derive that
\begin{align*}
 \int_{\bb A_p}  \bb P_a \big( B_{n, t}^c  \big) \mu_p(da)  
& \leq  \mu_p (F_n^c) + \int_{F_n}  \bb P_a \big( B_{n, t}^c  \big) \mu_p(da)   \notag\\
& \leq c e^{- \frac{\alpha}{2} n^{\ee}} C_{\alpha}(\mathfrak f) +  \int_{F_n}  \bb P_a \big( A_{n, t}^c \big) \mu_p(da)
+   \int_{F_n}  \bb P_a \big( A_{n, t} \cap B_{n, t}^c \big) \mu_p(da).  
\end{align*}
The conclusion follows from \eqref{inequa-Ant-000}, \eqref{inequa-Ant-001}, \eqref{inequa-Ant-002} and \eqref{inequa-Ant-003}. 
\end{proof}

\begin{proof}[Proof of Proposition \ref{Prop-CCLT-001-a}]
Note that, for any $n \geq 1$ and $t \in \bb R$, 
\begin{align*}
 \int_{\bb X}  \bb E \left( \theta(\omega);    \tau_{x, t}^{\mathfrak f} > n \right) \nu(dx)  
 = \int_{\bb A_p} \tilde \theta(a) \mathbb{P}_{a} 
\left( \tilde{\tau}_{t}^{\mathfrak f} > n \right)
\mu_p(da), 
\end{align*}
where $\tilde\tau^{\mathfrak f}_{t}$, $\mu_p$ and 
$\tilde \theta$ are defined by \eqref{def-tau-f-y}, \eqref{def-mu-p-da} and \eqref{def-tilde-theta}, respectively. 
Recall that $\tilde \nu_{n,t}$ and $B_{n,t}$ are defined by \eqref{nu n02} and \eqref{def-Bnt-001}, respectively. 
Note that for any $k \geq p + 1$, the set 
\begin{align}\label{def-B-ntk-00a}
B_{n, t, k}: = B_{n, t} \cap \{ \tilde \nu_{n,t} = k \}
\end{align}
 is $\mathscr G_k$-measurable. 
For $n \geq 1$, $a \in \bb A_p$ and $t \in \bb R$, 
we have
\begin{align}\label{WK201-a}
 H_n(a, t): = \mathbb{P}_a
\left(   \tilde{\tau}_{t}^{\mathfrak f} > n \right) 
 =  \mathbb{P}_a
\left(  \tilde{\tau}_{t}^{\mathfrak f} > n,  B_{n, t} \right) 
  +  \mathbb{P}_a
\left(   \tilde{\tau}_{t}^{\mathfrak f} > n, 
B_{n, t}^c \right). 
\end{align}
Due to Lemma \ref{Lem-Bnt-001}, the integral of the second term on the right-hand side of \eqref{WK201-a} is negligible. 
Now let us control the first term on the right-hand side of \eqref{WK201-a}, which gives the main contribution. 
On the set $B_{n, t}$, we have $\tilde \nu_{n,t} \leq n^{1-\ee}$, so that, by \eqref{def-B-ntk-00a}, 
\begin{align}\label{Equ-sum-nu-001-a}
 \mathbb{P}_a
\left(   \tilde{\tau}_{t}^{\mathfrak f} > n,  B_{n, t}  \right)  
  = \sum_{k=1}^{[n^{1-\ee}]} 
 \mathbb{P}_a
\left(  \tilde{\tau}_{t}^{\mathfrak f} > n,  B_{n, t, k} \right). 
\end{align}
By the Markov property and the definition of $\tilde\tau^{\mathfrak f}_{t}$ (cf.\  \eqref{def-tau-f-y}), 
we get that, for any $1 \leq k \leq [n^{1-\ee}]$, 
\begin{align}\label{Equ-sum-nu-002-a}
 \mathbb{P}_a
\left(  \tilde{\tau}_{t}^{\mathfrak f} > n,  B_{n, t, k} \right)   
 =  \bb E_a  \bigg( H_{n-k} \bigg( \xi_k, t + \sum_{i=1}^{k} \sigma_p(\xi_i) +  \tilde f(\xi_k) - \tilde f(\xi_0) \bigg);
\tilde{\tau}_{t}^{\mathfrak f} > k,  B_{n, t, k}  \bigg). 
\end{align}
Note that on the set $B_{n, t, k}$, by \eqref{nu n02}, \eqref{def-Bnt-001} and \eqref{def-B-ntk-00a}, 
we have that, for any $1 \leq k \leq [n^{1-\ee}]$, 
\begin{align}
& |\tilde f(\xi_k)| \leq n^{\ee},  \label{Inequ-t-ee-001-a}\\
& |\sigma_p(\xi_{k+i})| \leq n^{\ee}, \forall i \in [1, p],  \label{Inequ-t-ee-002-a}\\
&  t + \sum_{i=1}^{k} \sigma_p(\xi_i) +  \tilde f(\xi_k) - \tilde f(\xi_0) \leq n^{(1-\ee)/2} + n^{\ee} \leq n^{1/2}, \label{Inequ-t-ee-003-a}\\
& t + \sum_{i=1}^{k} \sigma_p(\xi_i) +  \tilde f(\xi_k) - \tilde f(\xi_0) \geq 2 n^{1/2-\ee} - n^{\ee} \geq n^{1/2-\ee}, \label{Inequ-t-ee-004-a}
\end{align}
so that the conditions of Lemma \ref{Lemma Weak Conv-a} are satisfied. 
Therefore, applying Lemma \ref{Lemma Weak Conv-a}, and using again \eqref{Inequ-t-ee-003-a} and \eqref{Inequ-t-ee-004-a}, we get that, uniformly in $0 \leq k\leq [n^{1-\ee}]$, 
\begin{align}\label{Inequa-H-nk-tilde-f}
& \Bigg| 
H_{n-k} \bigg( \xi_k, t + \sum_{i=1}^{k} \sigma_p(\xi_i) +  \tilde f(\xi_k) - \tilde f(\xi_0) \bigg)  
 - \frac{2 \left( t + \sum_{i=1}^{k} \sigma_p(\xi_i) +  \tilde f(\xi_k) - \tilde f(\xi_0) \right)}{  \sqrt{2\pi n} \bf{v} }   \Bigg|  \notag\\
& \leq c n^{-(1+\ee)/2} \bigg(  t + \sum_{i=1}^{k} \sigma_p(\xi_i) +  \tilde f(\xi_k) - \tilde f(\xi_0)  \bigg)
+ c e^{- \alpha n^{\ee} }  \sum_{i = 1}^{n} \mathcal{F}_i(\xi_k). 
\end{align}
Substituting \eqref{Inequa-H-nk-tilde-f} into \eqref{Equ-sum-nu-002-a}, and using \eqref{def-B-ntk-00a} and \eqref{Equ-sum-nu-001-a}, 
we get
\begin{align}\label{Equ-sum-nu-003-a}
& \left| \mathbb{P}_a
\left(  \tilde{\tau}_{t}^{\mathfrak f} > n,  B_{n, t} \right) 
  -  \frac{2}{  \sqrt{2\pi n} \bf{v} }  E_n(a, t)
     \right| \notag\\
& \leq  c  n^{-(1+\ee)/2}  E_n(a, t)   +  c e^{- \alpha n^{\ee} }
  \bb E_a \bigg(  \sum_{i = 1}^{n} \mathcal{F}_i(\xi_{\tilde{\nu}_{n,t}}); \tilde{\tau}_{t}^{\mathfrak f} > \tilde{\nu}_{n,t},  B_{n, t} \bigg), 
\end{align}
where, for $n \geq 1$, $a \in \bb A_p$ and $t \in \bb R$, we have denoted 
\begin{align}\label{def-En-at-00b}
E_n(a, t) =  \bb E_a  \bigg( t + \sum_{i=1}^{ \tilde{\nu}_{n,t} } \sigma_p(\xi_i) +  \tilde f(\xi_{\tilde{\nu}_{n,t}}) - \tilde f(\xi_0);
       \tilde{\tau}_{t}^{\mathfrak f} > \tilde \nu_{n,t},  B_{n, t} \bigg). 
\end{align}
If we define 
\begin{align}\label{def-eta-nt-001}
\eta_{n, t} = \begin{cases}
  [n^{1/2 - 2\ee}]   &   \text{ if }  \  \tilde \nu_{n,t} < n^{1/2 - 2\ee},    \\
  \tilde \nu_{n,t}   &  \text{ if }  \  \tilde \nu_{n,t} \in [n^{1/2 - 2\ee}, n^{1 - \ee}],  \\
     [n^{1 - \ee}]   &  \text{ if }  \  \tilde \nu_{n,t} > n^{1 - \ee},  
\end{cases}
\end{align}
then, by \eqref{def-U-f-n-t-002}, we get 
\begin{align}\label{equation-Enat-001-a}
&  \int_{\bb A_p} \tilde \theta(a) E_n(a, t) \mu_p(da)  
=  U_{\eta_{n, t}}^{\mathfrak f, \theta}(t)  \notag\\
& \quad - \int_{\bb A_p} \tilde \theta(a) \mathbb{E}_{a} 
\bigg( t + \sum_{i=1}^{\eta_{n, t}} \sigma_p(\xi_i) + \tilde f(\xi_{\eta_{n, t}}) - \tilde f(\xi_0); 
 \tilde{\tau}_{t}^{\mathfrak f} > \eta_{n, t}, B_{n,t}^c \bigg)
\mu_p(da). 
\end{align}
In particular, since $t + \sum_{i=1}^{\eta_{n, t}} \sigma_p(\xi_i) + \tilde f(\xi_{\eta_{n, t}}) - \tilde f(\xi_0) \geq 0$
 on the set $\{ \tilde{\tau}_{t}^{\mathfrak f} > \eta_{n, t} \}$, we have
\begin{align}\label{inequa-En-at-upper}
\int_{\bb A_p} \tilde \theta(a) E_n(a, t) \mu_p(da) \leq  U_{\eta_{n, t}}^{\mathfrak f, \theta}(t), 
\end{align}
so that integrating over $a \in \bb A_p$ with respect to $\mu_p$ in \eqref{Equ-sum-nu-003-a} yields
\begin{align*}
 \int_{\bb A_p}  \tilde \theta(a) \mathbb{P}_a
\left(  \tilde{\tau}_{t}^{\mathfrak f} > n,  B_{n, t} \right)  \mu_p(da)  
 \leq  \left( \frac{2}{  \sqrt{2\pi n} \bf{v} }  +  c  n^{-(1+\ee)/2} \right) 
  U_{\eta_{n, t}}^{\mathfrak f, \theta}(t) 
  +  c e^{- \frac{\alpha}{2} n^{\ee} } \|\theta\|_{\infty}  C_{\alpha}(\mathfrak f). 
\end{align*}
Since $\eta_{n, t} < n$, by Proposition \ref{Prop-U-f-t-increase-001}, we have 
\begin{align*}
U_{\eta_{n, t}}^{\mathfrak f, \theta}(t) \leq  U_{n}^{\mathfrak f, \theta}(t) 
+ \frac{c}{n^{\beta}}   \left( 1 + \max \{t,0\} \right)   \|\theta\|_{\infty}
 C_{\alpha}(\mathfrak f). 
\end{align*}
Therefore, by using Lemma \ref{Lem-Bnt-001}, with some constant $\beta > 0$ depending on $\ee$, we obtain 
\begin{align}\label{inte-tau-f-p-001}
  \int_{\bb A_p}  \tilde \theta(a) \mathbb{P}_a
\left(  \tilde{\tau}_{t}^{\mathfrak f} > n \right)  \mu_p(da)  
 \leq   \frac{2}{  \sqrt{2\pi n} \bf{v} }    
  U_{n}^{\mathfrak f, \theta}(t) 
  +  \frac{c (1 + \max \{t,0\})}{n^{1/2 + \beta}}  \|\theta\|_{\infty} C_{\alpha}(\mathfrak f), 
\end{align}
which concludes the proof of \eqref{CCLT-001-a}.

For the proof of \eqref{CCLT-002-a}, we notice that, 
by the moment assumptions \eqref{exp mom for f 001} and \eqref{exp mom for g 002},  
Lemma \ref{Lem-Bnt-001} and \eqref{equation-Enat-001-a}, 
we have 
\begin{align}\label{lower-bound-En-at-mu-da-001}
&  \int_{\bb A_p} \tilde \theta(a)  E_n(a, t) \mu_p(da)   \notag\\
& \geq  U_{\eta_{n, t}}^{\mathfrak f, \theta}(t) -  c n^{1 - \ee} \left( 1 + \max \{t,0\} \right)  \|\theta\|_{\infty}  \sqrt{ C_{\alpha}(\mathfrak f)}
 \int_{\bb A_p} \mathbb{P}_{a}^{1/2} \big(  B_{n,t}^c \big) \mu_p(da)  \notag\\
& \geq  U_{\eta_{n, t}}^{\mathfrak f, \theta}(t) -  c n^{1 - \ee} \left( 1 + \max \{t,0\} \right)  \|\theta\|_{\infty} \sqrt{ C_{\alpha}(\mathfrak f)}
\bigg(  \int_{\bb A_p} \mathbb{P}_{a} \big(  B_{n,t}^c \big) \mu_p(da) \bigg)^{1/2}  \notag\\
& \geq  U_{\eta_{n, t}}^{\mathfrak f, \theta}(t) -  c e^{- \frac{\alpha}{8} n^{\ee}}  \left( 1 + \max \{t,0\} \right)  \|\theta\|_{\infty} C_{\alpha}(\mathfrak f),
\end{align}
where we have also used Lemma 5.8 of \cite{GQX24a}. 
Then we proceed as in the proof of  \eqref{CCLT-001-a}. 
\end{proof}

\begin{proof}[Proof of Corollary \ref{Cor-CCLT-001-a}]
By Lemma \ref{Lemma Weak Conv-a}, 
for any $\ee \in (0, \frac{1}{8})$, 
there exists a constant $c_\ee>0$ such that for any $n \geq 1$, $1 \leq p \leq n^{\ee}$, 
$t \geq n^{1/2 - \ee}$, 
any sequence $\mathfrak f = (f_n)_{n \geq 0}$ of $\scr A_p$-measurable functions, 
and any $a = (g_0, \ldots, g_p, x, q) \in \mathbb{A}_p$ satisfying $|\sigma(g_i, g_{i-1} \cdots g_1 x)| \leq n^{\ee}$ for $1 \leq i \leq p$ and $\tilde f(a) \leq n^{\ee}$,
we have 
\begin{align}\label{upper-bound-tilde-tau-f-001}
\mathbb{P}_{a} \left( \tilde{\tau}_{t}^{\mathfrak f} > n  \right)  \leq
c_{\ee }\frac{t}{\sqrt{n}} + c e^{-\alpha n^{\ee} } \sum_{i = 1}^{n} \mathcal{F}_i(a).  
\end{align}
Indeed, if $t \leq \sqrt{n}$, the inequality \eqref{upper-bound-tilde-tau-f-001} follows directly from Lemma \ref{Lemma Weak Conv-a},
and if $t > \sqrt{n}$, then we use the fact that $\mathbb{P}_{a} ( \tilde{\tau}_{t}^{\mathfrak f} > n ) \leq 1 < \frac{t}{\sqrt{n}}$. 

Now we proceed with the proof of the bound \eqref{Corollary-of-CCLT-001-a}.
Our approach follows the same path as the proof for  \eqref{CCLT-001-a}, 
except that instead of using \eqref{Inequa-H-nk-tilde-f}, 
we apply the following inequality:
on the set $B_{n, t, k}$, uniformly in $0 \leq k\leq [n^{1-\ee}]$, 
\begin{align*}
&  H_{n-k} \bigg( \xi_k, t + \sum_{i=1}^{k} \sigma_p(\xi_i) +  \tilde f(\xi_k) - \tilde f(\xi_0) \bigg)   \notag\\
& \leq c \frac{t + \sum_{i=1}^{k} \sigma_p(\xi_i) +  \tilde f(\xi_k) - \tilde f(\xi_0)}{\sqrt{n}}  
  + c e^{- \alpha n^{\ee} }  \sum_{i = 1}^{n} \mathcal{F}_i(\xi_k). 
\end{align*}
Therefore, instead of \eqref{Equ-sum-nu-003-a}, we have 
\begin{align*}
\mathbb{P}_a
\left(  \tilde{\tau}_{t}^{\mathfrak f} > n,  B_{n, t} \right) 
 \leq  \frac{c}{\sqrt{n}}   E_n(a, t)   +  c e^{- \alpha n^{\ee} }
  \bb E_a \bigg(  \sum_{i = 1}^{n} \mathcal{F}_i(\xi_{\tilde{\nu}_{n,t}}); \tilde{\tau}_{t}^{\mathfrak f} > \tilde{\nu}_{n,t},  B_{n, t} \bigg),  
\end{align*}
so that, instead of \eqref{inte-tau-f-p-001}, we obtain 
\begin{align*}
 \int_{\bb A_p}  \mathbb{P}_a
\left(  \tilde{\tau}_{t}^{\mathfrak f} > n \right)  \mu_p(da) 
 \leq   \frac{c}{  \sqrt{n} }  
  U_{n}^{\mathfrak f}(t) 
  +  \frac{c(1 + \max \{t,0\})}{n^{1/2 + \beta}}  C_{\alpha}(\mathfrak f), 
\end{align*}
where $U_{n}^{\mathfrak f}(t)$ is defined by \eqref{def-U-varphi-001}. 
By Corollary 4.3 of \cite{GQX24a}, there exists a constant $c>0$ such that $U^{\mathfrak f}_n(t) \leq c \left( 1 + \max \{t, 0 \}  \right)$
for any $n \geq 1$ and $t \in \bb R$,  
hence \eqref{Corollary-of-CCLT-001-a} follows. 
\end{proof}

\subsection{Estimates in the conditioned central limit theorem}
As in the previous subsection, we assume that the sequence of perturbations $\mathfrak f = (f_n)_{n \geq 0}$
depends only on finitely many coordinates, 
and we keep the notations introduced in Section \ref{susec: perturb finite-001}.
For such perturbations, Theorem \ref{Thm-CCLT-limit} follows from the following more precise result. 

\begin{proposition}\label{Prop-CCLT-001}
Suppose that the cocycle $\sigma$ admits finite exponential moments \eqref{exp mom for f 001} and is centered \eqref{centering-001}.
Let $\ee \in (0, \frac{1}{8})$ and $b_0 >0$. 
Then, there exist constants $\beta, c >0$ such that for any $n \geq 1$, $1 \leq p \leq n^{\ee}$, $|t| \leq n^{\ee}$, $b \leq b_0$,
any sequence $\mathfrak f = (f_n)_{n \geq 0}$ of $\scr A_p$-measurable functions, 
and any non-negative bounded $\mathscr A_p$-measurable function $\theta$ on $\Omega$, 
\begin{align}\label{CCLT-001}
& \int_{\bb X} \bb E \left( \theta(\omega); \frac{t + \sigma(g_n \cdots g_1, x) }{ \bf{v} \sqrt{n}} \leq b,  
 \tau_{x, t}^{\mathfrak f} > n \right) \nu(dx)  \notag\\
& \leq \bigg( \frac{2}{  \sqrt{2\pi n} \bf{v} }  \int_{0}^{\max\{b, 0\}} \phi^+(u) du  \bigg) 
  U_{n}^{\mathfrak f, \theta}(t) 
  +  \frac{c (1 + \max \{t,0\})}{n^{1/2 + \beta}}  \|\theta\|_{\infty} C_{\alpha}(\mathfrak f)
\end{align}
and 
\begin{align}\label{CCLT-002}
& \int_{\bb X} \bb E \left( \theta(\omega); \frac{t + \sigma(g_n \cdots g_1, x) }{ \bf{v} \sqrt{n}} \leq b,  
 \tau_{x, t}^{\mathfrak f} > n \right) \nu(dx) 
 \notag\\
& \geq \bigg( \frac{2}{  \sqrt{2\pi n} \bf{v} }  \int_{0}^{\max\{b, 0\}} \phi^+(u) du  \bigg) 
  U_{n^{1/2 - 2\ee}}^{\mathfrak f, \theta}(t) 
  -  \frac{c (1 + \max \{t,0\})}{n^{1/2 + \beta}}  \|\theta\|_{\infty}  C_{\alpha}(\mathfrak f). 
\end{align}
\end{proposition}

We will deduce the statement of the proposition using the Markov property, 
along with the following lemma that addresses the case when the starting point $t$ is large.

\begin{lemma}\label{Lemma Weak Conv} 
Suppose that the cocycle $\sigma$ admits finite exponential moments \eqref{exp mom for f 001} and is centered \eqref{centering-001}.
Let $\ee \in (0, \frac{1}{8})$ and $b_0 > 0$. 
Then there exists a constant $c>0$ such that for any $n \geq 1$, $1 \leq p \leq n^{\ee}$, $0 \leq k\leq n^{1-\ee}$, 
$n^{1/2 - \ee} \leq t \leq n^{1/2}$,  $b \leq b_0$, 
any sequence $\mathfrak f = (f_n)_{n \geq 0}$ of $\scr A_p$-measurable functions, 
and any $a = (g_0, \ldots, g_p, x, q) \in \mathbb{A}_p$ satisfying $|\sigma(g_i, g_{i-1} \cdots g_1 x)| \leq n^{\ee}$ for $1 \leq i \leq p$ and $\tilde f(a) \leq n^{\ee}$,
we have 
\begin{align*}
& \left\vert 
\mathbb{P}_{a} \left( \frac{t + \sum_{i=1}^{n-k} \sigma_p(\xi_i) }{ \bf{v} \sqrt{n}} \leq b,  \tilde{\tau}_{t}^{\mathfrak f} >n - k
 \right) - \frac{2t}{  \sqrt{2\pi n} \bf{v} }  \int_{0}^{\max \{b, 0\} } \phi^+(u) du  \right\vert  \notag\\
& \leq  \frac{ct^2}{n} + c e^{-\alpha n^{\ee} } \sum_{i = 1}^{n} \mathcal{F}_i(a).   
\end{align*}
\end{lemma}

\begin{proof}
Let us fix $a = (g_0, \ldots, g_p, x, q) \in \mathbb{A}_p$, $n \geq 1$ and $t \geq n^{1/2-\ee}$.  
For brevity, we denote $t^{+} = t  + \sigma(g_p \cdots g_1, x) - \tilde f(a) + n^{1/2-2\ee },$ 
$t^{-} = t + \sigma(g_p \cdots g_1, x) - \tilde f(a) -n^{1/2-2\ee },$ 
$b^{+}=b + 2 \bf{v} n^{-2\ee}$ and $b^{-} = b -  2\bf{v} n^{-2\ee }.$ 
As in the proof of Lemma \ref{Lemma Weak Conv-a}, we set 
\begin{align*}
A_{n, a}  = \bigg\{ \sup_{0\leq s \leq 1} \bigg\vert  \sum_{i=p +1}^{p+ [(n-p)s]} \sigma_p(\xi_i) 
  - \bf{v} B_{(n-p)s} \bigg\vert \leq  \frac{1}{2} n^{1/2-2\ee } \bigg\} 
  \bigcap \left\{ |\tilde f(\xi_i) | \leq n^{\ee}, \forall i \in [1, n] \right\}.
\end{align*}
For $u \in \bb R$, define $\tau_{u}^{bm} = \inf \{ s \geq 0:  u  + B_s \leq 0 \}$, where $(B_s)_{s \geq 0}$ is the standard Brownian motion
on $\bb R$ with $B_0 = 0$. 
Thus, by \eqref{def-tau-f-y}, it holds that, $\bb P_a$-almost surely, for $n$ large enough and $k \leq n - p$, 
\begin{align*}
\left\{ \tilde{\tau}_{t}^{\mathfrak f} > n - k \right\}  \cap A_{n, a}
\subseteq \left\{ \tau_{ \bf{v}^{-1} t^{+} }^{bm}> n - k - p \right\}
\end{align*}
and 
\begin{align*}
 \bigg\{ \frac{t + \sum_{i=1}^{n-k} \sigma_p(\xi_i) }{\bf{v} \sqrt{n}} \leq b \bigg\}  \cap A_{n, a}  
& \subseteq  \bigg\{ \frac{t + \sigma(g_p \cdots g_1, x) - \tilde f(a) + \bf{v} B_{n-k-p}}{ \bf{v} \sqrt{n}}\leq b 
+ \bf{v} n^{-2\ee} \bigg\}  \notag\\
& \subseteq  \bigg\{ \frac{t^+  + \bf{v} B_{n-k-p}}{ \bf{v} \sqrt{n}}\leq b^+  \bigg\}. 
\end{align*}
Hence, we obtain the following upper bound: 
\begin{align}\label{WK001}
& \mathbb{P}_{a} \bigg( \frac{t + \sum_{i=1}^{n-k} \sigma_p(\xi_i) }{ \bf{v} \sqrt{n}} \leq b,   \tilde{\tau}_{t}^{\mathfrak f} >n - k \bigg) \notag\\
& \leq  \bb P_a \bigg( \frac{t + \sum_{i=1}^{n-k} \sigma_p(\xi_i) }{ \bf{v} \sqrt{n}} \leq b,  
  \tilde{\tau}_{t}^{\mathfrak f} >n - k, A_{n,a}  \bigg)
 + \bb P_a (A_{n,a}^c) \notag\\
 & \leq  \bb P \bigg( \frac{t^+  + \bf{v} B_{n-k-p}}{ \bf{v} \sqrt{n}}\leq b^{+},  
  \tau_{ \bf{v}^{-1} t^{+}}^{bm}> n - k - p  \bigg)
 + \bb P_a (A_{n,a}^c). 
\end{align}
In the same way, it holds that, $\bb P_a$-almost surely, for $k \leq n - p$, 
\begin{align*}
\left\{ \tilde{\tau}_{t}^{\mathfrak f} > n - k \right\} 
\supseteq   \left\{ \tau_{ \bf{v}^{-1} t^{-}}^{bm}> n - k - p \right\} \cap  \{ \tilde{\tau}_t^{\mathfrak f} > p \} \cap A_{n, a}. 
\end{align*}
Besides, notice that on the set $A_{n, a}$, for $1 \leq i \leq p$, we have
\begin{align*}
t + \sigma(g_i \cdots g_1, x) + \tilde f(\xi_i) - \tilde f(a) \geq n^{1/2 - \ee} - (p+2) n^{\ee} > 0
\end{align*}
for $n$ large enough. Thus, $\tilde{\tau}_t^{\mathfrak f} > p$. 
Therefore, we have 
\begin{align*}
\left\{ \tilde{\tau}_{t}^{\mathfrak f} > n - k \right\} 
\supseteq   \left\{ \tau_{ \bf{v}^{-1} t^{-}}^{bm}> n - k - p \right\}  \cap A_{n, a}. 
\end{align*}
Then, we get the following lower bound: 
\begin{align}\label{WK002}
&\mathbb{P}_{a} \bigg( \frac{t + \sum_{i=1}^{n-k} \sigma_p(\xi_i) }{ \bf{v} \sqrt{n}} \leq b,   \tilde{\tau}_{t}^{\mathfrak f} >n - k \bigg)  \notag \\
&\geq  
  \bb P \bigg( \frac{t^-   + \bf{v} B_{n-k-p}}{ \bf{v} \sqrt{n}}\leq b^{-},  
  \tau_{\bf{v}^{-1} t^{-}}^{bm}> n - k - p \bigg)    - \bb P_a (A_{n,a}^c).  
\end{align}
Now we handle the first probability on the right-hand side of \eqref{WK001}. 
By L\'evy \cite{Levy37} (Theorem 42.I, pp. 194-195) and the change of variable $s=u  \sqrt{n}$, 
\begin{align}\label{WK004} 
&\bb P \bigg( \frac{t^+  + \bf{v} B_{n-k-p}}{ \bf{v} \sqrt{n}}\leq b^{+},  
  \tau_{ \bf{v}^{-1} t^{+}}^{bm}> n - k - p \bigg)  \notag \\
& = \frac{1}{\sqrt{2\pi (n-k-p)}  }\int_{0}^{  \sqrt{ n} \max \{b^{+}, 0\} } \bigg( e^{-
\frac{\left( s - \bf{v}^{-1} t^{+}\right) ^{2}}{2  (n - k - p)   } } - e^{-\frac{\left(
s+ \bf{v}^{-1} t^{+}\right) ^{2}}{2 (n - k - p) }}  \bigg)   ds   \notag \\
& = \frac{e^{-\frac{\left( t^{+}/\sqrt{n}\right) ^{2}}{2\bf{v}^{2} (n-k-p)  /n}}}{\sqrt{2\pi (n-k-p)/n} }   
 \int_{0}^{ \max \{b^{+}, 0\} } e^{-\frac{u^{2}}{2 (n-k-p)/n}} 
\bigg( e^{\frac{u t^{+}/\sqrt{n}}{ \bf{v} (n-k-p) /n}}
 -e^{\frac{-u t^{+}/\sqrt{n}}{ \bf{v} (n-k-p)/n}} \bigg) du.  
\end{align}
Note that, as $k\leq n^{1-\ee }$, $p \leq n^{\ee}$ and 
$t \geq n^{1/2-\ee},$ we have, as $n\rightarrow \infty$,  
\begin{align}
(n-k-p) /n =1-O\left( n^{-\ee }\right)  \label{WK005}
\end{align}
and 
\begin{align}\label{WK006}
\frac{t^{+}}{t}=\frac{ t  + \sigma(g_p \cdots g_1, x) - \tilde f(a) + n^{1/2-2\ee } }{ t}
=1+O\left( n^{-\ee}\right) .  
\end{align}
Therefore, by using the inequality $|e^z - 1 - z| \leq z^2 e^{|z|}$ for $z \in \bb R$, we get, for $0 \leq u \leq \max\{ b^+, 0\}$ and 
$n^{1/2 - \ee} \leq t \leq n^{1/2}$, 
\begin{align*}
\left| e^{\frac{u t^{+}/\sqrt{n}}{ \bf{v} (n-k-p) /n}}
 -e^{\frac{-u t^{+}/\sqrt{n}}{ \bf{v} (n-k-p)/n}}
 - 2 \frac{u t^{+}/\sqrt{n}}{ \bf{v} (n-k-p) /n}  \right|
 \leq  c \frac{t^2}{n}, 
 \end{align*}
where $c$ depends on $b_0$. Thus, by \eqref{WK005} and \eqref{WK006}, 
 \begin{align*}
\bigg| e^{\frac{u t^{+}/\sqrt{n}}{ \bf{v} (n-k-p) /n}}
 -e^{\frac{-u t^{+}/\sqrt{n}}{ \bf{v} (n-k-p)/n}}
 - \frac{2 u t}{ \bf{v} \sqrt{n} }   \bigg|
 \leq  c \frac{t^2}{n} + c \frac{t}{ n^{1/2 + \ee} }. 
 \end{align*}
Similarly, using the inequality $|e^z - 1| \leq |z| e^{|z|}$ for $z \in \bb R$, 
we have, uniformly in $k\leq n^{1-\ee },$ $ n^{1/2-\ee  }\leq t \leq n^{1/2}$ and $0\leq u\leq \max\{ b^+, 0\}$, 
\begin{align}
\bigg| e^{-\frac{\left( t^{+}/\sqrt{n}\right) ^{2}}{2\bf{v}^{2} (n-k-p)  /n}} - 1 \bigg|
\leq  c \frac{t^2}{n} 
\quad {\rm and} \quad
e^{-\frac{u^{2}}{2 (n-k-p)/n}}   = e^{-\frac{u^{2}}{2}}  \left( 1+  O(n^{-\ee})   \right). 
\label{WK008}
\end{align}
From \eqref{WK004}, \eqref{WK005} and \eqref{WK008}, it follows that, for $n^{1/2-\ee} \leq t \leq n^{1/2}$, 
\begin{align*}
& \bb P \bigg( \frac{t^+  + \bf{v} B_{n-k-p}}{ \bf{v} \sqrt{n}}\leq b^{+},  
  \tau_{\bf{v}^{-1} t^{+}}^{bm}> n - k - p \bigg)  \notag\\
& \leq \frac{2 t}{\sqrt{2\pi n} \bf{v} }\int_{0}^{ \max\{ b^+, 0\} } u e^{-\frac{u^{2}}{2}}  du
\left( 1+  \frac{c}{n^{\ee}} +  \frac{ct}{\sqrt{n}}  \right)  \notag\\
& \leq \frac{2 t}{\sqrt{2\pi n} \bf{v} }\int_{0}^{ \max\{ b^+, 0\} } \phi^+(u) du
\left( 1 + \frac{ct}{\sqrt{n}}   \right). 
\end{align*}
 Since $b^{+}=b + 2 \bf{v} n^{-2\ee}$, we have $\int_{0}^{ \max\{ b^+, 0\} } \phi^+(u) du 
 = \int_{0}^{ \max\{ b, 0\} }  \phi^+(u)  du + O\left( n^{-2\ee}\right) $.  
As $n^{1/2-\ee  }\leq t \leq n^{1/2}$, we get
\begin{align}
& \bb P \bigg( \frac{t^+  + \bf{v} B_{n-k-p}}{ \bf{v} \sqrt{n}}\leq b^{+},  
  \tau_{\bf{v}^{-1} t^{+}}^{bm}> n - k - p \bigg)  \notag \\
& \leq  \frac{2t}{\sqrt{2\pi n} \bf{v} } \bigg( \int_{0}^{\max\{b, 0\}} \phi^+(u) du  +  \frac{c}{n^{2\ee}}  \bigg)
\bigg( 1 + \frac{ct}{\sqrt{n}}   \bigg)  \notag \\
& \leq \frac{2t}{\sqrt{2\pi n} \bf{v} }\int_{0}^{\max\{b, 0\}} \phi^+(u) du + \frac{ct^2}{n}.  \label{WK009}
\end{align}
Similarly, we obtain
\begin{align}\label{WK010}
 \bb P \bigg( \frac{t^-   + \bf{v} B_{n-k-p}}{ \bf{v} \sqrt{n}}\leq b^{-},  
  \tau_{\bf{v}^{-1} t^{-}}^{bm}> n - k - p \bigg) 
  \geq
  \frac{2t}{\sqrt{2\pi n} \bf{v} }\int_{0}^{b} \phi^+(u) du - \frac{ct^2}{n}. 
\end{align}
By \eqref{WK003-a}, 
we have 
\begin{align}
\bb P_a (A_{n,a}^c)  
\leq  c n^{-2 \ee} + e^{-\alpha n^{\ee} } \sum_{i = 1}^{n} \mathcal{F}_i(a). 
\label{WK003}
\end{align}
Combining \eqref{WK001}, \eqref{WK002}, \eqref{WK009}, \eqref{WK010} and \eqref{WK003} 
concludes the proof of the lemma. 
\end{proof}

\begin{proof}[Proof of Proposition \ref{Prop-CCLT-001}]
Note that, by \eqref{sigma-Gn-001}, \eqref{def-mu-p-da} and \eqref{def-tilde-theta}, it holds that,
for any $n \geq 1$ and $t \in \bb R$, 
\begin{align*}
 \int_{\bb X}  \bb E \left( \theta(\omega);  \frac{t + \sigma(g_n \cdots g_1, x) }{ \bf{v} \sqrt{n}} \leq b,  \tau_{x, t}^{\mathfrak f} > n \right) \nu(dx)  
 = \int_{\bb A_p} \tilde \theta(a)  I_n(a, t)  \mu_p(da), 
\end{align*}
where, for any $n \geq 1$, $a \in \bb A_p$ and $t \in \bb R$, we have denoted 
\begin{align*}
I_n(a, t): = \mathbb{P}_a
\left( \frac{t + \sum_{i=1}^{n} \sigma_p(\xi_i) }{ \bf{v} \sqrt{n}} \leq b,  \tilde{\tau}_{t}^{\mathfrak f} > n \right). 
\end{align*}
Recalling that the set $B_{n,t}$ is defined by \eqref{def-Bnt-001}, we have 
\begin{align*}
I_n(a, t) & =  \mathbb{P}_a
\bigg( \frac{t + \sum_{i=1}^{n} \sigma_p(\xi_i) }{ \bf{v} \sqrt{n}} \leq b,  
 \tilde{\tau}_{t}^{\mathfrak f} > n,  B_{n, t} \bigg)  \notag\\
& \quad  +  \mathbb{P}_a
\bigg( \frac{t + \sum_{i=1}^{n} \sigma_p(\xi_i) }{ \bf{v} \sqrt{n}} \leq b,  \tilde{\tau}_{t}^{\mathfrak f} > n,  B_{n, t}^c \bigg). 
\end{align*}
By Lemma \ref{Lem-Bnt-001}, the second term is negligible. 
Now let us control the first term which gives the main contribution. 
Similarly to \eqref{Equ-sum-nu-001-a}, 
we have
\begin{align}\label{Equ-sum-nu-001}
& \mathbb{P}_a
\left( \frac{t + \sum_{i=1}^{n} \sigma_p(\xi_i) }{ \bf{v} \sqrt{n}} \leq b,  
 \tilde{\tau}_{t}^{\mathfrak f} > n,  B_{n, t}  \right)   \notag\\
&  = \sum_{k=1}^{[n^{1-\ee}]} 
 \mathbb{P}_a
\left( \frac{t + \sum_{i=1}^{n} \sigma_p(\xi_i) }{ \bf{v} \sqrt{n}} \leq b,  
 \tilde{\tau}_{t}^{\mathfrak f} > n,  B_{n, t, k} \right), 
\end{align}
where $B_{n, t, k}: = B_{n, t} \cap \{ \tilde \nu_{n,t} = k \}$ is a $\mathscr G_k$-measurable set, for any $k \geq p + 1$. 
As in \eqref{Equ-sum-nu-002-a}, we apply the Markov property to get that, for any $1 \leq k \leq [n^{1-\ee}]$, 
\begin{align}\label{Equ-sum-nu-002}
& \mathbb{P}_a
\left( \frac{t + \sum_{i=1}^{n} \sigma_p(\xi_i) }{ \bf{v} \sqrt{n}} \leq b,  
 \tilde{\tau}_{t}^{\mathfrak f} > n,  B_{n, t, k} \right)  \notag\\
& =  \bb E_a  \bigg( I_{n-k} \bigg( \xi_k, t + \sum_{i=1}^{k} \sigma_p(\xi_i) +  \tilde f(\xi_k) - \tilde f(\xi_0) \bigg);
\tilde{\tau}_{t}^{\mathfrak f} > k,  B_{n, t, k}  \bigg). 
\end{align}
Note that on the set $B_{n, t, k}$, the inequalities \eqref{Inequ-t-ee-001-a}, \eqref{Inequ-t-ee-002-a}, 
\eqref{Inequ-t-ee-003-a} and \eqref{Inequ-t-ee-004-a} hold true, 
so that the conditions of Lemma \ref{Lemma Weak Conv} are satisfied. 
Consequently, we apply Lemma \ref{Lemma Weak Conv}, and use
 again \eqref{Inequ-t-ee-003-a} and \eqref{Inequ-t-ee-004-a} to get that, uniformly in $0 \leq k\leq [n^{1-\ee}]$, 
\begin{align*}
& \Bigg| 
I_{n-k} \bigg( \xi_k, t + \sum_{i=1}^{k} \sigma_p(\xi_i) +  \tilde f(\xi_k) - \tilde f(\xi_0) \bigg)  \notag\\
& \qquad - \frac{2 \left( t + \sum_{i=1}^{k} \sigma_p(\xi_i) +  \tilde f(\xi_k) - \tilde f(\xi_0) \right)}{  
 \sqrt{2\pi n} \bf{v} }  \int_{0}^{\max\{b, 0\}} \phi^+(u) du \Bigg|  \notag\\
& \leq c n^{-\ee/2} \frac{t + \sum_{i=1}^{k} \sigma_p(\xi_i) +  \tilde f(\xi_k) - \tilde f(\xi_0)}{\sqrt{n}}  + c e^{- \frac{\alpha}{2} n^{\ee} }  \sum_{i = 1}^{n} \mathcal{F}_i(\xi_k). 
\end{align*}
Substituting this into \eqref{Equ-sum-nu-002} and using \eqref{Equ-sum-nu-001}, we get 
\begin{align}\label{Equ-sum-nu-003}
& \bigg| \mathbb{P}_a
\left( \frac{t + \sum_{i=1}^{n} \sigma_p(\xi_i) }{ \bf{v} \sqrt{n}} \leq b,  
 \tilde{\tau}_{t}^{\mathfrak f} > n,  B_{n, t} \right)  
  -  \frac{2}{  \sqrt{2\pi n} \bf{v} }  \int_{0}^{ \max\{b, 0\} } \phi^+(u) du E_n(a, t)
     \bigg| \notag\\
& \leq  c  n^{-(1+\ee)/2}  E_n(a, t)   +  c e^{- \frac{\alpha}{2} n^{\ee} }
  \bb E_a \bigg(  \sum_{i = 1}^{n} \mathcal{F}_i(\xi_{\tilde{\nu}_{n,t}}); \tilde{\tau}_{t}^{\mathfrak f} > \tilde{\nu}_{n,t},  B_{n, t} \bigg), 
\end{align}
where $E_n(a, t)$ is defined by \eqref{def-En-at-00b}. 
Integrating over $a \in \bb A_p$ with respect to $\mu_p$ in \eqref{Equ-sum-nu-003} 
and using \eqref{inequa-En-at-upper}, we obtain 
\begin{align*}
& \int_{\bb A_p}  \tilde \theta(a) \mathbb{P}_a
\bigg( \frac{t + \sum_{i=1}^{n} \sigma_p(\xi_i) }{ \bf{v} \sqrt{n}} \leq b,  
 \tilde{\tau}_{t}^{\mathfrak f} > n,  B_{n, t} \bigg)  \mu_p(da) \notag\\
& \leq  \bigg( \frac{2}{  \sqrt{2\pi n} \bf{v} }  \int_{0}^{\max\{b, 0\}} \phi^+(u) du +  c  n^{-(1+\ee)/2} \bigg) 
  U_{\eta_{n, t}}^{\mathfrak f, \theta}(t) 
  +  c e^{- \frac{\alpha}{2} n^{\ee} } \|\theta\|_{\infty}  C_{\alpha}(\mathfrak f), 
\end{align*}
where $\eta_{n, t}$ is defined by \eqref{def-eta-nt-001}. 
Applying Proposition \ref{Prop-U-f-t-increase-001} and Lemma \ref{Lem-Bnt-001},
we conclude the proof of the upper bound \eqref{CCLT-001}. 

The proof of the lower bound \eqref{CCLT-002} can be performed in a similar way by using \eqref{lower-bound-En-at-mu-da-001}. 
\end{proof}

\subsection{Approximation through finite-size perturbations}
As outlined in  \cite[Section 6.3]{GQX24a}, we will use the approximation property \eqref{approxim rate for gp-002}
to deduce Theorems \ref{Thm-CCLT-limit-000}, \ref{Thm-inte-exit-time-001} and \ref{Thm-CCLT-limit}
in the general case of arbitrary perturbations from the case of perturbations depending only on finitely many coordinates, 
which was established in Proposition \ref{Prop-CCLT-001-a}, Corollary \ref{Cor-CCLT-001-a} and Proposition \ref{Prop-CCLT-001}.  

We begin by showing an effective version of the conditioned central limit theorem (Theorem \ref{Thm-CCLT-limit}),
which will serve as a crucial foundation for our analysis.

\begin{proposition}\label{Prop-CCLT-bound-001}
Suppose that the cocycle $\sigma$ admits finite exponential moments \eqref{exp mom for f 001}
and is centered \eqref{centering-001}. 
We also suppose that the strong approximation property \eqref{KMTbound001}  is satisfied.
Fix $\beta > 0$ and $B \geq 1$.
Fix $b_0 > 0$. 
Then, for any $\gamma > 0$, there exist constants $c, \kappa, \ee, A >0$ with the following property. 
Assume that $\mathfrak f = (f_n)_{n \geq 0}$ is a sequence of  measurable functions on $\Omega \times \bb X$
satisfying the moment condition \eqref{exp mom for g 002} and the approximation property \eqref{approxim rate for gp-002}
with $C_{\alpha}(\mathfrak f) \leq B$ and $D_{\alpha,\beta}(\mathfrak f) \leq B$. 
Assume that $\theta \in L^{\infty}(\Omega, \bb P)$ satisfies the 
approximation property \eqref{approx property of theta-001} 
with $\|\theta\|_{\infty} \leq B$ and $N_{\beta}(\theta) \leq B$. 
Then, we have, for any $n \geq 1$, $|t| \leq n^{\ee}$ and $b \leq b_0$, 
\begin{align}\label{CCLT-003}
& \int_{\bb X}  \bb E \left( \theta(\omega);  \frac{t + \sigma(g_n \cdots g_1, x) }{ \bf{v} \sqrt{n}} \leq b,  \tau_{x, t}^{\mathfrak f} > n \right) \nu(dx)  \notag\\
&  \leq \bigg( \frac{2}{  \sqrt{2\pi n} \bf{v} }  \int_{0}^{\max\{b, 0\}} \phi^+(u) du  \bigg) 
  U_{n}^{\mathfrak f, \theta}(t + A n^{-\gamma}) 
  +  \frac{c (1 + \max \{t,0\})}{n^{1/2 + \kappa}}    
\end{align}
and 
\begin{align}\label{CCLT-004}
& \int_{\bb X}  \bb E \left( \theta(\omega);  \frac{t + \sigma(g_n \cdots g_1, x) }{ \bf{v} \sqrt{n}} \leq b,  \tau_{x, t}^{\mathfrak f} > n \right) \nu(dx) 
 \notag\\
&  \geq \bigg( \frac{2}{  \sqrt{2\pi n} \bf{v} }  \int_{0}^{\max\{b, 0\}} \phi^+(u) du  \bigg) 
  U_{n}^{\mathfrak f, \theta}(t - A n^{-\gamma}) 
  -  \frac{c (1 + \max \{t,0\} )}{n^{1/2 + \kappa}}. 
\end{align}
\end{proposition}

\begin{proof}
Let $\ee > 0$ be as in Proposition \ref{Prop-CCLT-001} and set $p = [n^{\ee}]$.
For $n \geq 1$ and $x \in \bb X$, define 
\begin{align*}
& E_{n,x} = \bigg\{ \omega \in \Omega:  \max_{0 \leq k \leq n} 
 \left| f_k \circ T^k(\omega, x)  - f_{k,p} \circ T^k(\omega, x) \right| \leq \frac{1}{2} n^{-\gamma}, \notag\\
& \qquad\qquad\qquad\qquad\qquad\qquad\qquad\qquad
 |\theta(\omega) - \theta_p(\omega)| \leq  n^{-1/2 - \gamma} \bigg\},
\end{align*}
where, for $\omega=(g_1,g_2,\ldots) \in \Omega$
and $x \in \bb X$, 
$$f_{k,p}(\omega,x) = \bb E( f_{k}(\cdot, x) | \scr A_p)(\omega)$$
 (see (4.8) in \cite{GQX24a} for details). 
Note that, by Chebyshev's inequality and \eqref{approxim rate for gp-002}, we get 
\begin{align*}
\int_{\bb X} \bb P (E_{n,x}^c) \nu(dx) 
& \leq (n+1) \frac{ e^{-\beta p}  }{e^{\frac{\alpha}{2} n^{-\gamma} } - 1}  D_{\alpha, \beta}(\mathfrak f)
 + e^{-\beta p} n^{1/2 + \gamma} N_{\beta}(\theta)  \notag\\
& \leq  c n^{\gamma + 1} e^{-\beta n^{\ee}} ( D_{\alpha, \beta}(\mathfrak f) + N_{\beta}(\theta) ), 
\end{align*}
where $D_{\alpha, \beta}(\mathfrak f)$ and $N_{\beta}(\theta)$ are defined by \eqref{approxim rate for gp-002}
and \eqref{approx property of theta-001}, respectively. 
For $x \in \bb X$ and $t \in \bb R$, on the set $E_{n,x}$, 
by the definition of $\tau_{x,t}^{\mathfrak f}$ (cf.\ \eqref{def-stop time with preturb-001}), 
we have
$\tau_{x, t}^{\mathfrak f} \leq  \tau_{x, t + n^{-\gamma}}^{\mathfrak f_p}$. 
Therefore, 
\begin{align*}
& \bb E \bigg( \theta(\omega);   \frac{t + \sigma(g_n \cdots g_1, x) }{ \bf{v} \sqrt{n}} \leq b,  \tau_{x, t}^{\mathfrak f} > n \bigg) \notag\\
&\leq  \bb E \bigg( \theta_p(\omega);   \frac{t + n^{-\gamma} + \sigma(g_n \cdots g_1, x) }{ \bf{v} \sqrt{n}} \leq b,  \tau_{t + n^{-\gamma}}^{\mathfrak f_p} > n \bigg) 
 + n^{-1/2 - \gamma} + \|\theta\|_{\infty}  \bb P (E_{n,x}^c).  
\end{align*}
As $C_{\alpha}(\mathfrak f_p) \leq C_{\alpha}(\mathfrak f)$ and $\|\theta_p\|_{\infty} \leq \|\theta\|_{\infty}$, 
by \eqref{CCLT-001} of Proposition \ref{Prop-CCLT-001}, it follows that 
\begin{align*}
& \int_{\bb X} \bb E \left( \theta(\omega);  \frac{t + \sigma(g_n \cdots g_1, x) }{ \bf{v} \sqrt{n}} \leq b,  \tau_{x, t}^{\mathfrak f} > n \right) \nu(dx)  \notag\\
& \leq  \bigg( \frac{2}{  \sqrt{2\pi n} \bf{v} }  \int_{0}^{\max\{b, 0\}} \phi^+(u) du  \bigg) 
  U_{n}^{\mathfrak f_p, \theta_p}(t + n^{-\gamma})  \notag\\
& \quad  +  \frac{c (1 + \max \{t,0\})}{n^{1/2 + \beta}}  \|\theta\|_{\infty} C_{\alpha}(\mathfrak f)  
 + n^{-1/2 - \gamma} 
  +  c n^{\gamma + 1} e^{-\beta n^{\ee}}  \|\theta\|_{\infty} ( D_{\alpha, \beta}(\mathfrak f) + N_{\beta}(\theta) ). 
\end{align*}
Therefore, the upper bound \eqref{CCLT-003} follows from 
the approximation property of $U_{n}^{\mathfrak f, \theta}$ by $U_{n}^{\mathfrak f_p, \theta_p}$ 
established in \cite[Proposition 6.8]{GQX24a}. 

In the same way, by \eqref{CCLT-002} of Proposition \ref{Prop-CCLT-001},
we have
\begin{align*}
& \int_{\bb X}  \bb E \left( \theta(\omega); \frac{t + \sigma(g_n \cdots g_1, x) }{ \bf{v} \sqrt{n}} \leq b,  \tau_{x, t}^{\mathfrak f} > n \right) \nu(dx)  \notag\\
& \geq  \bigg( \frac{2}{  \sqrt{2\pi n} \bf{v} }  \int_{0}^{\max\{b, 0\}} \phi^+(u) du  \bigg) 
  U_{n^{1/2 - 2\ee}}^{\mathfrak f_p, \theta_p}(t - n^{-\gamma})  \notag\\
& \quad  -  \frac{c (1 + \max \{t,0\})}{n^{1/2 + \beta}}  \|\theta\|_{\infty}  C_{\alpha}(\mathfrak f)  
 - n^{-1/2 - \gamma}  
  -  c n^{\gamma + 1} e^{-\beta n^{\ee}} \|\theta\|_{\infty} ( D_{\alpha, \beta}(\mathfrak f) + N_{\beta}(\theta) ). 
\end{align*}
Again, the lower bound \eqref{CCLT-004} follows from \cite[Proposition 6.8]{GQX24a}, 
as well as from the bound on the difference $U_{n^{1/2 - 2\ee}}^{\mathfrak f, \theta} - U_{n}^{\mathfrak f, \theta}$
which is established in \cite[Theorem 4.1]{GQX24a}. 
\end{proof}

Likewise, we derive effective estimates for the persistence probability, as presented in 
Theorems \ref{Thm-CCLT-limit-000} and \ref{Thm-inte-exit-time-001}.

\begin{proposition}\label{Prop-CCLT-bound-001-a}
Suppose that the cocycle $\sigma$ admits finite exponential moments \eqref{exp mom for f 001}
and is centered \eqref{centering-001}. 
We also suppose that the strong approximation property \eqref{KMTbound001}  is satisfied.
Fix $\beta > 0$ and $B \geq 1$.
Then, for any $\gamma > 0$, there exist constants $c, \kappa, \ee, A >0$ with the following property. 
Assume that $\mathfrak f = (f_n)_{n \geq 0}$ is a sequence of  measurable functions on $\Omega \times \bb X$
satisfying the moment condition \eqref{exp mom for g 002} and the approximation property \eqref{approxim rate for gp-002}
with $C_{\alpha}(\mathfrak f) \leq B$ and $D_{\alpha,\beta}(\mathfrak f) \leq B$. 
Assume that $\theta \in L^{\infty}(\Omega, \bb P)$ satisfies the 
approximation property \eqref{approx property of theta-001} 
with $\|\theta\|_{\infty} \leq B$ and $N_{\beta}(\theta) \leq B$. 
Then, we have, for any $n \geq 1$, 
\begin{align*}
 \int_{\bb X}  \bb P \left(  \tau_{x, t}^{\mathfrak f} > n \right) \nu(dx) 
 \leq   \frac{c (1 + \max \{t,0\} )}{\sqrt{n} }, 
\end{align*}
and if $|t| \leq n^{\ee}$, 
\begin{align*}
& \frac{2}{  \sqrt{2\pi n} \bf{v} }   
  U_{n}^{\mathfrak f, \theta}(t - A n^{-\gamma}) 
  -  \frac{c (1 + \max \{t,0\} ) }{n^{1/2 + \kappa}}    \notag\\
&  \leq  \int_{\bb X}  \bb E \left( \theta(\omega);    \tau_{x, t}^{\mathfrak f} > n \right) \nu(dx) 
 \leq  \frac{2}{  \sqrt{2\pi n} \bf{v} }    
  U_{n}^{\mathfrak f, \theta}(t + A n^{-\gamma}) 
  +  \frac{c (1 + \max \{t,0\} )}{n^{1/2 + \kappa}}. 
\end{align*}
\end{proposition}

\begin{proof}
The proof of the first statement follows from Corollary \ref{Cor-CCLT-001-a},
employing a method analogous to that used in deriving Proposition 
 \ref{Prop-CCLT-bound-001}  from Proposition \ref{Prop-CCLT-001}. 


Similarly, the proof of the second statement is derived from Proposition \ref{Prop-CCLT-001-a}. 
\end{proof}

\begin{proof}[Proof of Theorems \ref{Thm-CCLT-limit-000}, \ref{Thm-inte-exit-time-001} and \ref{Thm-CCLT-limit}]
Under the assumptions of our theorems, by \cite[Corollary 4.2]{GQX24a}, 
the convergence \eqref{def-U-f-theta-001} holds. 
Therefore, 
these theorems are direct consequences of Propositions \ref{Prop-CCLT-bound-001} and \ref{Prop-CCLT-bound-001-a}. 
\end{proof}



\end{document}